\documentclass[11pt,reqno]{amsart}

\usepackage{mathtools}
\mathtoolsset{showonlyrefs}

\usepackage{color}
\usepackage{amsmath, amsthm, amssymb}
\usepackage{amsfonts}
\usepackage{mathtools}
\usepackage{mathrsfs}

\usepackage{hyperref}

\usepackage[ansinew]{inputenc}
\usepackage{epsfig}
\usepackage{graphicx}
\usepackage[english]{babel}

\theoremstyle{plain}
\newtheorem{theorem}{Theorem}[section]
\newtheorem{corollary}[theorem]{Corollary}
\newtheorem{lemma}[theorem]{Lemma}
\newtheorem{proposition}[theorem]{Proposition}

\theoremstyle{definition}
\newtheorem{definition}[theorem]{Definition}

\theoremstyle{remark}
\newtheorem{remark}[theorem]{Remark}

\numberwithin{equation}{section}

\newcommand{\average}{{\mathchoice {\kern1ex\vcenter{\hrule height.4pt
width 6pt depth0pt} \kern-9.7pt} {\kern1ex\vcenter{\hrule
height.4pt width 4.3pt depth0pt} \kern-7pt} {} {} }}
\newcommand{\ave}{\average\int}
\def\R{\mathbb{R}}
\newcommand{\ep}{\varepsilon}
\newcommand{\HH}{\mathcal H}
\newcommand{\N}{\mathbb N}

\newcommand{\heatop}{ \text{\small \bf  H}}
\newcommand{\anz}{\mathscr P}
\newcommand{\hhh}{\mathfrak h^{\frac 5 2}}

\newcommand{\CC}{\mathscr C}

\newcommand{\cutoff}{\zeta}

\begin{document}

\setcounter{tocdepth}{1}

\title{The singular set in the Stefan problem}

\author{Alessio Figalli}
\address{ETH Z\"urich, Department of Mathematics, R\"amistrasse 101, 8092 Z\"urich, Switzerland}
\email{alessio.figalli@math.ethz.ch}

\author{Xavier Ros-Oton}
\address{ICREA, Pg.\ Llu\'is Companys 23, 08010 Barcelona, Spain \& 
Universitat de Barcelona, Departament de Matem\`atiques i Inform\`atica, Gran Via de les Corts Catalanes 585, 08007 Barcelona, Spain.}
\email{xros@ub.edu}

\author{Joaquim Serra}
\address{ETH Z\"urich, Department of Mathematics, R\"amistrasse 101, 8092 Z\"urich, Switzerland}
\email{joaquim.serra@math.ethz.ch}


\keywords{Stefan problem; free boundary}
\subjclass[2010]{35R35; 35B65.}

\maketitle

\begin{abstract}
In this paper we analyze the singular set in the Stefan problem and prove the following results:
\begin{itemize}
\item The singular set has parabolic Hausdorff dimension at most $n-1$.

\item The solution admits a $C^\infty$-expansion at all singular points, up to a set of parabolic Hausdorff  dimension at most $n-2$.

\item In $\R^3$, the free boundary is smooth for almost every time $t$, and the set of singular times $\mathcal S\subset \R$ has Hausdorff dimension at most $1/2$.

\end{itemize}
These results provide us with a refined understanding of the Stefan problem's singularities and answer some long-standing open questions in the field.
\end{abstract}


\section{Introduction} \label{sec:intro}

The Stefan problem, dating back to the XIXth century \cite{Stefan3,Stefan2}, aims to describe phase transitions, such as ice melting to water, and it is among the most classical and well-known free boundary problems.
In its simplest formulation, this problem consists in finding the evolution of the temperature
$\theta(x,t)$ of the water when a block of ice is submerged inside.
Then, the function $\theta\geq0$ satisfies $\partial_t\theta-\Delta \theta=0$ in the region $\{\theta>0\}$, while the evolution of the free boundary $\partial\{\theta>0\}$ (the interphase ice/water) is dictated by the Stefan condition $\partial_t\theta=|\nabla_x\theta|^2$ on $\partial\{\theta>0\}$.

Let $\chi_A$ denotes the characteristic function of a set $A$.
After the transformation $u(x,t):=\int_0^t\theta(x,\tau)d\tau$, one can note that
$\{u>0\}=\{\theta>0\}$. Also, as explained in \cite{Duv,Baiocchi} (see also \cite{Fig18a}), the Stefan problem becomes locally equivalent to the so-called ``parabolic obstacle problem'':
\begin{equation} \label{Stefan}
\left\{
\begin{split}
\partial_t u-\Delta u &= -\chi_{\{u>0\}}\\
u & \geq 0 \\
\partial_t u & \geq 0\\
\partial_t u&>0\quad {\rm in}\, \{u>0\}
\end{split}
\right.
\qquad\qquad\quad \textrm{in}\quad \Omega\times (0,T)\subset\R^n\times \R.
\end{equation}
The regularity of the free boundary for the Stefan problem \eqref{Stefan} was developed in the groundbreaking paper \cite{C-obst}.
The main result therein establishes that the free boundary (i.e., the interface $\partial\{u>0\}$) is $C^\infty$ in space and time, outside some closed set $\Sigma\subset \Omega\times(0,T)$ of \emph{singular points} at which the contact set $\{u=0\}$ has zero density.


\addtocontents{toc}{\protect\setcounter{tocdepth}{1}}
\subsection{The singular set}

The fine understanding of singularities is a central research topic in several areas related to nonlinear PDEs and Geometric Analysis.
A major question in such context is to establish estimates for the \emph{size} of the singular set --- see \cite{Simons,CKN,W00,Almgren} for some famous examples.

For the Stefan problem \eqref{Stefan}, a variant of the techniques used in the study of the elliptic obstacle problem yields the following result: for every $t$, let $\Sigma_t$ denote the singular points at time $t$ (so that $\Sigma=\cup_{t \in (0,T)}\Sigma_t\times \{t\}$).
Then $\Sigma_t\subset \R^n$ is locally contained in a $(n-1)$-dimensional $C^1$ manifold \cite{LM,Blanchet}.
Furthermore, the whole singular set $\Sigma\subset \R^n\times \R$ is contained in a $(n-1)$-dimensional manifold which is $C^1$ in space and $C^{0,1/2}$ in time; see \cite{LM} for more details.
This result is optimal in space, in the sense that the singular set could be $(n-1)$-dimensional for a fixed time $t=t_0$.
However, it is not clear what the size of the singular set \emph{in time} should be.

The first natural question in this direction is to estimate the parabolic Hausdorff dimension\footnote{
The parabolic Hausdorff dimension is, by definition, the Hausdorff dimension associated with the ``parabolic distance'' $d_{\rm par}\bigl((x_1,t_1),(x_2,t_2)\bigr)=\sqrt{|x_1-x_2|^2+|t_1-t_2|}$.
Notice that, if we denote by ${\rm dim}_{\mathcal H}(E)$ the standard Hausdorff dimension of a set $E\subset \R^{n+1}=\R^n\times \R$, then ${\rm dim}_{\mathcal H}(E)\leq {\rm dim}_{\rm par}(E)$.
On the other hand, the time axis has parabolic Hausdorff dimension 2, while it has standard Hausdorff dimension 1.
} 
of the singular set $\Sigma$, denoted ${\rm dim}_{\rm par}(\Sigma)$.
The results in \cite{LM,Blanchet} imply that ${\rm dim}_{\rm par}(\Sigma)\leq n+\frac12$, and no better estimate was known.
Here, by refining our understanding of singular points, we establish the following:

\begin{theorem}\label{thm-Stefan-intro-0}
Let $\Omega\subset \R^n$ be an open set, let $u \in L^\infty(\Omega\times (0,T))$ solve the Stefan problem \eqref{Stefan}, and let
$\Sigma\subset\Omega\times(0,T)$ be the set of singular points.
Then
\[{\rm dim}_{\rm par}(\Sigma)\leq n-1.\]
\end{theorem}

As mentioned above, the singular set could be $(n-1)$-dimensional in space for a fixed time $t=t_0$. Hence our result gives the optimal bound for the parabolic Hausdorff dimension of $\Sigma$.
Also, since the time axis has parabolic Hausdorff dimension 2, Theorem \ref{thm-Stefan-intro-0} implies that, in $\R^2$, the free boundary is smooth for almost every time $t\in (0,T)$.
Therefore, it is natural to ask ourselves if a similar result holds in the physical space $\R^3$ and, more in general, ``how often'' singular points may appear.

In \cite{Caf78}, Caffarelli showed that any $C^1$ curve contained inside $\Sigma$ cannot evolve with time (i.e., it must be contained in a fixed time slice $\{t=t_0\}$). However, apart from this result, nothing else was known concerning this question.
Here, we prove the following estimate on the size of singular times in the physical space $\R^3$:

\begin{theorem}\label{thm-Stefan-intro-3d}
Let $\Omega\subset \R^3$, let $u \in L^\infty(\Omega\times (0,T))$ solve the Stefan problem \eqref{Stefan}, and let
\begin{equation}\label{setsingtimes}
\mathcal S:=\big\{t\in (0,T)\ : \ \exists\, (x,t)\in \Sigma\big\}
\end{equation} denote the set of ``singular times''.
Then,
\[
\dim_{\mathcal H}(\mathcal S)\leq \frac12.\]
In particular, for almost every time $t\in(0,T)$, the free boundary is a $(n-1)$-dimensional submanifold of $\R^n$ of class $C^{\infty}$.
\end{theorem}

This is the first result on the size of the set of singular times for the Stefan problem. In particular, prior to this result, it was not even known if solutions to the Stefan problem \eqref{Stefan} could have singularities for \emph{all} times $t\in(0,T)$ (not even when $n=2$).

We do not know if the dimensional bound $1/2$ is sharp for the Stefan problem in~$\R^3$, but this bound is definitely critical\footnote{The reader familiar with fluid equations may note that $1/2$ is the same bound currently known for the dimension of singular times in the Navier-Stokes equation; see the classical result of Caffarelli, Kohn, and Nirenberg \cite{CKN}. We do not see any connection between the fact that the numbers are the same. Actually, while for Navier-Stokes many people would like to prove that singular points do not exist, in the Stefan problem singular points exist and can be rather large. Also, given the analysis performed in our paper, the estimate $1/2$ is sharp in many points (see Remark~\ref{1/2}).}, and  we would not be surprised if our bound turned out to be optimal.

Our Theorem \ref{thm-Stefan-intro-3d} above follows from more general result valid in arbitrary dimension $\R^n$.
As explained in more detail below, we prove that the singular set is $C^\infty$, outside a small subset of parabolic dimension at most $n-2$.

\begin{theorem}\label{thm-Stefan-intro}
Let $\Omega\subset \R^n$, and let $u \in L^\infty(\Omega\times (0,T))$ solve the Stefan problem \eqref{Stefan}. Then there exists $\Sigma^\infty\subset \Sigma$ such that 
\[
{\rm dim}_{\rm par}(\Sigma\setminus \Sigma^\infty)\leq n-2,
\qquad
\dim_{\HH} \big( \{ t\in (0,T)\ : \ \exists\, (x,t)\in \Sigma^\infty\}\big) =0, 
\]
and 
$\Sigma^\infty\subset \Omega\times (0,T)$ can be covered by countably many $(n-1)$-dimensional submanifolds in $\R^{n+1}$ of class~$C^\infty$.\footnote{Here, the submanifolds that cover $\Sigma^\infty$ are of class $C^\infty$ as subset of $\R^{n+1}$ with the usual Euclidean distance, not with respect the parabolic distance. So, our statement is much stronger than the previously known results (for instance, \cite{LM} proved $C^1$ regularity of $\Sigma$ with respect to the parabolic distance, which implies only $C^{1/2}$ regularity in time).}
\end{theorem}

In a sense, this result says that the singular set can be split into two separate pieces: one which is very smooth and extremely  rare in time (the set $\Sigma^\infty$), and one which is small (of dimension at most $n-2$).

As a consequence, we also deduce the following corollary in 2 dimensions:

\begin{corollary}\label{thm-Stefan-intro-2d}
Let $\Omega\subset \R^2$, and let $u \in L^\infty(\Omega\times (0,T))$ solve the Stefan problem \eqref{Stefan}. 
Let $\mathcal S$ be as in \eqref{setsingtimes}.
Then
\[\dim_{\mathcal H}(\mathcal S)=0. \]
\end{corollary}

In the next section, we briefly explain the general strategy behind the proofs of our results.

\subsection{Ideas of the proof}

To prove Theorems \ref{thm-Stefan-intro-0}, \ref{thm-Stefan-intro-3d}, and \ref{thm-Stefan-intro} we need to introduce a variety of new ideas with respect to the existing literature, combining  tools from geometric measure theory (GMT), PDE estimates, dimension reduction-type arguments, and new monotonicity formulas.
Let us give a quick overview of the main steps in the proofs.

Let $(x_\circ,t_\circ)$ be a singular point. It is well known that 
\begin{equation}\label{initial-expansion}
u(x_\circ+x,t_\circ+t) = p_{2,x_\circ,t_\circ}(x) + o(|x|^2+|t|),
\end{equation}
where $p_{2,x_\circ,t_\circ}(x)$ is a quadratic polynomial of the form $\frac12x^TAx$, with $A\geq0$ and ${\rm tr}(A)=1$.
In particular, the set of singular points can be descomposed as $\Sigma=\cup_{m=0}^{n-1} \Sigma_m$, where
\[\Sigma_m:= \big\{(x_\circ,t_\circ)\in \Sigma \,:\, {\rm dim}(\{p_{2,x_\circ,t_\circ}=0\})=m\big\},\qquad m=0,\ldots,n-1.\]
Moreover, for each $m$, the set $\Sigma_m\cap \{t=t_\circ\}$ can be covered by a $C^1$ manifold of dimension $m$.
Unfortunately, the previous expansion implies only $C^{1/2}$ regularity in time for the covering manifolds.
In particular, because of this, \eqref{initial-expansion} implies a non-sharp bound on the parabolic Hausdorff dimension of $\Sigma_m$.

To prove Theorem \ref{thm-Stefan-intro-0}, our strategy is to refine \eqref{initial-expansion} in the time variable, by developing a parabolic version of \cite{AlessioJoaquim}.
This is the content of the first part of the paper.
A key tool for this is a truncated version of the parabolic frequency function
\[\phi(r,w) := \frac{r^2 \int_{\{t=-r^2\}} |\nabla w|^2 G\,dx}{\int_{\{t=-r^2\}} w^2 G\,dx},\]
where $G$ is the time-reversed heat kernel.
We will see that (a truncated version of) $\phi(r,w)$ is essentially monotone in $r$ for 
\[w=u(x_\circ+\cdot,t_\circ+\cdot) - p_{2,x_\circ,t_\circ}.\]
Thanks to this fact, assuming with no loss of generality that $(x_\circ,t_\circ)=(0,0)$, we can prove that 
\begin{equation}\label{eq:second blow up}
\frac{(u-p_2)(rx,r^2t)}{\|u-p_2\|_{r}} \longrightarrow q(x,t) \qquad \text{as }r\to 0
\end{equation}
along subsequences, where $q$ is a parabolically homogeneous function.

We then show the following:
\begin{itemize}
\item[(i)]  If $(0,0) \in \Sigma_m$ with $m\leq n-2$, then the function $q$ is a \emph{quadratic} caloric polynomial.
This means that the expansion \eqref{initial-expansion} cannot be improved at these points! Hence, to get an improved dimensional bound on $\Sigma_m$
we employ a barrier argument in the spirit of \cite{FRS}. More precisely, since $m\leq n-2$, we can exploit the fact that $\Sigma_m$ has zero capacity to build a refined barrier and show that 
\[\qquad \qquad \qquad \qquad {\rm dim}_{\rm par}(\Sigma_m)\leq m,\qquad 0\leq m\leq n-2.\]

\item[(ii)] If $(0,0) \in \Sigma_{n-1}$, then $q$ is a homogeneous solution of the parabolic thin obstacle problem.
We denote by $\Sigma_{n-1}^{<3}$ the subset at which the homogeneity is less than $3$.
\begin{itemize}
\item[(a)] If $(0,0) \in \Sigma_{n-1}^{<3}$, we show that $\partial_tq\not\equiv0$ and that $q$ is \emph{convex} in all directions that are tangential to $\{p_2=0\}$.
This allows us to perform a dimension reduction that, combined with a barrier argument, implies 
\[{\rm dim}_{\rm par}(\Sigma_{n-1}^{<3})\leq n-2.\]

\item[(b)] If $(0,0) \in\Sigma_{n-1}\setminus \Sigma_{n-1}^{<3}$ we show that $q$ is always 3-homogeneous, hence
\begin{equation}\label{second-expansion}
u(x_\circ+\cdot,t_\circ+\cdot) = p_{2,x_\circ,t_\circ} + O(|x|^3+|t|^{3/2}),
\end{equation}
and the same barrier argument as in (ii)-(a) implies that 
\[{\rm dim}_{\rm par}(\Sigma_{n-1}\setminus \Sigma_{n-1}^{<3})\leq n-1.\]
\end{itemize}
\end{itemize}
Combining these estimates, Theorem \ref{thm-Stefan-intro-0} follows.

We wish to remark that, although these proofs require a series of delicate new estimates, the strategy behind this first result is a generalization of the ideas in \cite{AlessioJoaquim,FRS} to the parabolic setting.
However, to prove Theorems \ref{thm-Stefan-intro-3d} and \ref{thm-Stefan-intro}, completely new ideas are needed.

Indeed, in the elliptic setting one, can show that, outside a small set of dimension $n-2$, the ``second blow-up'' $q$ is a cubic harmonic polynomial and the frequency function is still (almost) monotone for the function $w=u-p_2-q$.
This fact was crucially used in \cite{FRS} and it allowed us to prove that, in the elliptic setting, an expansion of the form
\[u=p_2+p_3+p_4+o(|x|^{5-\varepsilon})\]
holds, up to a set of dimension $n-2$.

Unfortunately, these methods completely break down in the parabolic setting, as the ``second blow-up'' $q(x,t)$ is {\it never} a cubic caloric polynomial (as we shall see, it is typically of the form $t|x_n|+\frac16|x_n|^3$), and the frequency function $\phi(r,w)$ is \emph{never (almost) monotone} for $w=u-p_2-q$.\footnote{This issue is already present in dimension $1$, and it can be understood as follows: since $u\sim \frac12 x_n^2 +t|x_n|+\frac16|x_n|^3$, in a parabolic cylinder $\CC_r:=B_r\times [-r^2,0]$ the set $\{u=0\}$ has volume $r^{n+3}$ (roughly, it behaves as the set $\CC_r\cap \{|x_n|\lesssim |t|\}$). This error becomes critical when considering a frequency formula of order at least 3, and a completely different strategy needs to be found.}
Thus, if we want to improve \eqref{second-expansion}, we need to develop completely new methods.

One of the main difficulties in the present paper is actually to pass from a cubic expansion \eqref{second-expansion} to an expansion of order $3+\beta$ for some $\beta>0$. More precisely, the goal is to prove that 
\begin{equation}\label{third-expansion}
u(x,t) = {\textstyle\frac12}x_n^2 + a|x_n|(t+{\textstyle\frac16}x_n^2) +p_3^{odd}(x,t)+ O\big((|x|+|t|^{1/2})^{3+\beta}\big),
\end{equation}
for some  $a>0$ and  $p_3^{odd}$ an odd cubic caloric polynomial. One of our key results here (which is probably the most delicate part of the whole paper) is that, up to a change of coordinates,
\eqref{third-expansion} holds up to an $(n-2)$-dimensional set.

This is a new paradigm of dimension reduction, where we do \emph{not} have a frequency or similar powerful monotonicity formula, and it is based on barriers, maximum principle, compactness arguments, parabolic regularity estimates, non-homogeneous blow-ups, and new GMT covering arguments.
This is developed in the second part of the paper, see Sections~\ref{sec-10}-\ref{sec:E3B}, and in particular Propositions~\ref{prop_dicho} and~\ref{prop:heiht4367}.

Once we have proven \eqref{third-expansion}, we need to push the expansion to higher order.
For this, we develop a barrier argument to show that the set $\{u>0\}$ splits into two separate connected components inside the set $\Omega^\beta:=\{|x|^{2+\beta}<-t\}$.

Note that, 
under the parabolic scaling $(x,t) \to (rx,r^2t)$, the set $\Omega^\beta$ converges to  $\R^n\times(-\infty,0)$ as $r\to 0.$ In other words, we have ``broken the parabolic scaling''.
This means that we may expect the solution $u$ to behave, inside $\Omega^\beta$, as two (almost) independent solutions of the Stefan problem inside each connected component of $\{u>0\}$.
In other words, $u$ should behave as the sum of two solutions of the Stefan problem for which $(0,0)$ is a regular point!

The last step, which is carried out in the third part of the paper, is to prove a $C^\infty$ regularity result near \emph{regular} points, which is robust enough to be applied in our setting.
More precisely, we want to show that if $\bar u$ is a solution of the Stefan problem such that $\{\bar u=0\}$ is sufficiently close to $\{x_n\leq0\}$ inside $\Omega^\beta$, then we have a $C^\infty$ expansion for $\bar u$ at $(0,0)$.
We then apply this result to our solution $u$ restricted to each connected component of $\{u>0\}\cap \Omega^\beta$, and we obtain a $C^\infty$-type regularity for $u$.

As a corollary of this $C^\infty$ expansion we are able to prove that, outside a $(n-2)$-dimensional set, if $(x_\circ,t_\circ)$ and $(x_1,t_1)$ are singular points then 
\[|t_\circ-t_1|=o(|x_\circ-x_1|^k)\quad \textrm{for every}\ k\gg1.\]
This implies Theorem \ref{thm-Stefan-intro}, and finally Theorem \ref{thm-Stefan-intro-3d} and Corollary \ref{thm-Stefan-intro-2d} follow as immediate consequences.

\begin{remark}\label{1/2}
In view of our proof strategy, the dimensional bound $1/2$ in Theorem \ref{thm-Stefan-intro-3d} for the set of singular times in $\R^3$ is \emph{critical} in at least two ways.

On the one hand, in the lower strata $\Sigma_m$ with $m\leq n-2$, both \eqref{initial-expansion} and the bound ${\rm dim}_{\rm par}(\Sigma_m)\leq m$ are optimal.
In particular, in $\R^3$, the projection of $\Sigma_{1}$ on the $t$ axis has dimension at most $1/2$, and it looks unlikely to us that one can improve this bound.

On the other hand, the dimension bound $1/2$ is \emph{also} critical when we look at $\Sigma_{2}$.
Indeed, the set of points at which \eqref{third-expansion} does \emph{not} hold has parabolic dimension at most $n-2(=1)$. So, also for this set, its projection on the $t$ axis is expected to have dimension  $1/2$.
Whether this bound can be improved is very unclear to us, and if possible, it would require a completely new understanding of this type of points.
\end{remark}

\subsection{Organization of the paper}

In Section \ref{sec-notation} we introduce some notation that will be used throughout the paper.
Then, in Section \ref{sec-3}, we first recall some classic results on the Stefan problem and then establish some new basic properties of solutions.
In Sections \ref{sec:PFF} and \ref{sec:PMF1} we introduce some new parabolic functionals and monotonicity formulas of Weiss and Almgren-type.
This allows us, in Section \ref{sec:E2BP},  to consider the second blow-up at singular points (see \eqref{eq:second blow up}), and then prove in Section \ref{sec-7} that this blow-up is 3-homogeneous at ``most'' singular points.
In Section \ref{sec-8} we construct some appropriate barriers, adapted to each type of singular point, which allow us to prove Theorem \ref{thm-Stefan-intro-0}.
We then study cubic blow-ups in Section \ref{sec:EG2BP}.
As explained before, a key and difficult part in the proof of our results is to pass from the cubic blow-ups to an estimate of order $3+\beta$, with $\beta>0$ (see \eqref{third-expansion}).
This is accomplished in Sections \ref{sec-10}, \ref{sec:dichotomy}, and \ref{sec:E3B}.
Finally, we show in Section \ref{sec-13} that this implies a $C^\infty$ estimate at ``most'' singular points, and in Section \ref{sec:ESC} we finally give the proof of Theorem \ref{thm-Stefan-intro} and its consequences.

\subsection{Acknowledgments}

AF and JS have received funding from the European Research Council (ERC) under the Grant Agreement No 721675.
XR has been supported by the European Research Council (ERC) under the Grant Agreement No 801867, and by the Swiss NSF.
JS has been supported by Swiss NSF Ambizione Grant PZ00P2 180042 and by the European Research Council (ERC) under the Grant Agreement No 948029.

\section{Notation used throughout  the paper}  \label{sec-notation}

In this section we introduce the notation and the mathematical objects that will be used consistently throughout the paper. In particular, the letters $G$ and $\cutoff$   always refer in the sequel to the Gaussian and the spatial cut-off introduced below, and --- since they appear often in the paper --- we will not recall this every time.

\subsection{Operators}\label{sec:operators} We define the following useful operators that will be used in all the paper.
\[
\begin{array}{lll}
\heatop   := (\Delta - \partial _t)  &  \qquad& (\mbox{that is, } \heatop f   :=  \Delta f - \partial_t f) 
\vspace{5pt}
\\
Z  := (x\cdot \nabla  +2t \partial_t)  & \qquad& (\mbox{that is, } Z f   :=  x\cdot \nabla f + 2t \partial_t f). 
\end{array}
\]

\subsection{Gaussian kernel}  We introduce the ``reversed heat kernel'' in $\R^n\times(-\infty,0)$:
\[
G(x,t) := \frac{1}{(4\pi t)^{n/2}} \exp\left( \frac{|x|^2}{4t}\right).
\]

\subsection{A bilinear form} Given $r\in (0,1]$, we denote by $\langle \,,\,\rangle_r$ the following bilinear form defined for pairs of functions $f, g: \R^n\times(-1,0)\to \R^k$:
\[
\langle f,g\rangle_{r}  := \int_{\R^n}  (f\cdot g)(x, -r^2)  \,G(x,-r^2) \,dx .
\]
In the special case $r=1$, we sometimes use the notation $\langle \,,\,\rangle :=  \langle \,,\,\rangle_1$

\subsection{Functionals $D$ and $H$}  Throughout the paper we shall use the following  dimensionless quantities defined for all functions $w:\R^n\times(-1,0)\to \R$ sufficiently regular (in our application, $w$ will be at least $C^{1,1}$ in space and $C^{0,1}$ in time):
\[
D(r,w) :=   2r^2 \langle\nabla w, \nabla w \rangle_r = 2r^2 \int_{\{t=-r^2\}}  |\nabla w|^2G,
\qquad
H(r,w) :=  \langle w, w \rangle_r = \int_{\{t=-r^2\}}  w^2G.
\]

\subsection{Parabolic rescaling} We introduce here a useful notation for parabolic rescaling and normalization. 
Given a function $w:\R^n\times(-1,0)\to \R$, for $r>0$ we define $w_r$ and $\tilde w_r$ as
\begin{equation}\label{defwrP}
w_r(x,t) : = w(rx,r^2 t) \qquad \mbox{and}\qquad \tilde w_r(x) : = \frac{w_r}{H(1,w_r)^{\frac 1 2}}  = \frac{w(r\,\cdot\,, r^2\,\cdot\, )}{H(r,w)^{\frac 1 2}}  .
\end{equation}
Note that $H(1,\tilde w_r)=1$.
Also, the following commutation properties of the parabolic rescaling with $\heatop$ and $Z$ hold:
\[
\heatop  (f_r) = r^2 (\heatop f)_r\quad \mbox{and}\quad Z  (f_r) =   (Zf)_r.
\]
Furthermore,
\[
\langle f,g \rangle_r  = \langle f_r,g_r \rangle, \qquad \mbox{where } \langle \,,\,\rangle =  \langle \,,\,\rangle_1 . 
\]

\subsection{Homogeneous functions} We say that a function $f(x,t)$ is (parabolically) $\lambda$-homogeneous whenever $f(r x,r^2 t)=r^\lambda f(x,t)$ for all $r>0$. Equivalently, $f$ is $\lambda$-homogeneous whenever $Zf =\lambda f$.

\subsection{Spatial cut-off}\label{sect:cutoff} Throughout the paper, we will save the letter $\cutoff$ to denote a spatial cut-off as follows:
\begin{equation}\label{eta}
\mbox{ $\cutoff = \cutoff(x)  \quad$ with  $\cutoff\in C^\infty_c(B_{1/2})$, \ $\cutoff \ge0$, \ and \ $\cutoff\equiv 1$ in $B_{1/4}$}.
\end{equation}
This cut-off function will be crucial, since the functionals $D$ and $H$ that appear in monotonicity formulae need functions defined in the whole  space. Hence, we shall multiply our solution and its derivatives (which are defined only in a bounded domain) by $\cutoff$ in order to define functions defined in the whole space (of course, by doing so we will introduce errors in the equation but these will be exponentially small).

\subsection{Parabolic cylinders} Given $r>0$, we define the parabolic cylinder $\CC_r$ as
\[
\CC_r: = B_r\times (-r^2,0).
\] 

\subsection{Spatial projection and projection  onto time axis}
We use the notation $\pi_x :\R^n \times \R  \to \R^n$ and $\pi_t :\R^n \times\R   \to \R$ to refer to the canonical projections:
\begin{equation}
\label{eq:proj}
\pi_x(x,t) : =  x\qquad \mbox{and} \qquad  \pi_t(x,t) : =  t.
\end{equation}

\section{Classical and new facts about solutions to the Stefan problem}
\label{sec-3}

It is well known, by the results of Caffarelli  \cite{C-obst}, that solutions to the Stefan problem are $C^{1,1}$ in space and $C^{0,1}$ in time.
So, we consider here $u\in C^{1,1}_{x}\cap C^{0,1}_{t}(B_1\times(-1,1))$ a solution of the parabolic obstacle problem
\begin{equation}\label{eq:UPAR1}
\begin{cases}
\heatop u = \chi_{\{u>0\}} \\
u \ge 0\\
\partial_t u >0   \quad   \mbox{in } \{u>0\}
\end{cases}
\qquad \textrm{in}\quad B_{1}\times(-1,1).
\end{equation}
Note that, since $u$ is nonnegative, \eqref{eq:UPAR1} implies that $u$ is nondecreasing in time inside $B_1\times(-1,1)$.

By Caffarelli \cite{C-obst},  any solution $u$ of \eqref{eq:UPAR1} with $(0,0)\in \partial \{u=0\}$ satisfies
\begin{equation}\label{optimalreg+nondegP}
\sup_{B_{1/2}\times (-1/2,0)} | D^2u| + |\partial_t u| \le C \|u(\,\cdot\,,0)\|_ {L^\infty(B_1)}  \qquad \mbox{and}\qquad \sup_{B_r(0)} u(\,\cdot\,, 0)\ge cr^2 \quad \forall \,r\in (0,1),
\end{equation}
where $C,c$ are positive dimensional constants. 

Notice that, if $u$ is a solution of \eqref{eq:UPAR1}, then
\[u^{x_\circ, t_\circ, r}  : = r^{-2}u(x_\circ +r\,\cdot\,, t_\circ +r^2\,\cdot\,)\]
is also a solution (in the corresponding rescaled and translated domain).

The classical theory defines the following two type of points on the {\em free boundary} $\partial \{u>0\}$:
\begin{itemize}

\item  {\em Regular points}:  $(x_\circ,t_\circ)\in \partial \{u>0\}$ is a {\em regular point} if
\[\lim_{r\to 0} u^{x_\circ, t_\circ, r} (t,x)= \frac 1 2 \big( \max\{0, \boldsymbol e\cdot x\}\big)^2\] for some $\boldsymbol e\in \mathbb S^{n-1}$.

\vspace{2mm}

\item {\em Singular points}:  $x_\circ\in \partial \{u>0\}$ is a {\em singular point} if
\[
 \lim_{r\to 0} u^{x_\circ, t_\circ, r}(t,x)\to p_{2,x_\circ, t_\circ}(x),
\]
where  $p_{2,x_\circ, t_\circ}$ is a convex quadratic $2$-homogeneous polynomial which does not depend on time.
\end{itemize}
When $(x_\circ, t_\circ)= (0,0)$ is a singular point, we simplify the notation by writing $p_2$ instead of $p_{2,0,0}$.
Throughout the paper we will denote
\[\Sigma := \{ \mbox{singular points in } B_{1}\times(-1,1) \}\subset \partial \{u>0\}. \]

The following result follows from \cite{C-obst}, from \cite{KN77} for the higher regularity near regular points, and from \cite{Blanchet} for the uniqueness of blow-up at singular points:

\begin{theorem}
Let $u\in C^{1,1}_{x}\cap C^{0,1}_{t}(B_1\times(-1,1))$ solve \eqref{eq:UPAR1}.
Then, every free boundary point is either regular or singular. 
Moreover, the set of regular points is relatively open inside the free boundary, and it is a $C^\infty$ $(n-1)$-dimensional manifold inside $B_1\times(-1,1)$.
\end{theorem}

As a consequence of the parabolic Monneau-type monotonicity formula proved in \cite{Blanchet,LM}, we have the following result for the singular set.

\begin{proposition}\label{LinMon}
Let $u\in C^{1,1}_{x}\cap C^{0,1}_{t}(B_1\times(-1,1))$ solve \eqref{eq:UPAR1}, and let $\Sigma\subset B_1\times(-1,1)$ denote the set of singular points.
Then, $\pi_x(\Sigma) \subset B_1$ can be locally covered by a $(n-1)$-dimensional $C^1$ manifold. 
\end{proposition}

\begin{proof}
This follows from the more general result of $C^1$ regularity in space and $C^{0,1/2}$ in time given in \cite[Theorem 1.9]{LM}.
\end{proof}


We will also need the following estimate, whose proof is contained in \cite{C-obst,CF}.

\begin{proposition} \label{logarithmic}
Let $u\in C^{1,1}_{x}\cap C^{0,1}_{t}(B_1\times(-1,1))$ solve \eqref{eq:UPAR1}, and let $(0,0)$ be a free boundary point.
Then
\[
|\partial_t u| +(D^2 u)_-  \le C |\log r|^{-\ep} \quad \mbox{ in }  B_r\times(-r^2, r^2),
\]
where $C,\ep>0$ depend only on $n$.
\end{proposition}

After this summary of some well-known properties, we now prove some new results that will be crucial in our analysis.
We begin with the following local semiconvexity property in time.

\begin{proposition} \label{prop.semtime}
Let $u\in C^{1,1}_{x}\cap C^{0,1}_{t}(B_1\times(-1,1))$ solve \eqref{eq:UPAR1}. Then, 
\[
\partial_{tt} u  \ge  -C \quad \mbox{ in }  B_{1/2}\times(-{\textstyle \frac12}, {\textstyle \frac12}),
\]
where $C$ depends only on $n$ and $\|u\|_{L^\infty}$.
\end{proposition}

To prove the proposition above, we need the parabolic version of the standard  $L^1$ to $L^1_{weak}$ estimate of Calder\'on and Zygmund for elliptic PDE \cite{Jones}: 

\begin{theorem}\label{CZpar}
Let $w: \CC_1\to \R$ be any solution of $\heatop w = f$ in $\CC_1$, with $f\in L^1$. Then
\[
 \sup_{\theta>0} \theta\big|  \{ |D^2 w| + |\partial_t w| > \theta \} \cap \CC_{1/2} \big| \le  C(\| f\|_{L^1 (\CC_1)} +  \|w\|_{L^1(\CC_1)}),
\]
for some constant $C$ depending only on $n$.
\end{theorem}

We will also need the following result (see e.g.  \cite[Theorem 4.16]{Wang92}):

\begin{lemma}\ \label{buauav}
Let $w: \CC_1\to R$ satisfy $\heatop w \le 0$. Then
\[
 \sup_{\CC_{1/2}} w \le  C\bigg(\int_{\CC_1} (w_+)^\ep\bigg)^{\frac 1 \ep},
\]
for some $\ep>0$ and $C$ depending only on $n$.
\end{lemma}
We can now prove the desired semiconvexity in time:

\begin{proof}[Proof of Proposition \ref{prop.semtime}]
Fix $r_1, r_2$ such that $1/2< r_1 < r_2 <1$, and for $r \in (0,1)$ define the domains $\Omega_r := B_r \times (-r,r)$.
For $h\in \R$ and given a  function $w$, we define the discrete time derivative as
$\delta_{t,h}  w(x,t): = h^{-1} \big( w(x, t+h) - w(x) \big)$.

Now, consider $v^h : = \delta_{t,h}\partial_t  u$. By   Proposition \ref{logarithmic} the function $v^h$  is continuous.
Also, by \eqref{eq:UPAR1},  $v^h\ge 0$ on $\{u=0\}$.
Notice in addition that   $\heatop v^h \le 0$ in $\{u>0\}$. Indeed,  inside the open set $\{u>0\}$ we have $\heatop (\partial_t \delta_{t,h} u) =0$ and hence we have (in the sense of distributions)
 \[
  \heatop (\partial_t \delta_{t,h} u)(\,\cdot\,, \,\cdot\,)  =  \frac{\partial_t (\chi_{\{u>0\}}) (\,\cdot\,,  \,\cdot \,+ h) }{h}\ge 0,
 \]
where the right hand side is a nonnegative measure since $\partial_t (\chi_{\{u>0\}})$ is a nonnegative measure by the monotonicity property of $u$.

This proves that $(v^h)_-$ is subcaloric, and using Lemma \ref{buauav} we get
\[
{\rm ess\, sup}_{\Omega_{1/2}} (v^h)_- \ \le\   \sup_{\theta>0}  \theta \big|  \{   (v^h)_- >\theta \} \cap \Omega_{r_1} \big|   
=  C \sup_{\theta>0}   \theta |\{  |\partial_{t} \hat w| > \theta \} \cap \Omega_{r_1}\big|,  \quad \mbox{ where } \hat w : = \delta_{t,h} u.
\]
Note that the function $\hat w$ satisfies  $\heatop \hat w  = \delta_{t,h} (\chi_{\{u>0\}}) \ge0$  and $\hat w\ge0$, since $u$ and  $\{u>0\}$ are nondecreasing in time.
Moreover, by the parabolic Calder\'on-Zygmund type estimate (Theorem~\ref{CZpar}), we have
\[
\sup_{\theta>0}\theta |\{  |\partial_{t} \hat w| > \theta \} \cap \Omega_{r_1} \big|  \le  C (\|\heatop \hat w\|_{L^1(\Omega_{r_2})} +  \|\hat w\|_{L^1(\Omega_{r_2})}).
\]
Now, fix $\xi= \xi(x,t)$ a smooth nonnegative  cut-off,  compactly supported in $\Omega_{1}$ and such that $\xi\equiv 1$ in $\Omega_{r_2}$. 
For $h\ll (1-r_2)$, we have
\[
\|\heatop \hat w\|_{L^1(\Omega_{r_2})} \le \int_{\Omega_1} \heatop \hat w\, \xi =  \int_{\Omega_1} (\heatop  \delta_{t,h}u) \,\xi = \int_{\Omega_1} u \, \delta_{t,-h}(\Delta \xi+\partial_t\xi) \le C\|u\|_{L^1(\Omega_1)},
\]
and similarly
\[
\|\hat w\|_{L^1(\Omega_{r_2})} \le \int_{\Omega_1} \hat w\, \xi =  \int_{\Omega_1} \delta_{t,h} u\, \xi = \int_{\Omega_1} u  \,\delta_{t,-h} \xi \le C\|u\|_{L^1(\Omega_1)}.
\]
Hence, combining all these estimates together, we have shown that
\[
{\rm ess\, sup}_{\Omega_{1/2}} (v^h)_- \le C\|u\|_{L^1(\Omega_1)}.
\]
Letting $h\downarrow 0$, we conclude the proof.
\end{proof}

We will also need the following relation between the derivatives of $u$.
 
\begin{proposition}\label{proputD2u-}
Let $u\in C^{1,1}_{x}\cap C^{0,1}_{t}(B_1\times(-1,1))$ solve \eqref{eq:UPAR1}. Then there exists a constant $C>0$ such that
\[
(D^2  u)_-  \le C\,\partial_t u \quad \mbox{ in }  B_{1/2}\times(-1/2, 1/2).
\]
\end{proposition}

\begin{proof}
Given $(x_\circ,t_\circ)\in B_{1/2}\times(-1/2, 1/2)$, and fixed $e\in \mathbb S^{n-1}$, we consider
\[
w_{x_\circ,t_\circ} := \partial_t u +c_1 \left(\textstyle \frac{|x-x_\circ|^2}{4n} -\frac{(t-t_\circ)}{2} -u\right) -c_2 (\partial_{ee} u)_-,
\]
where $c_1$ and $c_2$ are small positive constants to be chosen.
Note that $w_{x_\circ,t_\circ}$ is nonnegative on $\partial \{u>0\}$. 

Let $\Omega :=B_{1/4}(x_\circ)\times(t_\circ-1/4,t_\circ)$ and denote by $\partial_{par}\Omega$ its parabolic boundary.
Since $\partial_t u>0$ in $\{u>0\}$, it follows by the first bound in \eqref{optimalreg+nondegP} that
\[
w_{x_\circ,t_\circ} \ge  c_1\big({\textstyle \frac{1}{64 n}} - C_0\varrho^2\big) - c_2 C_0  \quad \mbox{in }  \partial_{par} \Omega \cap \big( \{u=0\}  + (B_{\rho}\times(0,\rho^2)) \big)
\]
where $C_0$ depends only on $n$ and $\|u\|_{L^\infty}$, and  $\{u=0\}  + (B_{\rho}\times(0,\rho^2))$ denotes the Minkowski sum of sets.
Hence, choosing $\varrho>0$ sufficiently small and  $c_1\ge CC_0 c_2$, we will have $w_{x_\circ,t_\circ} >0$ on this piece of boundary.

On the other hand, since by assumption $\partial_t u>0$ in $\{u>0\}$, we have 
\[
w_{x_\circ,t_\circ} \ge \partial_t u - C(c_1+c_2) \ge c_u- C(c_1+c_2)  \quad \mbox{in  } \Omega \setminus \big( \{u=0\}  + (B_{\rho}\times(0,\rho^2)) \big),
\]
where $c_u>0$ is a constant (depending on $u$).

Hence, choosing $c_1$ and $c_2$ sufficiently small  so that $c_u- C(c_1+c_2)>0$ and  $c_1\ge CC_0 c_2$, it follows by the maximum principle that 
\[
w_{x_\circ,t_\circ} \ge 0 \quad \mbox{in }\Omega.
\]
In particular, evaluating at $(x_\circ, t_\circ)$ we obtain
\[
\partial_t u(x_\circ, t_\circ) -c_1 u(x_\circ, t_\circ) -c_2 (\partial_{ee} u)_-(x_\circ, t_\circ) >0.
\]
Since $u(x_\circ, t_\circ)\ge 0$, and both $(x_\circ, t_\circ)\in B_{1/2}\times(-{\textstyle \frac12}, {\textstyle \frac12})$ and $e\in\mathbb S^{n-1}$ can be chosen arbitrarily, the lemma follows.
\end{proof}

Finally, we shall also need the following result.

\begin{lemma}\label{lemconvxn}
Let $u\in C^{1,1}_{x}\cap C^{0,1}_{t}(B_1\times(-1,1))$ solve \eqref{eq:UPAR1}. There exists a dimensional constant $\ep_\circ>0$ such that the following holds: Assume that
\[
\big|  u -{\textstyle \frac{1}{2}} (x\cdot e)^2\big| \le \ep_\circ\qquad \text{and}\qquad \partial_{ee}u \geq -\ep_\circ
\]
for some $e\in\mathbb  S^{n-1}$.  Then
\[
\partial_{ee} u \ge 0\quad  \mbox{in }B_{1/2}\times(-{\textstyle \frac12}, {\textstyle \frac12}).
\]
\end{lemma}

\begin{proof}
We may assume $e=\boldsymbol e_n$. Write $x=(x',x_n) \in \R^{n-1}\times \R$, and 
set $v^h(x',x_n,t) :=  h^{-2}\big( u(x',x_n+h , t)+  u(x',x_n-h , t) -2u(x',x_n,t)\big)$.
We note that 
\[
\heatop v^h  \le 0 \quad \mbox{in }\{u>0\}  \qquad \mbox{ and } \qquad v^h\ge 0 \quad \mbox{in }\{u=0\}.
\]
Now, fix $(x_\circ, t_\circ)\in B_{1/2}\times(-1/2, 1/2)$ and consider the function
\[
w^{h}_{x_\circ,t_\circ} :=  v^h  +  256\,n\, \ep_\circ  \left( \textstyle \frac{|x-x_\circ|^2}{4n} -\frac{(t-t_\circ)}{2}  -u\right) \quad \text{inside } U_{x_\circ,t_\circ}: = \{|x-x_\circ|\le 1/4, -1/4 \le  t-t_\circ\le 0\}.
\]
Note that $w_{x_\circ,t_\circ} $ is supercaloric in $U_{x_\circ,t_\circ}\cap \{u>0\}$. 
Therefore,  if we show that ---for $\ep_\circ$ small---  $w_{x_\circ,t_\circ} $ is nonnegative on $\partial_{par}U_{x_\circ,t_\circ}\cup(U_{x_\circ,t_\circ}\cap \{u=0\})$, then it will follow by the maximum principle that $w^h_{x_\circ,t_\circ} \geq 0$ inside $U_{x_\circ,t_\circ}$, and in particular
\[
 v^{h} (x_\circ,t_\circ)\geq w^{h}_{x_\circ,t_\circ}(x_\circ,t_\circ)\ge 0.
\]
Since $(x_\circ, t_\circ)\in B_{1/2}\times(-1/2, 1/2)$ is arbitrary, and  $v^h \to \partial_{nn} u$ as $h\downarrow 0$, the lemma will follow.
So, we are left with proving that $w_{x_\circ,t_\circ}\geq 0$ on $\partial_{par}U_{x_\circ,t_\circ}\cup(U_{x_\circ,t_\circ}\cap \{u=0\})$

Since $v^h\ge 0$ on $\{u=0\}$, it follows that $w^{h}_{x_\circ,t_\circ}\ge 0$  on  $U_{x_\circ,t_\circ}\cap \{u=0\})$.
Hence, it only remains to show that $w^{h}_{x_\circ,t_\circ}\ge 0$   on $\partial_{par}U_{x_\circ,t_\circ}$ ---provided $\ep_\circ$ is chosen sufficiently small. 
Note that, on $\partial_{par}U_{x_\circ,t_\circ}$, we have $\frac{|x-x_\circ|^2}{4n} -\frac{(t-t_\circ)}{2} > \frac{1}{64n}$.
Let us divide the parabolic boundary in two pieces: the piece (I) where $\{u\le \frac{1}{128n} \}$, and  (II) where $\{u\ge \frac{1}{128n} \}$.
On  the piece (I), by the assumption $\partial_{nn}u\geq -\ep_\circ$  we obtain 
\[
w^{h}_{x_\circ,t_\circ} =  v^h  +  256\,n\, \ep_\circ  \left(  \textstyle \frac{|x'-x_\circ'|^2}{4(n-1)} -\frac{(t-t_\circ)}{2} -u\right) \ge -2\ep_\circ + 256\,n\, \ep_\circ  \left(\frac{1}{64n}-\frac{1}{128n}\right ) \ge 0.
\]
To estimate $w^{h}_{x_\circ,t_\circ} $ on the piece (II) we note that, since $\big|  u -{\textstyle \frac{1}{2}} x_n^2\big| \le \ep_\circ$, it follows that
$$
\{ u> \frac{1}{128n}\} \subset 
\left\{|x_n|>\frac{1}{10n^{1/2}}\right\},\qquad\{ u> 0\} \supset 
\left\{|x_n|>\frac{1}{20n^{1/2}}\right\},\qquad \|u\|_{L^\infty(U_{x_\circ,t_\circ})}\leq 1,
$$
provided $\ep_\circ$ is taken sufficiently small.
Thus, using standard interior regularity for the heat equation and interpolation\footnote{Note that $u-\frac 1 2 x_n^2$ is caloric inside $\left\{|x_n|>\frac{1}{20n^{1/2}}\right\}$ and it has small $L^\infty$ norm.}   we obtain that $|\partial_{nn} (u-\frac 1 2 x_n^2)| \le C_n\ep_\circ$ inside $\left\{|x_n|>\frac{1}{10n^{1/2}}\right\}$. 
In particular $v^h \ge \frac 1 2$ in $\left\{|x_n|>\frac{1}{10n^{1/2}}\right\}$ for $\ep_\circ$ sufficiently small.
Therefore, recalling that $|u|\leq 1$ inside $U_{x_\circ,t_\circ}$,
it follows that
\[
w^{h}_{x_\circ,t_\circ} =  v^h  +  256\,n\, \ep_\circ  \left( \textstyle \frac{|x'-x_\circ'|^2}{4(n-1)} -\frac{(t-t_\circ)}{2}  -u\right) \ge \frac 1 2 -C_n\ep_\circ \ge 0\qquad \text{on (II)},
\]
provided we choose $\ep_\circ$ small.
\end{proof}

\section{Parabolic functionals and useful formulae} \label{sec:PFF}

In this section we establish formulae which are valid  for arbitrary functions $\R^n\times (-1,1) \rightarrow \R$ having $C^{1,1}_{x}\cap C^{0,1}_{t}$ regularity.
Later on we will apply these formulae for example to $w(x,t) =(u-p_2)(x,t)\cutoff(x)$.
Recall that $Z$, $G$, $H$, $D$, and $\zeta$, were defined in Section \ref{sec-notation}.

\begin{lemma} \label{lem:PFF0} Let $f \in C^{1,1}_{x}\cap C^{0,1}_{t}$. Then
\[
\frac{d}{dr}\bigg|_{r=1} \langle f,G\rangle_r=\langle Zf,G\rangle.
\]
\end{lemma}

\begin{proof}
Recall that
\[
 \langle f,G\rangle_r= \int_{\{t= -r^2 \}} fG,\qquad  \langle Zf,G\rangle= \int_{\{t= -1 \}} Zf \,G,\qquad Z = x\cdot \nabla + 2t\cdot \partial_t.
\]
Also, thanks to the scaling property $G(rx,r^2t) r^n=G(x,t)$, we have
\[
\int_{\{t= -r^2 \}} f(x,t)G(x,t) dx 
=   \int_{\{t= -1 \}} f(rx,rt)  G(rx,r^2t) r^n \,dx
=   \int_{\{t= -1\}} f(rx,r^2t)  G(x,t)  \,dx.
\]
Then, the result follows by differentiating with respect to $r$ the expression above, and evaluating at $r=1$.
\end{proof}

\begin{lemma} \label{lem:PFF000} 
Let $f,g \in C^{1,1}_{x}\cap C^{0,1}_{t}$. 
Then
\begin{equation}\label{intbyparts}
 2\langle \nabla f, \nabla g\rangle  =   \langle Zf, g\rangle - 2 \langle \heatop f, g\rangle =  \langle f, Zg\rangle - 2 \langle f, \heatop g\rangle.
\end{equation}
\end{lemma}

\begin{proof}
Noting that  $\nabla G = -\frac {x}{2}G$ and $Z = x\cdot \nabla - 2\cdot \partial_t$ on $\{t=-1\}$, using integration by parts  we obtain
\[
\begin{split}
 2\langle \nabla f, \nabla g\rangle  &=2 \int_{\{t= -1 \}} \nabla f \cdot \nabla g \,G = -2 \int_{\{t=-1\}}  f {\rm div} (\nabla g \,G)
\\
&=  -2\int_{\{t= -1 \}}  \big( f\Delta g \,G + f \nabla g \cdot \nabla G\big)  = -\int_{\{t= -1 \}}  \big( 2f\Delta g \,G - f x \cdot  \nabla g\,  G\big) 
\\ 
&=  -\int_{\{t= -1 \}}  f\big(2 \Delta g- 2\partial_t g - (-2\partial_t g + x\cdot \nabla g) \big) G= -\int_{\{t= -1 \}} f \big(2\heatop g -Zg \big)G
\\
 & =  - 2 \langle f, \heatop g\rangle + \langle f, Zg\rangle.
\end{split}
\]
By symmetry, also the other identity $2\langle \nabla f, \nabla g\rangle  =  - 2 \langle \heatop f, g\rangle+\langle Zf, g\rangle $ holds.
\end{proof}

\begin{lemma} \label{lem:PFF1}
Let $w \in C^{1,1}_{x}\cap C^{0,1}_{t}$. 
Then
\begin{equation}\label{thirda} 
 H'(1,w)  = 2 \langle w, Zw\rangle =   2D(1,w) + 4\langle w,\heatop w\rangle.
\end{equation}
\end{lemma}

\begin{proof}
The first equality is an immediate application of Lemma \ref{lem:PFF0} with $f = w^2$.
To prove the second equality, we use \eqref{intbyparts}.
\end{proof}

\begin{remark}\label{rem:PFF3}
As a consequence of Lemma \ref{lem:PFF1} we obtain the following identity:
\begin{equation} \label{idD} 
D(r,w)  =    \langle w, Zw\rangle  - 2\langle w, \heatop w\rangle.
\end{equation}
\end{remark}

\begin{lemma}\label{lem:PFF4}
Let $w \in C^{1,1}_{x}\cap C^{0,1}_{t}$. Then
\[
D'(1,w)  =   2\langle Zw, Zw\rangle   -  4 \langle Zw, \heatop w\rangle.
\]   
\end{lemma}

\begin{proof}
Using Lemma \ref{lem:PFF0} with $f= |\nabla w|^2$ we obtain 
\[
\begin{split}
 \frac{d}{dr}  D(r,w)\bigg|_{r=1}  &=\frac{d}{dr}\bigg|_{r=1} 2 \bigg(r^2\int_{\{t= -r^2 \}} |\nabla w|^2  G \bigg) =   \sum_i 4 \langle Z\partial_i w, \partial_i w \rangle + 2 D(1,w)  
\\
&=  \sum_i 4 \langle Z\partial_i w  + \partial_i w , \partial_i w \rangle  = \sum_i 4 \langle \partial_i  Zw   , \partial_i w \rangle = 4\langle \nabla  Zw   , \nabla w \rangle
\\
&=   2\langle Zw, Zw\rangle   - 4 \langle Zw, \heatop w\rangle,
\end{split}
\]
where the last equality follows from Lemma \ref{lem:PFF000}.
\end{proof}

We now introduce the frequency functions
\[
\phi(r,w)  : = \frac{D(r,w)}{H(r,w)}  \qquad \textrm{and} \qquad
 \phi^\gamma(r,w)  : = \frac{D(r,w) + \gamma r^{2\gamma}}{H(r,w)+ r^{2\gamma}}.
\]

\begin{lemma} \label{lem:PFF5} Let $w \in C^{1,1}_{x}\cap C^{0,1}_{t}$. Then 
\[
 \frac{d}{dr} \phi^\gamma(r,w) \ge \frac{2}{r}\,  \frac{ \big(\langle Zw,Zw\rangle_r  \langle w,w\rangle_r -  \langle w,Zw\rangle_r^2\big) + \big( 2r^2\langle w,\heatop w\rangle_r\big)^2  + E^\gamma(r,w) }{ \big( H(r,w) + r^{2\gamma}\big)^2 },
\]
where
\begin{equation}\label{eq:ErwgammaP}
E^\gamma(r,w) :=     2r^2\langle w,\heatop w\rangle_r \big( D(r,w)+  \gamma r^{2\gamma}  \big) -  2r^2\langle Zw,\heatop w\rangle_r \big(H(r,w) +r^{2\gamma} \big).
\end{equation}
\end{lemma}

\begin{proof}
By scaling it is enough to compute, for $a  >0$,
\[
\frac{d}{dr}  \phi^{\gamma,a}(r,w)  \bigg|_{r=1}   \qquad\mbox{where} \quad \phi^{\gamma,a}(r,w):=   \frac{D(r,w) + \gamma (a r)^{2\gamma}}{ H(r,w) + (a r)^{2\gamma} }.
\]
By  Lemmas \ref{lem:PFF0}, \ref{lem:PFF1}, and \ref{lem:PFF4}, and by \eqref{idD}, we have
\[
A_1 : =   \frac{d}{dr}\bigg|_{r=1} \big(D(r,w) +\gamma(a r)^{2\gamma} \big) = 2\langle Zw,Zw\rangle - 4\langle Zw,\heatop w\rangle +  2\gamma^2 a^{2\gamma},
\]
\[
D(1,w) +\gamma a^{2\gamma} = \langle w,Zw\rangle -  2 \langle w,\heatop w\rangle +\gamma a^{2\gamma} ,
\]
\[
A_2 := \frac{d}{dr}\bigg|_{r=1} \big( H(r,w) + (a r)^{2\gamma} \big) =  2 \langle w,Zw\rangle+ 2\gamma a^{2\gamma} ,
\]
\[
H(1,w) + (a r)^{2\gamma} =  \langle w,w\rangle + a^{2\gamma}.
\]
Therefore, 
\[
\begin{split}
\frac{d}{dr}\bigg|_{r=1} \phi^{\gamma,a}(r,w)& =  \frac{ A_1 \big( H(1,w) + a^{2\gamma} \big) - A_2 \big(D(1,w) +\gamma a^{2\gamma}  \big)}{ \big( H(1,w) +  a^{2\gamma} \big)^2} 
\\
&= 2\,\frac{X^2 +  2\langle w,\heatop w\rangle\big( \langle Zw, w\rangle +  \gamma a^{2\gamma}  \big) -  2\langle Zw,\heatop w\rangle\big( \langle w,  w\rangle +a^{2\gamma} \big)  }{\big( H(1,w) + a^{2\gamma}\big)^2 },
\end{split}
\]
where 
\[ 
\begin{split}
X^2 :=& \,  \big( \langle Zw,Zw\rangle+ \gamma^2a^{2\gamma} \big)\big( \langle w,w\rangle+ a^{2\gamma} \big) - \big( \langle w,Zw\rangle+ \gamma a^{2\gamma} \big)^2\\
\ge & \,\langle Zw,Zw\rangle  \langle w,w\rangle -  \langle w,Zw\rangle^2 \ge0.
\end{split}
\]
Using again \eqref{idD}, this gives
\[
\begin{split}
\frac{d}{dr}\bigg|_{r=1} \phi^{\gamma,a}(r,w) &\geq 2 \, \frac{\big(\langle Zw,Zw\rangle  \langle w,w\rangle -  \langle w,Zw\rangle^2\big)+4 \big(\langle w,\heatop w\rangle\big)^2   }{ \big( H(1,w) + a^{2\gamma}\big)^2  }\\
&\qquad+  2 \, \frac{2\langle w,\heatop w\rangle \big( D(1,w)+ \gamma  a^{2\gamma}  \big) -  2\langle Zw,\heatop w\rangle\big(H(1,w) +a^{2\gamma} \big)  }{ \big( H(1,w) + a^{2\gamma}\big)^2 }.
\end{split}
\]
Applying the previous equality with $w$ replaced by $w_r = w(r\,\cdot\,, r^2\, \cdot\,)$ and $a$ replaced by $r$, the lemma follows.
\end{proof}

\section{Monotonicity formulae for the Stefan problem}\label{sec:PMF1}


Let $\mathcal P$ denote the set of nonnegative $2$-homogeneous polynomials in the $x$ variable ---i.e. $p=p(x)$--- which satisfy $\Delta p=1$. 
In particular, such polynomials always satisfy the same equation as the parabolic obstacle problem.

\begin{remark}\label{remsignP}  Let $p\in \mathcal P$.
Then,  since  $\heatop u = \chi_{\{u>0\}}$ and  $\heatop p \equiv \Delta p \equiv 1$, we have
\begin{equation}\label{wlapwaa}
(u-p)\heatop (u-p)   =  p\chi_{\{u=0\}} \ge 0,\qquad Z(u-p)\, \heatop (u-p)= Zp\,\chi_{\{u=0\}} = 2p\,\chi_{\{u=0\}}  \ge 0.
\end{equation}
\end{remark}

Note that the solution $u(x,t)$ that we are considering is defined only for $x\in B_1$. However the formulae given in Section \ref{sec:PFF} are for function $w$ defined for all $x$ all of $\R^n$. Morally, we would like to apply the formulae in   Section \ref{sec:PFF} with  $w$ replaced by $u-p$. Since this is not possible, we will instead use make use of the cut-off function 
$\cutoff$ in \eqref{eta} and set  $w:= (u-p)\cutoff$, with $w=0$ outside $B_{1/2}$.
However, such $w$ is not a exact solution and errors coming from this truncation need to be taken into account.

Note that if $u$ is a solution of the parabolic obstacle problem in $\CC_1=B_1 \times(-1,0)$ and $p\in \mathcal P$, then by Remark \ref{remsignP} we have\footnote{From now on, given a function $v: \CC_1 \to \R$, the product $ \cutoff\, v$ will always be seen as a function defined in all $\R^n\times(-1,0)$, where we extended the product by $0$ outside of $B_1$.
}
\begin{equation}\label{ZwLapwbbb}
\langle \cutoff \,Z(u-p), \cutoff \,\heatop (u-p) \rangle_r=  2\langle  \cutoff\,(u-p) ,  \cutoff \,\heatop (u-p) \rangle_r\ge 0.
\end{equation}

The following lemma will be useful  to control truncation errors, which are exponentially small as $r\downarrow 0$.

\begin{lemma}\label{lemerrors}
Let $v \in C^{1,1}_{x}\cap C^{0,1}_{t}(\CC_1)$. Then, for all $r\in (0,1/2)$ we have
\[
 \big| \langle \cutoff\, v ,\heatop (\cutoff v)\rangle_{r}  - \langle\cutoff\, v,\cutoff\,\heatop v\rangle_{r}  \big| +\big| \langle Z(\cutoff v) ,\heatop (\cutoff v)\rangle_{r}  - \langle\cutoff\,Zw,\cutoff\,\heatop v\rangle_{r}  \big| \le C M_v^2\,   {\textstyle \exp\left( - \frac{1}{(8r)^2} \right)},
\]
where
\[
M_v : = \sup_{B_{1/2}\times (-1/2,0)} |v| + |\nabla v| + |\partial_t v| + |\Delta v|,
\] 
and $C$ depends only on $\cutoff$.
\end{lemma}

\begin{proof}
We note that, since $\cutoff= \cutoff(x)$,
\[
\heatop (\cutoff v) =  \cutoff\, \heatop v   + \Delta \cutoff \,v + 2  \nabla \cutoff  \cdot  \nabla v 
\qquad\text{and}\qquad 
Z (\cutoff v)  =  \cutoff\,Zv +  (x\cdot\nabla \cutoff) v.
\]
Thus
\[
\begin{split}
\langle \cutoff\, v ,\heatop (\cutoff v)\rangle_{r}  = \int_{\{t= -r^2\}} \cutoff\, v\, \heatop (\cutoff v) \, G
= \int_{\{t= -r^2\}} \cutoff\, v\, \cutoff \, \heatop v \,G    +  \int_{\{t= -r^2\}}  \cutoff\, v\,( \Delta \cutoff\, v  + 2  \nabla v \cdot  \nabla\cutoff ) \,G,
\end{split}
\]
and similarly 
\[
\begin{split}
\langle Z(\cutoff v) ,\heatop (\cutoff v)\rangle_{r}  &= \int_{\{t= -r^2\}} Z(\cutoff v) \,\heatop (\cutoff v) \, G
\\
&= \int_{\{t= -r^2\}} \cutoff \,Zv \,\cutoff \,\heatop v\, G    + \int_{\{t= -r^2\}}  \big(\cutoff \,Zv( \Delta \cutoff\, v  + 2  \nabla v \cdot  \nabla\cutoff ) + (x\cdot\nabla \cutoff)\, v \,\heatop v\big)\, G.
\end{split}
\]
Since $\Delta \cutoff$, $\nabla \cutoff$, and $Z\cutoff$ vanish in $B_{1/4}$, and 
\[G(x,-r^2) =  \frac{1}{(4\pi)^{n/2}r^n}  \exp\left( {\textstyle  - \frac{|x|^2}{4r^2}} \right) \le C_n \exp\left( {\textstyle - \frac{1}{(8r)^2}} \right)\quad \mbox{  for }|x|\ge \frac 1 4,
\]
the lemma follows.
\end{proof}

We next prove the following Weiss' type monotonicity formula (as well as a useful consequence of it).

\begin{lemma}[Weiss'-type formula]\label{LOClem:WeissP}
Let $u:B_1\times(-1,1) \to [0,\infty)$ be a bounded solution of  \eqref{eq:UPAR1},  and $(0,0)$  a singular point. Given $p \in \mathcal P$, set $w:= (u-p)\cutoff$.   Let 
\[W_{\lambda} (r, w) :=   r^{-2\lambda}  \big( D(r,w)-   \lambda H(r,w)\big).\] Then:

{\rm (a)} For all $\lambda\ge 2$
\[
\hspace{25mm} \frac{d}{dr} W_\lambda(r,w) \ge2 r^{-2\lambda-1}   \langle Zw-\lambda w, Zw-\lambda w\rangle_r  ^2  -C e^{-\frac{1}{r}}, \qquad \forall \,r\in (0,{\textstyle \frac12}),
\]
where $C>0$ depends only on $n$, $\|u(\,\cdot\,, 0)\|_{L^\infty(B_1)}$, and $\lambda$.

{\rm (b)} Moreover,  $W_2(0^+,w) =0$ and
\begin{equation}\label{eq:Dge2HP}
\hspace{25mm}D(r,w)- 2H(r,w) \ge -C e^{-\frac{1}{r}}, \qquad\forall \,r\in (0,{\textstyle \frac12}),
\end{equation}
where $C>0$ depend only on $n$ and $\|u(\,\cdot\,, 0)\|_{L^\infty(B_1)}$.
\end{lemma}

\begin{proof}
As we shall see, (b) is a rather direct consequence of (a). We begin with the first part.

\smallskip

\noindent
$\bullet$ {\it Proof of (a).} To compute $W'_{\lambda}(r,w)$ we use Lemmas \ref{lem:PFF1} and  \ref{lem:PFF4}, as well as \eqref{idD}  (rescaled), to obtain
\[
\begin{split} 
 W'_{\lambda}(r,w) &=    \frac{1}{r^{2\lambda}} \left( D'(r,w)-  \lambda  H'(r,w)\big)  - \frac{2\lambda}{r} D(r,w) +  \frac{2\lambda^2}{r} H(r,w)\right)
\\ &=  \frac{2}{r^{2\lambda+1}}\big( \langle Zw, Zw\rangle_r  - 2 r^2\langle Zw, \heatop w\rangle_r - \lambda \langle w, Zw\rangle_r  - \lambda D(r,w)  + \lambda^2  H(r,w)\big)
\\ & =  \frac{2}{r^{2\lambda+1}} \big(  \langle Zw, Zw\rangle   - 2 r^2\langle Zw,\heatop w\rangle_r  -  \lambda\langle w, Zw\rangle_r-\lambda(\langle w, Zw\rangle_r -  2r^2\langle w,\heatop w\rangle_r) + \lambda^2\langle w, w\rangle_r\big)
\\& =  \frac{2}{r^{2\lambda+1}} \big( \langle Zw-\lambda w, Zw-\lambda w\rangle_r  + 2r^2 \big\{ \lambda \langle w,\heatop w\rangle   -  r^2\langle Zw,\heatop w\rangle \big\} \big).
\end{split}
\]
Set $\tilde w := u-p_2$, so that $w=\zeta\tilde w.$
Using  \eqref{ZwLapwbbb} (see also Remark \ref{remsignP}) with $p=p_2$, and by Lemma \ref{lemerrors} applied to $\tilde w$, we have 
\begin{multline}
\lambda \langle w,\heatop w\rangle_r   -  \langle Zw,\heatop w\rangle_r  =\lambda \langle \cutoff \tilde w ,\heatop (\cutoff \tilde w) \rangle_r   -  \langle Z(\cutoff \tilde w),\heatop (\cutoff \tilde  w)\rangle_r \\
\ge   \lambda \langle \cutoff \tilde w , \cutoff \,\heatop \tilde w \rangle_r   -  \langle  \cutoff\,Z \tilde w,\heatop  \tilde  w\rangle_r  - C_\cutoff \big(M_{\tilde w}\big)^2 e^{-\frac{1}{(8r)^2}}
 =  (\lambda-2) \langle \cutoff \tilde w , \cutoff \,\heatop \tilde w \rangle_r  - C e^{-\frac{2}{r}},
\end{multline}
where $M_{\tilde w}$ is defined as in \eqref{lemerrors}. Recalling \eqref{optimalreg+nondegP} and that the cut-off $\zeta$ is fixed, it follows that $M_{\tilde w}$ (and therefore also the constant $C$ above) depends only on $n$ and $\|u(\,\cdot\,, 0)\|_{L^\infty(B_1)}$.
Since  $\lambda\ge2$, $-r^{-2\lambda+ 1}  e^{-\frac{2}{r}} \ge - C e^{-\frac{1}{r}}$, and $\langle \cutoff \tilde w , \cutoff \,\heatop \tilde w \rangle_r\ge 0$ (see Remark \ref{remsignP}), the conclusion follows.

\smallskip

\noindent
$\bullet$ {\it Proof of (b).}  Since $(0,0)$ is a singular point we have that $w_r(x) = (u-p)(r x, r^2 t)  = (p_2-p) (rx)+ o(r^2)$, and therefore
 \[W_2(0^+, w) = \lim_{r\downarrow 0} W(1, r^{-2}w_r) =  W_2(1, p_2-p) =0.\]
As a consequence,  \eqref{eq:Dge2HP} follows by integrating $(a)$ with $\lambda=2$ between $0$ and $r$.
\end{proof}

We can now prove our approximate frequency formula.

\begin{proposition}[Frequency formula]\label{prop:EMF1P}
Let  $u:B_1\times(-1,1) \to [0,\infty)$ be a  bounded solution of  \eqref{eq:UPAR1},  and  $(0,0)$ a singular point.  Given $p \in \mathcal P$, set $w:= (u-p)\cutoff$. 
Then, for all 
$\gamma\in (2, \infty)$ we have
\[
\frac{d}{dr} \phi^\gamma(r, w) \ge \frac 2 r  \,\frac{ \big(\langle Zw,Zw\rangle_r  \langle w,w\rangle_r -  \langle w,Zw\rangle_r^2\big) +\big( 2r^2\langle w,\heatop w\rangle_r \big)^2 }{  \bigl(H(r,w)+r^{2\gamma}\bigr)^2  }  -C e^{-\frac1r} , \quad \forall \,r\in (0,1),
\]
and 
\[
r^2\langle w,\heatop w\rangle_r \ge -C e^{-\frac{1}{r}} , \quad \forall \,r\in (0,1),
\]
where  $C>0$  depends only on $n$, $\|u(\,\cdot\,, 0)\|_{L^\infty(B_1)}$, and $\gamma$.
\end{proposition}

\begin{proof}
Since $-r^{-4\gamma}e^{-\frac{2}{r}} \geq -Ce^{-\frac{1}{r}}$, thanks to  Lemma \ref{lem:PFF5} it suffices to show that $E^\gamma(r,w) \ge  -C e^{-\frac{2}{r}}$. 
In particular, it is enough to prove that
\[
\langle w,\heatop w\rangle_r \big( D(r,w)+  \gamma r^{2\gamma}  \big) -  \langle Zw,\heatop w\rangle_r \big(H(r,w) +r^{2\gamma} \big) \ge  -C e^{-\frac{2}{r}}.
\]
On the one hand, by Lemma  \ref{LOClem:WeissP}, for $\gamma>2$ and $r$ sufficiently small we have  $$0\le   \frac 1 2 D(r,w)+ \frac{\gamma}{2} r^{2\gamma}   -H(r,w) -2r^{2\gamma}  \le C.$$
Thus, using \eqref{ZwLapwbbb} and Lemma \ref{lemerrors}  (as in the proof of Lemma \ref{LOClem:WeissP}) we obtain
\[
\langle w,\heatop w\rangle_r  \big( D(r,w)+ \gamma r^{2\gamma}   -2H(r,w) -2r^{2\gamma} \big) \ge   -C e^{-\frac{2}{r}}.
\]
On the other hand, we similarly have $0\le H(r,w) +r^{2\gamma} \le C$, so by using \eqref{ZwLapwbbb} and Lemma \ref{lemerrors}   we obtain
\[
\big(2\langle w,\heatop w\rangle_r - \langle Zw,\heatop w\rangle_r \big)\big(H(r,w) +r^{2\gamma} \big) \ge  -C e^{-\frac{2}{r}}.
\]
This finishes the proof.
\end{proof}

We now begin to discuss a series of consequences from the previous bounds.
The first observation is that, as an immediate consequence of \eqref{eq:Dge2HP}, the following holds:

\begin{lemma}\label{lem:PMF2}
Let $u:B_1\times(-1,1) \to [0,\infty)$ be a bounded solution of  \eqref{eq:UPAR1}, and $(0,0)$ is a singular point.   Given $p \in \mathcal P$, set $w:= (u-p)\cutoff$.
Then, for all $\gamma > 2$, it holds $\phi^\gamma(0^+,w) \geq 2$.
\end{lemma}

We also have the following:

\begin{lemma}\label{lem:HP}
Let $u:B_1\times(-1,1) \to [0,\infty)$ be a  bounded solution of  \eqref{eq:UPAR1}, and $(0,0)$ a singular point.  Given $p \in \mathcal P$, set $w:= (u-p)\cutoff$, and let $\gamma>2$.
\begin{enumerate}
\item[(a)] We have 
\[
\frac{d}{dr} \frac{H(r,w)+ r^{2\gamma}}{r^{4}}  \ge -Ce^{-\frac{1}{2r}}.
\]

\item[(b)]  Assume  $\phi^\gamma(0^+,w) = \lambda\le 3$, and  that there exist $\delta>0$ and $R\in (0,1)$ such that  $\phi^\gamma(r, w)\le \lambda + \frac{\delta}{2}$  for all $r\in (0,R)$. Then
\[
c \left(\frac{R}{r} \right)^{2 \lambda}  \le \frac{H(R,w)+ R^{2\gamma}}{H(r,w)+r^{2\gamma}} \le  C_\delta  \left(\frac{R}{r} \right)^{2\lambda+\delta},
\]
where $c,C_\delta>0$  depend only on $n$, $\|u\|_{L^\infty}$,  $\gamma$, and  (only in the case of $C_\delta$) $\delta$.
\end{enumerate}
\end{lemma}

\begin{proof}
It is  a minor modification of the proof of \cite[Lemma 4.1]{FRS}. We define
\[
F(r)  := \frac{ 2r^2\langle w,\heatop w\rangle_r }{H(r,w) + r^{2\gamma}},
\]
and use that, by Proposition  \ref{prop:EMF1P},
\[
\frac{d}{dr} \phi^\gamma(r,w)  \ge   \frac{2}{r} \big(F(r)\big)^2 - Ce^{-\frac{1}{r}}.
\]
In particular, thanks to Lemma~\ref{lem:PMF2}, it follows that 
\begin{equation}\label{phigamma2}
\phi^\gamma(r,w) \ge 2 - Ce^{-\frac{1}{r}}.
\end{equation}
In particular, from Lemma \ref{lem:PFF1} we have
\begin{equation}\label{ahiguaifguia}
\begin{split}
\frac{\frac{d}{dr} (H(r,w) +r^{2\gamma})}{ (H(r,w)+r^{2\gamma})} &= \frac  2 r \, \frac{ D(r,w)  +\gamma r^{2\gamma} + 2r^2\langle w,\heatop w\rangle_r}{H(r,w) +r^{2\gamma}}
= \frac 2 r   \phi^\gamma(r,w) +  \frac{2}{r} F(r)\\
&\geq  \frac 2 r   \phi^\gamma(r,w) -  \frac{C}{r} e^{-\frac{1}{r}}\geq \frac 4 r   -  C e^{-\frac{1}{2r}},
\end{split}
\end{equation}
where the first inequality follows from Proposition  \ref{prop:EMF1P}, while the second one from \eqref{phigamma2}.

Then, (a) is a direct consequence of \eqref{ahiguaifguia}.  
On the other hand, (b) follows exactly as in the proof of \cite[Lemma 4.1]{FRS}, that is, integrating the first line of \eqref{ahiguaifguia} and using the Cauchy-Schwartz inequality to control  $\left|\int_r^R  \frac{1}{\rho} F(\rho)   d\rho\right|$.
\end{proof}

It is well known (see \cite{Blanchet}) that the singular set $\Sigma$ can be split into the following sets:
\begin{equation} \label{ahoiah0P}
\Sigma_m :=  \big\{ (x_\circ, t_\circ) \mbox{ singular point with } {\rm dim}\big( \{p_{2,x_\circ, t_\circ} =0 \}\big) =m    \big\}\,, \quad 0\le m \le n-1.
\end{equation}
Recall that, as an immediate consequence of \cite[Theorem 1.9]{LM}, we have the following.

\begin{proposition}\label{LinMon-m}
Let $u\in C^{1,1}_{x}\cap C^{0,1}_{t}(B_1\times(-1,1))$ solve \eqref{eq:UPAR1}.
Then $\pi_x(\Sigma_m) \subset B_1$ can be locally covered by an $m$-dimensional $C^1$ manifold. 
\end{proposition}

We next show the following:

\begin{lemma}\label{lem:freq2and3}
Let $u:B_1\times(-1,1) \to [0,\infty)$ be a  bounded solution of  \eqref{eq:UPAR1},  and $(0,0)\in \Sigma_m$.   
Given $p \in \mathcal P$, set $w:= (u-p_2)\cutoff$.
\begin{itemize}
\item[(a)] If $0\leq m\leq n-2$, then $\phi^\gamma(0^+, w) =2$ for all $\gamma>2$. 

\item[(b)] If $m=n-1$, then $\phi^\gamma(0^+, w)\leq 3$ for all $\gamma>3$. 
\end{itemize}
\end{lemma}

\begin{proof}
As we shall see, the proofs of the two results are very similar.

\smallskip

\noindent
$\bullet$ {\it Proof of (a).} By  Lemma \ref{lem:HP}(b) it is enough to show that if  $(0,0)\in \Sigma_m$, with $m\leq n-2$,  then for each $\epsilon>0$ we have  
\begin{equation}\label{whioghwioh}
\liminf_{r\downarrow 0}  \frac{\ave_{B_r} |u(\cdot, -r^2)-p_2 | }{ r^{2+\epsilon}} >0.
\end{equation}
We will use a subsolution to show that, since $(0,0)$ is a singular point ant $\partial_t u>0$ outside of the contact set $\{u=0\}$, then
\begin{equation}\label{0+eps}
\ave_{\partial B_r} \partial_t u( x , t) \,dx \ge  c_\epsilon r^{\epsilon} \qquad \forall \,t\in (-r^2,0).
\end{equation}
Then using that both $u$ and $\partial_t u$ are nonnegative, \eqref{whioghwioh} follows immediately by integrating \eqref{0+eps} with respect to $t$. 

To show \eqref{0+eps}, set $L := \{p_2=0\}$ and define the  ``parabolic cone''
\[
\mathcal C_\delta : = \left\{  (x,t) \ :\ x\in \R^n,  t<0,  \  {\rm dist}(x,L) \ge \delta (|x| + |t|^{1/2})\right\}.
\]
Note that, since $r^{-2}u_r -p_2 \to 0$   locally uniformly, for all $\delta>0$ there exists $r_\delta>0$  such that 
\[
\mathcal C_{\delta/2} \cap \CC_{r_\delta} \subset \{u>0\}
\]
(recall that $\CC_{r_\delta}=B_ {r_\delta} \times (-r_\delta^2, 0)$).
We now consider the following eigenvalue problem and make the following claim:

\smallskip

\noindent {\em
{\bf Claim:} For any $\epsilon>0$ there exists $\delta$ such that the following holds: There exist
$N \in (0,\epsilon)$ and a positive
 $N$-homogenous function $\Phi:\R^n\times (-\infty,0)\to (0,+\infty)$ (i.e., $\Phi(rx,r^2t) =  r^N \Phi(x,t)$ for all $r>0$) satisfying  }
$$
\text{$\heatop \Phi =0$ in $\mathcal C_\delta$}\qquad\text{ and }\qquad \text{$\heatop \Phi =0$ on  $\partial  {\mathcal C_\delta}$.}
$$

\smallskip

\noindent
To prove this, we look for $\Phi$ of the form  $\Phi(x,t) = t^{N/2}\phi(x/t^{1/2})$, where $\phi:\R^n\to \R_+$ solves the following eigenvalue problem for the Ornstein-Uhlenbeck operator in $\widetilde{\mathcal  C_\delta}: =  \{x_n \ge \delta (|x| +1) \}$:
\[
\begin{cases}
\mathcal L_{OU} \phi +\frac{N_\delta}{2} \phi =0 \quad &\mbox{in } \widetilde{\mathcal  C_\delta}\\
 \phi (x) =0 &\mbox{on }\partial \widetilde{\mathcal  C_\delta},
\end{cases}
\]
where
\begin{equation}\label{auiwgugwar}
\mathcal L_{OU}  \phi (x):=  \Delta \phi (x) - \frac{x}{2} \cdot \nabla \phi(x)  = e^{|x|^2 /4} {\rm div}( e^{-|x|^2 /4} \nabla \phi)
\end{equation}
is the Ornstein-Uhlenbeck operator.

Note that, as $\delta \to 0$, the first eigenvalue $N_\delta$ small is very close to the first eigenvalue $N_0$ of $\mathcal L_{OU}$ in  the limiting domain $\R^n \setminus L$. 
Also, since ${\rm dim}(L) \leq n-2$, $L$ has zero harmonic capacity. Therefore, it follows by the Raleigh quotient characterization 
\begin{equation}
\label{eq:Raleigh}
\frac{N_0}{2} = \inf_{\phi \in C^1_c (\R^n\setminus L)}  \frac{\int |\nabla \phi|^2\,  e^{-|x|^2 /4} }{\int \phi^2 \,e^{-|x|^2 /4}}
\end{equation}
that $N_0=0$, hence $N_\delta$ can be made arbitrarily small provided we choose $\delta$ small enough. This proves the claim.

\smallskip 

We now note that
\[
\overline{\mathcal C_\delta} \cap \partial_{par}  \CC_{r_\delta} \subset\subset \{u>0\},
\]
and that $\partial_t u$ is positive and supercaloric inside $\{u>0\}$. Hence, we can use  $\Phi$ as lower barrier to show that 
\[
\partial_t u \ge c \Phi  \qquad \mbox{in }\CC_{r_\delta},
\]
for some $c>0$, and \eqref{0+eps} follows.
 
\smallskip

\noindent
$\bullet$ {\it Proof of (b).} 
This part is similar to the argument used for (a). However in this case,  since ${\rm dim}(L) =n-1$, $L$ is a set of positive harmonic capacity and therefore $N_0 >0$. By Talenti's rearrangement (with Gaussian weights), one can easily show that the infimum in \eqref{eq:Raleigh} is attained by the function $|x_n|$, and thus $N_0=1$. Therefore, in this case we can find a $N$-homogeneous function $\Phi$ as above, but with $N \leq 1+\epsilon$.
Using this $\Phi$ as lower barrier for $\partial_t u$, we conclude that 
$$
\ave_{\partial B_r} |\partial_t u|( x , t) \,dx \ge  c_\epsilon r^{1+\epsilon} \qquad\Rightarrow\qquad  \liminf_{r\downarrow 0}  \frac{\ave_{B_r} |u(\cdot, -r^2)-p_2 | }{ r^{3+\epsilon}} >0
$$
(cp.  \eqref{0+eps}-\eqref{whioghwioh}), which leads to $\phi(0^+,u-p_2) \le 3$.
\end{proof}

Choosing $\gamma>3,$ as a direct consequence of Lemmas \ref{lem:freq2and3} and \ref{lem:HP} we obtain the following:

\begin{corollary}\label{corqwrtyu}
Let $u:B_1\times(-1,1) \to [0,\infty)$ be a  bounded solution of  \eqref{eq:UPAR1},  and $(0,0)$ a singular point. Then, for any $\delta>0,$
\[\mbox{$H\big(r,(u-p_2)\cutoff\big)\gg r^{3+\delta}$}\qquad \text{as $r\downarrow 0$.}\]
 Moreover,  the limit  $\lim_{r\downarrow 0}\phi\big(r, (u-p_2)\cutoff\big)$ exists and equals $\phi^\gamma(r, (u-p_2)\cutoff)$ for any $\gamma>3$. 
 Finally, given $K>1$, there exists $C_K>1$ such that the following holds:
$$
\frac{1}{C_K}\leq \frac{H(r,(u-p_2)\cutoff)}{H(\theta r,(u-p_2)\cutoff)} \leq C_K\qquad \forall\,r \in (0,1),\,\theta \in [K^{-1},K].
$$
\end{corollary}

\section{The  2nd blow-up} \label{sec:E2BP}

After the preliminary results from the previous sections, we now start investigating the structure of ``2nd blow-ups'' at singular points, namely, blow-ups of $u-p_2$.

The following lemma gives estimates for the difference of two solutions.
Note that both  $u$ and  $p_2$ satisfy the same equation $\heatop v = \chi_{\{v>0\}}$. Although we will first apply the following lemma to $u$ and $p_2$, we give a more general version (for difference of approximate solutions) which will be very useful later on in the paper. Recall that $\CC_r=B_r\times (-r^2,0)$.

\begin{lemma}\label{lem:u-v}
Let $u_i: \CC_2\to  \R$, $i =1,2$, solve
\begin{equation}\label{gakljhgaghk;}
\begin{cases}
\heatop u_i = \chi_{\{u_i>0\}}(1+ \ep_i(x,t))\quad \textrm{in} \quad \CC_2
\\
u_i \ge 0
\\
\partial_t u_i\ge 0,
\end{cases}
\end{equation}
with $|\ep_i(x,t)|\le \bar \ep<\frac{1}{100}$, and denote $w := u_1-u_2$.
Then
\begin{equation}\label{hhhaiuguiag1}
\| w\|_{L^\infty(\CC_1)}  \le C\big(  \|w\|_{L^2(\CC_2)}  + \bar\ep \big)
\end{equation}
and
\begin{equation}\label{hhhaiuguiag2}
\left(\int_{\CC_1} |\nabla w|^2   + |w\,\heatop w|  + |\partial_t w|^2 \right)^{\frac 1 2}  \le C\big(  \|w\|_{L^2(\CC_2)}  + \bar\ep\big),
\end{equation}
where $C$ is a dimensional constant.
\end{lemma}

\begin{proof} 
We first prove \eqref{hhhaiuguiag1}. 
On the one hand, we note that inside $\{u_1=0\}$ it holds
\begin{equation}
\heatop w = 1 +\ep_1- \heatop u_2 \ge 1-\bar \ep - (1+\ep_2) \chi_{\{u_2 >0\}} \ge  -2\bar\ep ,\qquad  w =  -u_2  \le 0,
\end{equation}
therefore $w_+ = \max(0, w)$  satisfies $\heatop w_+\ge -2\bar \ep$ in  $\CC_2$.  Since $(u_2-u_1)_+ = w_-$,  by symmetry we also have  $\heatop w_-\ge -2\bar \ep$. Thus, 
\[
(\Delta -\partial_t) |w| \ge -2\bar \ep \qquad \mbox{ in } \CC_2,
\]
so the parabolic Harnack inequality yields
\[
\sup_{B_1\times (-1,0)} |w| \le C_n \bigg(\int_{B_2\times (-2,0)}  |w|  + \bar \ep\bigg)  \le C_n  \big(  \|w\|_{L^2(B_2 \times(-2,0))}  + \bar \ep\big),
\]
which proves \eqref{hhhaiuguiag1}.

We next prove \eqref{hhhaiuguiag2}.
Notice that  $w\,\heatop w \ge- 2\bar \ep |w|$, since
\[
(u_1-u_2)\,\heatop(u_1-u_2) = 
\begin{cases}
-|u_1-u_2| (|\ep_1| +|\ep_2|) \ge -2\bar \ep |u_1-u_2|  \quad &\mbox{if } u_1>0,\  u_2>0 \\
u_1(1+\ep_1-0)\ge 0  &\mbox{if } u_1>0,\  u_2=0 \\
-u_2(0-1-\ep_2 )\ge 0  &\mbox{if } u_1=0,\  u_2>0 \\
0&\mbox{if } u_1=0,\  u_2=0.
\end{cases}
\]
Similarly   $\partial_t w\,\heatop w \ge- 2\bar \ep |\partial_t w|$, since
\[
\partial_t (u_1-u_2)\,\heatop(u_1-u_2) = 
\begin{cases}
-|\partial_t(u_1-u_2) | (|\ep_1| +|\ep_2|) \ge -2\bar \ep |\partial_t(u_1-u_2)|  \quad &\mbox{if } u_1>0,\  u_2>0 \\
\partial_t u_1(1+\ep_1-0)\ge 0  &\mbox{if } u_1>0,\  u_2=0 \\
- \partial_t u_2(0-1-\bar\ep_2 )\ge 0  &\mbox{if } u_1=0,\  u_2>0 \\
0&\mbox{if } u_1=0,\  u_2=0.
\end{cases}
\]
On the one hand, choosing 
\begin{equation}\label{hgyfghilkv}
\tilde \eta= \tilde \eta (x) \in C^\infty_{c}(B_{5/3})  \quad \mbox{nonnegative and such that  $\tilde \eta\equiv 1$ in $\overline B_{4/3}$}
,\end{equation}
we obtain
\begin{equation}\label{anioqioqbioboi2bis}
\frac{d}{dt} \int_{B_2 \times \{t\}}  w^2 \tilde  \eta  + 2\int_{B_{2} \times \{t\}}  (w\,\heatop w  + |\nabla w|^2) \tilde \eta         =    \int_{B_2\times\{t\} }   \Delta ( w^2)\tilde \eta 
=   \int_{B_{2}\times\{t\}} w^2 \Delta \tilde \eta \le C \int_{B_{5/3}\times\{t\} }  w^2.
\end{equation}
Integrating in $t$ the last inequality, and recalling that $w\,\heatop w \geq -2\bar\ep |w|$, using \eqref{hgyfghilkv} we obtain
\[
\int_{-3/2}^0\biggl(\int_{B_{4/3} \times \{t\}}  \big(|w\,\heatop w|  + |\nabla w|^2\big) \biggr)\,dt \le C_n \sup_{t\in (-3/2,0)} \int_{B_{5/3}\times \{t\}} \big(|w|^2 + \bar \ep |w| \big) .
\]
Now from the $L^\infty$ estimate in Step 1 (plus a standard covering argument, to obtain a control in $\CC_{3/2}$ instead of $\CC_1$), we obtain that 
\[
\sup_{t\in (-3/2,0)}  \int_{B_{5/3}\times\{t\} } \big(|w|^2 + \bar \ep |w| \big)   \le C_n\big( \|w\|_{L^2(\CC_2)} +\bar \ep \big)^2,
\]
and therefore
\begin{equation}\label{anioqioqbioboi2}
\int_{-3/2}^0\biggl(\int_{B_{4/3} \times \{t\}}  \big(|w\,\heatop w|  + |\nabla w|^2\big) \biggr)\,dt \le C_n\big( \|w\|_{L^2(\CC_2)} +\bar \ep \big)^2.
\end{equation}
On the other hand, choosing now $\eta= \eta (x) \in C^\infty_{c}(B_{4/3}) $ such that $\eta\equiv 1$ in $\overline B_{1}$, and using $\heatop w ( \partial_ t w) \ge -2\bar\ep |\partial_ t w|$,  for any $M>0$ we have
\[
\begin{split}
2\int_{B_2 \times \{t\}} (\partial_t w)^2 \eta^2  + & \frac{d}{dt} \int_{B_2 \times \{t\}} |\nabla w|^2 \eta^2     =  2\int_{B_2 \times \{t\}}  (\partial_t w)^2  \eta^2   +\int_{B_2\times\{t\} }2 \nabla w  \cdot (\nabla \partial_ t w)\eta^2 
\\
& =   \int_{B_2\times\{t\} }  \Big(2(\partial_t w -\Delta w ) ( \partial_ t w)\eta^2  - 4 \partial_t w \nabla w\cdot \nabla \eta \,\eta\Big)
\\
&\le    \int_{B_2\times\{t\} }  \Big( 4\bar \ep  |\partial_ t w|\eta^2  +2 M | \nabla w|^2 |\nabla \eta|^2  + \frac{2}{M}   (\partial_ t w)^2 \eta^2\Big).
\end{split}
\]
Since $4\bar\ep |\partial_t w| \leq C\bar\ep^2+\frac12 |\partial_t w|^2$, choosing $M=4$ and using that $\eta$ is supported in $B_{4/3}$, we obtain  
$$
\int_{B_1 \times \{t\}} (\partial_t w)^2 \eta^2  +  \frac{d}{dt} \int_{B_2 \times \{t\}} |\nabla w|^2 \eta^2     \le  C_n  \bigg( \int_{B_{4/3}\times\{t\} }   | \nabla w|^2  + \bar \ep^2\bigg).
$$
Multiplying this inequality by a nonnegative smooth function $\psi \in C_c^\infty ((-3/2,0])$ such that $\psi|_{[-1,0]}\equiv 1$, and integrating over $[-3/2,0]$, we get 
\begin{align*}
\int_{\CC_1} (\partial_t w)^2 \eta^2   &\le \int_0^{3/2} \biggl(\int_{B_1}(\partial_t w)^2 \eta^2\biggr)\psi(t)\,dt \\
&\leq   \int_0^{3/2} \biggl( \int_{B_2 \times \{t\}} |\nabla w|^2 \eta^2 \biggr)\psi'(t)\,dt +C_n   \bigg[\int_0^{3/2} \biggl(\int_{B_{4/3}}  | \nabla w|^2\biggr) \psi(t) \,dt + \bar \ep^2\bigg]\\
&\leq C_n   \bigg[\int_0^{3/2} \biggl(\int_{B_{4/3}}  | \nabla w|^2\biggr) \,dt + \bar \ep^2\bigg].\end{align*}
Combining this bound with  \eqref{anioqioqbioboi2}, the result follows.
\end{proof}

As a consequence, we find:

\begin{corollary}\label{lem:H1spat}
Let $u:B_1\times(-1,1) \to [0,\infty)$ be a  bounded solution of  \eqref{eq:UPAR1}, $(0,0)$ a singular point, and let $w:= u-p_2$. 
Then, for all $r\in (0,1)$,
\[
\|w_r\|_{L^\infty(\CC_1)}+ \|\partial_t w_r\|_{L^\infty(\CC_1)}  +\|\nabla w_r\|_{L^2(\CC_1)} +\|\nabla\partial_t w_r\|_{L^2(\CC_1)}
 \le  C\|w_r\|_{L^2(\CC_2)}.
\]
Also, for all $e\in\{p_2=0\}\cap \mathbb S^{n-1}$, we have
\[
\|\nabla \partial_e w_r\|_{L^2(\CC_1)}   \le  C\|w_r\|_{L^2(\CC_2)} .
\]
The constant $C$ depends only on $n$ and $\|u\|_{L^\infty}$.
\end{corollary}
\begin{proof}
We apply Lemma \ref{lem:u-v} with $u_1= r^{-2}u_r$, $u_2 =p_2$,  and $\bar \ep =0$, to obtain 
\[
\|w_r\|_{L^\infty(\CC_1)} + \|\partial_t w_r\|_{L^2(\CC_1)}  +  \|\nabla w_r\|_{L^2(\CC_1)} \le  C\|w_r\|_{L^2(\CC_2)}. \]
Now, to prove the $L^\infty$ bound on $\partial_t w_r$ we observe that $\partial_t(u-p_2) = \partial_t u \ge 0$ and $\heatop\partial_t(u-p_2)= \partial_t \chi_{\{u>0\}}\ge 0$. Hence $\partial_t w_r$ is nonnegative and subcaloric, and therefore the weak Harnack inequality implies that
\[
\|\partial_t w_r\|_{L^\infty(\CC_{1/2})} \le C_n\|\partial_t w_r\|_{L^1(\CC_{1})} \le C_n\|\partial_t w_r\|_{L^2(\CC_{1})}.
\]
\
Also, since $(\partial_t w_r)\,\heatop (\partial_t w_r)\ge 0$, similarly to \eqref{anioqioqbioboi2} (replacing $w$ by $\partial_t w_r$) we obtain 
\[
\|\nabla \partial_t w_r\|_{L^2(\CC_{1/2})} \le C\|\partial_t w_r\|_{L^2(\CC_{1})}\le C\|\partial_t w_r\|_{L^\infty(\CC_{1})}\le  C\|w_r\|_{L^2(\CC_2)} .
\]
By a covering argument we can replace $\CC_{1/2}$  by $\CC_{1}$, the first part of the corollary follows.

Finally, given $e\in\{p_2=0 \}$, since $\partial_e p_2\equiv 0$ we obtain that  $\partial_e w_r =0$  on $\{u=0\}$ and that $\heatop \partial_e w_r=0$  in $\{u>0\}$. This yields $\partial_e w_r \heatop\partial_e w_r =0$, and arguing as before we obtain 
\[
\|\nabla \partial_e w_r\|_{L^2(\CC_{1/2})} \le C\|\partial_e w_r\|_{L^2(\CC_{1})}\le C\|\nabla w_r\|_{L^2(\CC_{1})}\le  C\|w_r\|_{L^2(\CC_2)}  .
\]
Again by a covering argument we can replace $\CC_{1/2}$  by $\CC_{1}$, which concludes the proof.
\end{proof}

The next lemma will be very useful in the sequel.
We recall that $\zeta$ denotes the cut-off function defined in Section~\ref{sect:cutoff}.

\begin{lemma}\label{lemvyg8g276187}
Let $u:B_1\times(-1,1) \to [0,\infty)$ be a  bounded solution of  \eqref{eq:UPAR1}, $(0,0)$ a singular point, and let $w:= u-p_2$.  Then, for all $r\in (0,1)$ we have 
\[
\frac{1}{C}H(r, w\cutoff)^{1/2} \le  \| w_r\|_{L^2(\CC_1)} \le C H(r, w\cutoff)^{1/2},
\]
for some constant $C>1$ depending only on $n$ and $\|u\|_{L^\infty}$.
\end{lemma}

To prove Lemma \ref{lemvyg8g276187}, we will use the following simple consequence of the Gaussian log-Sobolev inequality.

\begin{lemma}\label{lem877r6edfcv}
Let $f:\R^n \to \R$ satisfy
\[
\int_{\R^n} f^2   dm = 1 \quad \mbox{and}\quad \int_{\R^n} |\nabla f|^2   dm \le 4,
\]
where $dm = G(x,-1)dx=\frac{1}{(4\pi)^{n/2}} \exp\left(- \frac{|x|^2}{4}\right)dx$ is the Gaussian measure. 
Then, for some dimensional $R_\circ >0$, we have 
\[
\int_{B_{R_\circ}} f^2   dm  \ge \frac{1}{2}.
\]
\end{lemma}

\begin{proof}
Given $\lambda$ large (to be fixed later), choose $R_\circ$ so that 
\[
 \lambda \int_{\R^n \setminus B_{R_\circ}} dm \le \frac{1}{4}.
\]
By assumption 
\[
\int_{B_{R_\circ}} f^2 dm + \int_{\R^n \setminus B_{R_\circ}} f^2 \chi_{\{f^2\le \lambda\}} dm + \int_{\R^n \setminus B_{R_\circ}} f^2 \chi_{\{f^2>\lambda\}} dm =1.
\]
Suppose now, by contradiction, that $\int_{B_{R_\circ}} f^2 dm < 1/2$. 
Since 
\[
\int_{\R^n \setminus B_{R_\circ}} f^2 \chi_{\{f^2\le \lambda\}} dm \le \lambda \int_{\R^n \setminus B_{R_\circ}} dm \le \frac{1}{4},
\]
we have 
\begin{equation}
\label{eq:int f BR}
  \int_{\R^n \setminus B_{R_\circ}} f^2 \chi_{\{f^2>\lambda\}} dm \ge \frac{1}{4}.
\end{equation}
Recall the Gaussian log-Sobolev inequality:
\[
\int_{\R^n} F^2 \log F^2  dm \le \int_{\R^n} |\nabla F|^2   dm +\biggl(\int_{\R^n} F^2 dm\biggr)\, \log\biggl(\int_{\R^n} F^2 dm\biggr).
\]
To reach a contradiction, we apply this inequality to $F:=|f|  \chi_{\{|f|>e^{1/2}\}}$.
Indeed, since $|\nabla F|\le |\nabla f|$, $\log(F^2) \geq 1$, $F^2\leq e+f^2$, and $F=f$ inside $\{f^2\geq \lambda\}$, we get
\begin{align*}
\log\lambda   \int_{\R^n \setminus B_{R_\circ}} f^2 \chi_{\{f^2>\lambda\}} dm&
\leq \int_{\R^n \setminus B_{R_\circ}} F^2\log(F^2) \chi_{\{f^2>\lambda\}} dm 
\leq \int_{\R^n} F^2 \log F^2  dm \\
&\leq \int_{\R^n} |\nabla F|^2   dm+\biggl(\int_{\R^n} F^2 dm\biggr)\, \log\biggl(\int_{\R^n} F^2 dm\biggr)\\
&\leq \int_{\R^n} |\nabla f|^2dm +\biggl(e+\int_{\R^n} f^2 dm\biggr)\, \log\biggl(e+\int_{\R^n} f^2 dm\biggr)\\
&\leq 4+ (e+1)\log(e+1).
\end{align*}
Combining this bound with \eqref{eq:int f BR}, we get
$$
\log\lambda \leq 16+ 4(e+1)\log(e+1),
$$
a contradiction if $\lambda$ is chosen large enough.
\end{proof}

We can now prove Lemma \ref{lemvyg8g276187}.

\begin{proof}[Proof of Lemma \ref{lemvyg8g276187}]
We divide the proof into three steps.

\smallskip

\noindent $\bullet$ {\em Step 1}. We first show that 
\begin{equation}\label{haihaiht7yug7}
\frac{1}{C}H(r, w\cutoff)^{1/2}  \le \| w_{r}\|_{L^2(\CC_{2R_\circ})},
\end{equation}
where $R_\circ$ is the dimensional  constant from Lemma \ref{lem877r6edfcv}.

Let $\tilde w_r := \frac{(w\cutoff)_r}{H(r, w\cutoff)^{1/2}}$. 
Since $\phi(r, w\cutoff)\le 4$ for $r$ small, we have 
\[
\int_{\{t = -1 \}} (\tilde w_r )^2 G = 1 \quad \mbox{ and }\quad \int_{\{t = -1 \}} |\nabla \tilde w_r |^2 G \le  4.
\]
Then, using Lemma \ref{lem877r6edfcv} we obtain 
\[
\int_{\{|x|\le R_\circ, \ t=-1\}}  (\tilde w_r )^2 G \ge \frac 1 2 .
\]
Replacing $r$ by  $\theta^{1/2} r$ with $\theta \in (1,4)$, by scaling we obtain 
\[
\int_{\{|x|\le 2R_\circ,\ t=-\theta\}}  (\tilde w_r )^2 G \ge \frac 1 2.
\]
Hence, after integrating with respect to $\theta \in (1,4)$, we find 
\[
\int_{\{|x|\le 2R_\circ,\  -4\le t\le -1\}}  (\tilde w_r )^2 G \ge \frac{1}{C}.
\]
Since $\{|x|\le 2R_\circ,\  -4\le t\le -1\} \subset \CC_{2R_\circ}$, \eqref{haihaiht7yug7} follows.

\smallskip

\noindent $\bullet$ {\em Step 2}. We now prove that, for any given $M>1$, there exists a constant $C_M \geq 1,$ depending only on $n$ and $M$, so that the following holds for all in $r\in(0,\frac14)$:
\begin{equation}\label{haihaiht7yug22}
\| w_{r}\|_{L^2(\CC_{1})} \ge C_M H(r, w\cutoff)^{1/2}  \qquad \Rightarrow \qquad \| w_{2r}\|_{L^2(\CC_{1})} \ge M \| w_{r}\|_{L^2(\CC_{1})}.
\end{equation}

To prove this, note that the $L^\infty$ estimate from Corollary \ref{lem:H1spat} gives
\[
\|w_r\|_{L^\infty(\CC_1)} \le  C_\circ \|w_r\|_{L^2(\CC_2)}, 
\]
where $C_\circ$ is dimensional. 
Therefore, whenever $\| w_{2r}\|_{L^2(\CC_{1})} \le M \| w_{r}\|_{L^2(\CC_{1})}$, we have 
\[
 \|w_r\|_{L^\infty(\CC_1)} \le  C_\circ \|w_r\|_{L^2(\CC_2)} = 2^{n+2}C_{\circ}\|w_{2r}\|_{L^2(\CC_1)} \le 2^{n+2}C_{\circ}M  \|w_r\|_{L^2(\CC_1)}. 
\]
In particular, for any $\tau >0$ small, we have 
\[
\int_{\{x\in B_1, -1<t< -\tau \}} w_r^2   \geq    \int_{\{x\in B_1, -1<t<0 \}} w_r^2    - \tau  \|w_r\|^2_{L^\infty(\CC_1)}   \ge  (1  -\tau 2^{2n+4}C_{\circ}^2M^2)\| w_{r}\|^2_{L^2(\CC_{1})}.
\]
Choosing $\tau := (2^{2n+5}C_{\circ}^2M^2)^{-1}$ and using that $G \geq c_\tau>0$ inside $\{x\in B_1, -1<t< -\tau \}$, we get \[
\frac{1}{2}\| w_{r}\|^2_{L^2(\CC_{1})}\le  \int_{\{x\in B_1, -1<t< -\tau \}} w_r^2   \le C_1  \int_{\{x\in B_1, -1<t< -\tau \}} w_r^2 G \le C_2 \int_{\tau}^1 H(\theta^2r,w\cutoff) d\theta.
\]
Since the last term can be bounded by $C_3H(r,w\cutoff)$ (thanks to Corollary \ref{corqwrtyu}), this proves the contrapositive of \eqref{haihaiht7yug22} with $C_M = (2C_3)^{1/2}$.

\smallskip

\noindent $\bullet$ {\em Step 3}. We now conclude the proof of the lemma combining Steps 1 and 2. 

First, \eqref{haihaiht7yug7} rescaled yields 
\[
\frac{1}{C}H(c_\circ r, w\cutoff)^{1/2}  \le \| w_{r}\|_{L^2(\CC_1)},
\]
where $c_\circ= \frac{1}{2R_\circ}$, that combined with
$H(c_\circ r, w\cutoff)^{1/2} \ge \frac{1}{C}H(r, w\cutoff)^{1/2}$  (see Corollary \ref{corqwrtyu}) gives
$$
\frac{1}{C}H(r, w\cutoff)^{1/2}  \le \| w_{r}\|_{L^2(\CC_1)}.
$$
This proves the first inequality in the statement.

To prove the second inequality, assume by contradiction that 
$\| w_{r}\|_{L^2(\CC_1)} \ge C_M H(r, w\cutoff)^{1/2}$, with $C_M\geq 1$ as in \eqref{haihaiht7yug22} and $M$ large enough to be chosen. Then, by \eqref{haihaiht7yug22},
\[
\| w_{2r}\|_{L^2(\CC_{1})} \ge M \| w_{r}\|_{L^2(\CC_{1})} \ge  MC_M H(r, w\cutoff)^{1/2} .
\]
Let $N>1$ be a large constant to be fixed.
Since $H(r, w\cutoff)^{1/2} \ge C_4^{-1} H(2r, w\cutoff)^{1/2}$ (by Corollary \ref{corqwrtyu}), choosing $M = N C_4$ we obtain 
\[
\| w_{2r}\|_{L^2(\CC_{1})} \ge N C_M H(2r, w\cutoff)^{1/2} .
\]
This allows us to apply again \eqref{haihaiht7yug22} (with the same $M$) and with $r$ replaced by $2r$, and  we obtain
\[
\| w_{4r}\|_{L^2(\CC_{1})} \ge M \| w_{2r}\|_{L^2(\CC_{1})} 
\geq M N C_M H(2r, w\cutoff)^{1/2} \geq M N C_M C_4^{-1} H(4r, w\cutoff)^{1/2} = N^2 C_M H(4r, w\cutoff)^{1/2} .
\]
Iterating this argument $\ell$ times, this yields
\[
\| w_{2^\ell r}\|_{L^2(\CC_{1})} \ge N^\ell C_M H(2^{\ell}r, w\cutoff)^{1/2} 
\]
Choosing $\ell$ such that $\frac 1 8 \leq 2^{\ell}r \le \frac14$, and using that at scales of order $1$ both quantities $\| w_{2^\ell r}\|_{L^2(\CC_{1})}$ and $H(2^{\ell}r, w\cutoff)^{1/2} $ are comparable to $1$, we obtain
\[
C \ge   N^\ell C_M\geq N^\ell,
\]
a contradiction if $N$ is chosen large enough.
\end{proof}

The following Lipschitz estimate for the rescaled difference $u-p_2$ will be useful in the sequel. We recall that $w_r$  has been defined in \eqref{defwrP},
while $\Sigma_{n-1}$ is defined as in \eqref{ahoiah0P}.

\begin{lemma}\label{lem:E2B1P}
Let $u:B_1\times(-1,1) \to [0,\infty)$ be a bounded solution of  \eqref{eq:UPAR1},
and $(0,0)\in \Sigma_{n-1}$. Then, given $R\ge1$, for all $r\in \left(0,\frac{1}{10R}\right)$ and $\boldsymbol e\in \mathbb S^{n-1}\cap\{p_2=0\}$,  we have 
\[
 \sup_{\CC_1} \big(  (\partial_{\boldsymbol e \boldsymbol e} w_r )_- +  |\partial_t w_r| + |\nabla w_r| \big)  \le  C   \|w_r\|_{L^2(\CC_2)},
\]
for some constant $C$ depending only on $n$, $\|u(\,\cdot\,, 0)\|_{L^\infty(B_1)}$, and $R$.
\end{lemma}

\begin{proof}
Given a function $f: \R^n \to \R$, a vector $\boldsymbol e \in \mathbb S^{n-1}$, and $h\in (0,1)$, let
\[\delta^2_{\boldsymbol e,h} f:= \frac{f(\,\cdot\,+h\boldsymbol e, \,\cdot\,)+f(\,\cdot\,-h\boldsymbol e, \,\cdot\,)-2f}{h^2}\qquad \mbox{and}  \qquad  \delta_{t,h} f  :
 =  \frac{f\big(\,\cdot\, , \cdot +h\big) - f }{h} .
 \]
Note that, for $\boldsymbol e\in  \{p=0\}\cap \mathbb S^{n-1}$, we have $\delta^2_{\boldsymbol e,h} p\equiv 0$. Thus, since $\heatop  u =1$ outside of  $\{u=0\}$ and $\heatop u\le 1$ in $B_1\times (-1,0)$ we have
\[
\heatop \big(\delta^2_{\boldsymbol e,h} w \big)
 =  \frac{\heatop u\big(\,\cdot\,+h\boldsymbol e, \,\cdot \,\big)+\heatop u\big(\,\cdot\,-h\boldsymbol e, \,\cdot \,\big) - 2\heatop u }{h^2}
\le 0   \qquad \mbox{ in } \{u>0\}.
\]
Similarly,
\[
\heatop \big(\delta_{t,h} w \big)
\le 0  \qquad\mbox{and} \qquad \heatop \big(\delta_{t, -h} w \big)
\le 0  \qquad \mbox{ in } \{u>0\}.
\]
On the other hand, since $u\ge 0$, we have
\[
\delta^2_{\boldsymbol e,h} w = \delta^2_{\boldsymbol e,h}  u(\,\cdot\,) \ge 0  \quad \mbox{and} \quad \delta_{t,\pm h} w = \delta_{t,\pm h}  u(\,\cdot\,) \ge 0  \quad \mbox{ inside } \{u=0\}.
\]
As a consequence, the negative part of the second order incremental quotient $(\delta^2_{\boldsymbol e,h} w_r)_-$ is subcaloric, and so is its limit $(\partial_{\boldsymbol e\boldsymbol e}^2 w_r)_-$
(recall that $u\in C^{1,1}_{x}\cap C^{0,1}_{t}$ ,  and thus $(\delta^2_{\boldsymbol e,h}  w_r)_- \to( \partial_{\boldsymbol e\boldsymbol e}^2 w_r)_-$ and 
a.e. as $h \to 0$). Similarly,  both $(\delta_t  w_r)_-$ and $(-\delta_t  w_r)_-$ are subcaloric, and thus so is $|\partial_t w_r|$.

Therefore, by weak Harnack inequality (see Lemma \ref{buauav}), for any $\ep>0$ we have
\[
\sup_{\CC_{4/3}} \big(  |\partial_{t}  w_r|  + (\partial_{\boldsymbol e\boldsymbol e}^2w_r)_- \big)
\leq C_n \biggl(\int_{\CC_{5/4}} ( |\partial_{t}  w_r|  + |\partial_{\boldsymbol e\boldsymbol e} w_r|)^\ep  \biggr)^{1/\ep}.
\]
Also, by standard interpolation inequalities, the $L^\ep$ norm with $\ep<1$ can be controlled by the weak $L^1$ norm, namely
$$
\biggl(\int_{\CC_{5/4}} (|\partial_{t} w_r|  + |\partial_{\boldsymbol e\boldsymbol e}w_r|)^\ep \biggr)^{1/\ep} \leq C(n,R)\,\sup_{\theta >0} \,\theta \bigl|\bigl\{(|\partial_{t}w_r|  + |\partial_{\boldsymbol e\boldsymbol e}w_r|) >\theta\bigr\}\cap \CC_{5/4}\bigr|.
$$
Furthermore, by the parabolic Calder\'on-Zygmund theory (see Theorem \ref{CZpar}), the right hand side above is controlled by
\[
\|\heatop  w_r\|_{L^1(\CC_{5/3})}+
\|w_r\|_{L^1(\CC_{5/3})}.
\]
In addition, since $ \heatop w_r \leq 0$ in $\CC_{5/3}$, the norm $\|\heatop w_r\|_{L^1(\CC_{5/3})}$ is controlled by the $L^1$ norm of $ w_r$ inside $\CC_{2}$: indeed, if $\chi$ is a smooth nonnegative cut-off function that is equal to $1$ in $\CC_{5/3}$ and vanishes outside $\CC_{2}$, then integration by parts gives
\begin{equation}
\label{eq:control Delta w}
\|\heatop  w_r\|_{L^1(\CC_{5/3} )}\leq -\int_{\CC_{2}}\chi\,\heatop  w_r =
-\int_{\CC_2} (\Delta +\partial_t) \chi\,  w_r
\leq C_n\int_{\CC_2}| w_r| \le\| w_r\|_{L^2{(\CC_2)}}.
\end{equation}
This proves that
\begin{equation}\label{abuigwoubw}
\sup_{\CC_{4/3}} \big(  |\partial_{t}  w_r|  + (\partial_{\boldsymbol e\boldsymbol e}^2w_r)_- \big)  \le C_n \| w_r\|_{L^2{(\CC_2)}}.
\end{equation}
Finally,   since $\{p_2=0\}$ is $(n-1)$-dimensional (recall $(0,0)\in \Sigma_{n-1}$),  as a consequence of \eqref{abuigwoubw} and  $(\Delta-\partial_t)  w_r  \le 0 $ we deduce that 
\[
\partial_{\boldsymbol e'\boldsymbol e'} w_r \le C \| w_r\|_{L^2{(\CC_2)}}  \quad \mbox{in } \CC_{4/3},\quad  \mbox{ where } \boldsymbol e'\in \{p=0\}^\perp \mbox{ with } |\boldsymbol e'|=1.
\]
The above semiconcavity estimate,
combined with the semiconvexity bound in \eqref{abuigwoubw}, implies the desired uniform bound on $\|\nabla w_r\|_{L^\infty(\CC_1)}$.
\end{proof}

As a consequence of the previous lemma, we get the following result
(recall that, as a consequence of Lemma \ref{lem:freq2and3} and Corollary \ref{corqwrtyu}, we know that $\phi(0^+,w\cutoff) \in [2,3]$).

\begin{proposition}\label{prop:E2B3P}
Let  $u:B_1\times(-1,1) \to [0,\infty)$ be a bounded solution of  \eqref{eq:UPAR1}, $(0,0)\in \Sigma_{n-1}$, and set $w:=u-p_2$ and 
$\lambda^{2nd}:=\phi(0^+,w\cutoff) \in [2,3]$.
 Then
\begin{equation}\label{eqbhuab1}
\{u(\cdot\,, t) =0\} \cap B_r \subset \big\{ x \, : \, {\rm dist}(x, \{p_2=0\}) \le C r^{\lambda^{2nd}-1}\big\}
\end{equation}
for all $r\in (0,1/2)$ and $t\ge -r^2$.
In addition, the constant $C$  depends only on $n$ and $\|u(\,\cdot\, ,0)\|_{L^\infty(B_1)}$.
\end{proposition}

\begin{proof}
As a consequence of Lemmas  \ref{lem:HP}  and  \ref{lem:E2B1P} we obtain
\[
\| \nabla u-\nabla p_2\|_{L^\infty(B_r\times \{-r\}))} \le  Cr^{\lambda^{2nd}-1} .
\]
Thus, on $\{u=0\}\cap \bigl(B_r\times\{-r^2\}\bigr)$ we have $|\nabla p_2(x)| = \big|{\rm dist}(x, \{p_2=0\})\big| \le Cr^{\lambda^{2nd}-1}$. Since $\{u=0\}$ shrinks with time, we obtain \eqref{eqbhuab1}.
\end{proof}

Using the previous results, we can now prove the following:

\begin{proposition}
\label{prop:E2B2P}
Let  $u:B_1\times(-1,1) \to [0,\infty)$ be a solution of  \eqref{eq:UPAR1}, and $(0,0)$ be a singular point.
Let  $w:=u-p_2$, $m: = {\rm dim }(\{p_2=0\}) \in \{0,1,2,\dots n-1\}$,  $\lambda^{2nd}:=\phi(0^+,w\cutoff)$, and define
\begin{equation}\label{tildew11}
\tilde w_{r} := \frac{(u-p_2)_r}{H\big(r, (u-p_2)\cutoff\big)}.
\end{equation}
Then, for every sequence $r_k\downarrow 0$ there is a subsequence $r_{k_\ell}$ such that
\[\tilde  w_{r_{k_\ell}} \rightarrow q \qquad \mbox{and} \qquad \nabla \tilde  w_{r_{k_\ell}} \rightharpoonup \nabla q\qquad \mbox{in } L^2_{\rm loc}(\R^n \times (-\infty, 0])\]
as $\ell \to \infty$, where $q\not\equiv 0$ is a $\lambda^{2nd}$-homogeneous function. Moreover, for any $e$ unit vector tangent to $\{p_2=0\}$, the following  ``growth estimates'' hold:
\begin{equation}\label{hihaiogwt}
\sup_{\CC_R}\big( |W|  + R^2 |\partial_tW|\big)  +\bigg(\ave_{\CC_R}  R^2|\nabla W|^2  +R^4|\nabla \partial_eW|^2 + R^6 |\nabla \partial_t W|^2  \bigg)^{\frac 1 2} \le  C_\delta R^{\lambda^{2nd}+\delta}
\end{equation}
with $W=\tilde  w_{r_{k_\ell}}$ (for all $R$ satisfying  $1\le R\ll r_{k_\ell}^{-1}$) and for $W=q$ (for all $R\ge1$).

In addition, we have:
\begin{enumerate}
\item[(a)]
If $0\le m\le n-2$ then $\lambda^{2nd}=2$, and ---in some appropriate coordinates---
\begin{equation}\label{howisD2q}
\qquad p_2(x)= \frac{1}{2} \sum_{i=m+1}^n \mu_i x_i^2 \qquad \mbox{and}\qquad  q(x)= At+\nu \sum_{i=m+1}^n  x_i^2  - \sum_{j=1}^m \nu_j x_j^2 \,,
\end{equation}
where $\mu_{i}>0$, $\nu$ and $A$ are nonnegative, and $A-2(n-m)\nu + 2\sum_{j=1}^m \nu_j=0$.

\item[(b)] If $m=n-1$ then $\tilde  w_{r_{k_\ell}} \rightarrow q$ in $C^0_{\rm loc}(\R^n \times(-\infty,0])$, $\lambda^{2nd} \in [2+\alpha_\circ,3]$ for some dimensional constant $\alpha_\circ >0$,  and $q$ solves the parabolic thin obstacle problem
\begin{equation}\label{PTOP}
\begin{cases}
\heatop q\le 0\quad \text{and} \quad q\heatop q=0 \quad & \mbox{in }\R^n\times(-\infty,0)
\\
\heatop q=0  &\mbox{in }\R^n\times(-\infty,0) \setminus  \{p_2=0\}
\\
\,q\ge 0 &\mbox{on } \{p_2=0\}
\\
\partial_t q\ge 0 &\mbox{in }\,\R^n\times(-\infty,0).
\end{cases}
\end{equation}
\end{enumerate}
\end{proposition}

\begin{proof}
Given $r_k \downarrow 0$, by the $H^1$ space-time estimates  of Corollary  \ref{lem:H1spat} we obtain (up to subsequence)
\[
\tilde w_{r_k} \rightarrow  q \quad \mbox{and } \quad (\nabla, \partial_t)\tilde w_{r_k} \rightharpoonup  (\nabla, \partial_t)q\quad \mbox{in } L^2_{loc} \big(\R^n\times(-\infty, 0]\big), 
\]
for some  $q\in L^2_{loc}((-\infty, 0], H^1_{loc} (\R^n))$.
Moreover, since $\phi(r_k, w\cutoff)\le  C$ we have $\int_{\{t = -1 \}} \big|\nabla (\widetilde{w\cutoff})_{r_k}\big|^2 G \le C$, from which it follows that
$\int_{\{t = -1 \}} |\nabla q|^2 G \le C$.

Furthermore, by Lemma \ref{lem:HP}(b) applied with $\gamma=4$,
Lemma \ref{lemvyg8g276187} combined with Corollary  \ref{lem:H1spat} implies the validity of the ``growth estimates''
\eqref{hihaiogwt} for $W=\tilde  w_{r_{k}}$, for all $R$ satisfying  $1\le R\ll r_{k}^{-1}$. Taking the limit as $k \to \infty$, since all the seminorms in the estimates are lower semicontinuous, we obtain that \eqref{hihaiogwt} holds also $W=q$ (for all $R\ge1$).
It is also easy to see, using the monotonicity of the  frequency  in Proposition \ref{prop:EMF1P}, that $q$ is  $\lambda^{2nd}$-homogeneous.

Let us show that, in addition, $q$ satisfies
 \begin{equation}\label{ghwiowhiohwio}
\int_{\{t=-1\}} (p_2-p)q \,  G \, dx\ge 0\qquad \textrm{for all}\quad p\in \mathcal P.
\end{equation} 
Indeed, for any fixed $p\in \mathcal P$, Lemma \ref{lem:HP}(a) gives
\[
\frac{1}{r^4} (H\big(r, (u-p)\cutoff\big) + r^8\big) + Ce^{-c/r}  \ge \lim_{r\downarrow 0}\frac{1}{r^4} (H\big(r, (u-p)\cutoff\big)\big),
\]
and therefore
\[
\int_{\{t=-1\}} \big(r^{-2}u_r-p)\cutoff_r\big)^2 G + Cr^{4} \ge \int_{\{t=-1\}} (p_2-p)^2 G.
\]
This allows us to proceed exactly as in the proof of \cite[Lemma 2.11]{AlessioJoaquim}  (with obvious modifications), and \eqref{ghwiowhiohwio} follows.

Now in order to conclude the proof, we consider two cases.

\vspace{2mm}

(a)  If $m\le n-2$ then $w_{r_k}\in H^1_{\rm loc}$  it is caloric outside of an infinitesimal neighborhood of  $\{p_2=0\}\times(-\infty, 0)$.
Since ${\rm dim}(\{p_2=0\})\le n-2$ which has zero capacity, we deduce that $\heatop q \equiv 0$ in all of $\R^n \times (-\infty, 0)$.
Also, 
by Lemma \ref{lem:freq2and3}(a) we know that $\lambda^{2nd}=2$.
Then, using that $\partial_t q\ge 0$ and \eqref{ghwiowhiohwio}, we easily conclude  that $q$ must be of the form \eqref{howisD2q}.

\vspace{2mm}

(b) If $m=n-1$, then by Lemma \ref{lem:freq2and3}(b) we have $\lambda^{2nd}\in [2,3]$. Moreover, using the Lipschitz estimate in Lemma \ref{lem:E2B1P},  we deduce that (up to subsequence)  $\tilde w_{r_k}\to  q$ locally uniformly in $\R^n \times (-\infty, 0)$. 
Clearly $\partial_tq\ge 0$ (since $\partial _t\tilde w_{r_k}\geq0$). 
Moreover, $q$ is a solution of \eqref{PTOP} because $\heatop w_{r_k} \le 0$ implies $\heatop w_{r_k} \rightharpoonup^* \ \heatop q \le 0$ as measures, and thus $w_{r_k}\heatop w_{r_k}\ge 0$ implies $q\heatop q\ge 0$. 
Also, $\heatop q\le 0$ is supported on $\{p_2=0\}$ where $q\ge 0$ (since on $\{p_2=0\}$ we have $w_{r_k}= u_{r_k} \ge 0$). 


Now, we notice that  there are no $2$-homogeneous solutions $q$ to \eqref{PTOP} and satisfying  \eqref{ghwiowhiohwio} (see \cite[Proposition 2.10]{AlessioJoaquim} for a very similar argument), and therefore it must be $\lambda^{2nd}>2$.
Finally, the fact that $\lambda^{2nd}\geq 2+\alpha_\circ$ can be proved either by compactness, or using \cite[Proposition 9]{Wenhui} to deduce that there are no $\lambda$-homogeneous  solutions to  \eqref{PTOP} for $\lambda\in(2,2+\alpha_\circ)$.
\end{proof}

\section{The set $\Sigma_{n-1}^{<3}$} 
\label{sec-7}

Let us define the sets
\begin{equation}
\label{eq:Sigma n-1 <3}
\Sigma_{n-1}^{<3} := \big\{ (x_\circ, t_\circ)\in  \Sigma_{n-1} \  : \  \phi\big(0^+, u(x_\circ+\,\cdot\,, t_\circ + \,\cdot\,)-p_{2,x_\circ,t_\circ}\big) < 3 \big\}
\end{equation}
and
\begin{equation}
\label{eq:Sigma n-1 =3}
\Sigma_{n-1}^{= 3} := \Sigma_{n-1}\setminus \Sigma_{n-1}^{<3}.
\end{equation}

In this section we investigate the structure of the possible blow-ups $q$ at points of $\Sigma_{n-1}^{<3}$. 
We start proving the following .

\begin{lemma}\label{haihoiaha}
Let $(0,0)\in \Sigma_{n-1}^{<3}$ with $p_2 = \frac 1 2 x_n^2$. Let $\tilde w_r$ be defined by \eqref{tildew11} and suppose that  $\tilde w_{r_k} \rightharpoonup q$ in $W^{1,2}_{\rm loc}(\R^n \times (-\infty, 0])$, for some sequence  $r_k\downarrow 0$. 
Then:
\begin{itemize}
\item[(i)]  $q$ is $\lambda^{2nd}$-homogeneous, where $\lambda^{2nd}\in(2,3)$, and $\int_{\{t=-1\}} q^2 G  =1$;
\item[(ii)] $q$ satisfies the growth estimates in \eqref{hihaiogwt};
\item[(iii)]  $q$ is a solution of  \eqref{PTOP};
\item[(iv)]  $\partial_t q \not\equiv 0$;
\item[(v)] $\partial_{ee} q\ge 0$ in $\R^n$ for all $e\in \{x_n =0\}$ and $\partial_{tt} q\ge 0$ in $\R^n$. 
\end{itemize}
\end{lemma}

In  the proof of Lemma \ref{haihoiaha}(iv) we will use the following result, which gives a classification of all convex solutions to the Signorini problem.
(Notice that, previously,  this was only known for homogeneities $\lambda\leq 2$.)

\begin{lemma}\label{ahioahgoigha} 
Let $q=q(x)$ be a $\lambda$-homogeneous solution of the elliptic Signorini problem, namely 
$$
\text{$\Delta q \le 0$ and $q\Delta q = 0$ in $\R^n$,\,\,\,\, $\Delta q =0$  on  $\{x_n\neq 0\}$,\,\,\,\,  $q\ge 0$ on $\{x_n=0\}$.}
$$
Assume in addition that $q$ is convex with respect to all directions parallel to $\{x_n=0\}$. Then, either $\lambda =3/2$, or $q$ is a harmonic polynomial and $\lambda$ is an integer.
\end{lemma}

\begin{proof}
Since $q$  is homogeneous and convex with respect to the first $n-1$ variables, this implies that the set $K:=\{q=0\}\subset \{x_n=0\}$ is convex. 
Now, if $K$ has empty interior as a subset of $\{x_n=0\}$, then it will be contained in a $(n-2)$-dimensional subspace, so it follows form a standard capacity argument that $q$ must be harmonic in all of $\R^n$.
Hence, $q$ is a homogenous harmonic polynomial and $\lambda$ will be an even integer. 
On the other hand, if $K=\{x_n=0\}$, then $q$ would be an odd harmonic polynomial and $\lambda$ will be an odd integer. 

Consider now the case when $K\neq\{x_n=0\}$ has non-empty interior inside $\{x_n=0\}$.
Up to replacing $q(x',x_n)$ by $q(x',x_n)+q(x',-x_n)$, we may assume that $q$ is even.
Pick a direction  $e\in \mathbb S^{n-1}\cap\{x_n=0\}$ such that $-e \in {\rm int }(K)$.
Then, for any point $z'\in\R^{n-1}$ we have $(z',0)-et \in K$ for $t>0$ large enough. 
We claim that this implies that $\partial_e q\ge 0$  in all $\R^n$. 
Indeed, if it was $ \partial_e q(z)<0$ at some point $z=(z',z_n)$ with $z_n>0$ then, by convexity of $q$ in the direction $e= (e',0)$, we would have 
\[
q(z'-se',z_n) \ge q(z',z_n) - \partial_e  q(z',z_n)s \quad \mbox{for all }s>0.
\]
Hence, $q(z'-s_\circ e',z_n) >0$ for $s= s_\circ$ sufficiently large.
However, since $(z-s_\circ e',0) \in K$, then (recall that $q$ is even)
\begin{equation}\label{heioahoih}
q(z-s_\circ e',0) =0 \quad \mbox{and}\quad \partial_n q(z-s_\circ e',0^+) \le0.
\end{equation}
Since $\Delta q\le 0$ and $\partial_{ii} q\ge 0$ for all $i\le n-1$ we have that $q$ is concave in the variable $x_n$, a contradiction.

Once we have shown that  $\partial_e q\ge 0$ in all $\R^n$ for all $e \in -{\rm int}(K)$, we can use the argument from \cite{CRS,RS-fully} in order to classify blow-ups.
Indeed, we look at two different directions in the monotonicity cone, $e$ and $e'$, and observe that 
$\partial_e q|_{\mathbb S^{n-1}}$ and $\partial_{e' } q |_{\mathbb S^{n-1}}$ are both positive harmonic functions in the same cone, and they both vanishing on the boundary.
Therefore, by uniqueness of positive harmonic functions in cones, one must be a multiple of the other. 
This argument, applied to $n-1$ independent directions of the monotonicity cone, leads to the fact that $q$ depends only on two euclidean variables, namely, $x_n$ and a linear combination of $x_1,\dots, x_{n-1}$. 
One then reduces the classification problem to the well-known situation of $\R^2$, in which all homogeneous solutions are explicit and the only convex ones have homogeneity $3/2$. This concludes the proof.
\end{proof}

We can now prove Lemma \ref{haihoiaha}.

\begin{proof}[Proof of Lemma \ref{haihoiaha}]
Note that (i)-(iii) follow from  Proposition \ref{prop:E2B2P}. 

To show (iv), we recall that, by Proposition \ref{proputD2u-}, we have
\[
(D^2  u)_-  \le C\,\partial_t u .
\]
Hence, if $e\in \{x_n =0\}$, 
\[
\big(\partial_{ee}  (u-p_2)\big)_-  = (\partial_{ee}  u)_- \le C\, \partial_{t}  u =C\, \partial_{t}  (u-p_2) ,
\]
and, as a consequence, 
\begin{equation}\label{quiaioghaoibhe}
(\partial_{ee}  \tilde w_r )_-  \le C\,\partial_t   \tilde w_r, \quad \mbox{ and thus } (\partial_{ee} q )_-  \le C\,\partial_t  q .
\end{equation}
Assume now by contradiction that $\partial_{t}  q \equiv 0$.
Then $q=q(x)$ is convex with respect to all directions in $\{x_n=0\}$.
However, using Lemma \ref{ahioahgoigha},  we obtain a contradiction with  the fact that, by assumption, $q$ is homogeneous of degree $\lambda\in (2,3)$.

We now prove (v). By differentiating \eqref{PTOP} with respect to $t$, and using that $q$ is homogeneous, it follows that $\varphi_1:=\partial_t q(x,-1)$ solves
\begin{equation}\label{quiaioghaoibhe2}
\mathcal L_{OU}\varphi_1 = \lambda_1 \varphi_1 \quad \mbox{in }\R^n\setminus \big(\{x_n=0\}\cap \{ q =0\}\big), \qquad \varphi_1 = 0 \quad \mbox{on }\{x_n=0\}\cap \{ q =0\},
\end{equation}
where $\lambda_1 : =  \frac{\lambda^{2nd}-2}{2}$, and $\mathcal L_{OU}$  is as in  \eqref{auiwgugwar}.
Let $dm = G(x,-1)dx$ be the Gaussian measure.
Since $\varphi_1$ belongs to $H^1(\R^n, dm)$  (as a consequence of the estimates in \eqref{hihaiogwt}), then $\varphi_1$ must be the first eigenfunction of $\mathcal L_{OU}$, and $\lambda_1$ the first eigenvalue.

Also,  as shown in the beginning of the proof of  Lemma \ref{lem:E2B1P},  the function $\big(\partial_{ee}  (u-p_2)\big)_- \ge 0$ is subcaloric. 
Therefore, $f_r:= (\partial_{ee}  \tilde w_r )_-$  satisfies
$f_r\,\heatop f_r \ge 0$, and we obtain (cp. \eqref{anioqioqbioboi2})
$$
 R^2 \ave_{\CC_R} |\nabla f_r|^2 \le  C_n\ave_{\CC_{2R}}  f_r^2.
 $$
Since $f_r\le C\,\partial_t \tilde w_r$ (by \eqref{quiaioghaoibhe}), the estimates  in \eqref{hihaiogwt} yield
\begin{equation}\label{wiowhoihw}
\bigg(R^2\ave_{\CC_R} |\nabla(\partial_{ee}  \tilde w_r )_-|^2 \bigg)^{\frac 1 2}\le C\bigg(\ave_{\CC_{2R}} |\partial_t \tilde w_r|^2 \bigg)^{\frac 1 2} \le C_\delta R^{\lambda^{2nd}+\delta-2}. 
\end{equation}
Set $\psi: =\big(\partial_{ee}  q(x,-1)\big)_-$. Then, $\psi$ is a $(\lambda^{2nd}-2)$-homogeneous 
nonnegative  subsolution of  \eqref{quiaioghaoibhe2}. Also, since 
 $\tilde w_{r_k} \to  q$ in $L^2_{\rm loc}$, it follows by  \eqref{wiowhoihw} that $\psi$  belongs to $H^1(\R^n, dm)$.
By uniqueness of the first eigenfunction, it follows that 
$\psi$ is a nonnegative multiple of $\varphi_1$. 

We claim that $\psi \equiv 0$. Indeed, since
$\varphi_1$ is the first eigenfuction, it is positive everywhere (except on the boundary).
Hence, if by contradiction $\psi$ is not zero everywhere, then we get $(\partial_{ee}   q(x,-1) \big)_- >0$ in $\R^n\setminus \big(\{x_n=0\}\cap \{ q(x,-1) =0\}\big)$, and in particular $\partial_{ee}  q<0$ on $\{x_n=0\}\setminus \{ q(x,-1) =0\}$. However, since $q =|\nabla q|=0$  on $\{x_n=0\} \cap\partial  \{ q(x,-1) =0\} $, this contradicts the fact that $ q\ge 0$ on $\{x_n=0\}$. Hence $\psi\equiv 0$, or equivalently 
$\partial_{ee} q\ge 0$, as desired.

Finally, to show that $\partial_{tt}q\ge 0$, we recall that, thanks to  Proposition \ref{prop.semtime}, we have 
\[
\partial_{tt} (u-p_2) = \partial_{tt}  u \ge -C.
\]
After scaling, we obtain
\[
\partial_{tt} \tilde w_r \ge -C\frac{r^4}{H\big(r, (u-p_2)\cutoff\big)^{1/2}} = o(1)\qquad \mbox{as }r\downarrow 0,
\]
and letting $r\to0$ the result follows.
\end{proof}

We can now prove a ``dimension reduction'' lemma for the set $\Sigma_{n-1}^{<3}$. 
We say that a set of points $\{y_i\}_{i \in I}$ has rank at most $m$  if, for any finite subset $\{y_{i_j}\}_{1\leq j \leq N}$, it holds ${\rm span}(y_{i_1},\ldots,y_{i_N})\leq m$.

\begin{lemma}\label{lem627ygfryu4f8}
Let $(0,0)\in \Sigma_{n-1}^{<3}$ with $p_2 = \frac 1 2 x_n^2$, let $\tilde w_r$ be defined as in \eqref{tildew11}, and suppose that $\tilde w_{r_k} \to q$ in $L^2_{\rm loc}(\R^n \times (-\infty, 0])$, for some sequence  $r_k\downarrow 0$. 
Suppose in addition that $(x_k^{(j)}, t_k^{(j)})$ are sequences of singular points with $t_k^{(j)}\ge 0$ such that $y_k^{(j)}: = \frac{x_k^{(j)}}{r_k} \to y_\infty^{(j)}\neq 0$. Then
$y_\infty^{(j)}\in \{x_n=0\}$, and the set $\big\{y_\infty^{(j)}\big\}_j$ has rank at most $n-2$.
\end{lemma}

Before giving the proof of Lemma \ref{lem:EG2B1bisP} we need a couple of auxiliary lemmas.

\begin{lemma}\label{lem:EG2B1bisP}
Let $u: B_1\times(-1,1)\to [0,\infty)$ be a bounded solution of  \eqref{eq:UPAR1}. Then:

{\rm (a)} The singular set  is closed ---more precisely $\Sigma\cap \overline B_\varrho$ is closed for any $\varrho<1$. Moreover,  
\[ \Sigma \cap \overline B_\varrho\ni(x_k,t_k)\to (x_\infty, t_\infty) \qquad \Rightarrow \qquad p_{2,x_k,t_k} \to p_{2,x_\infty, t_\infty}.\]

{\rm (b)} The  function \[\Sigma \ni (x_\circ, t_\circ) \mapsto \phi(0^+, u(x_\circ + \,\cdot\, \,, t_\circ)- p_{2,x_\circ, t_\circ} )\] is upper semi-continuous.
\end{lemma}

\begin{proof}[Proof of Lemma \ref{lem:EG2B1bisP}]

(a) In \cite{C-obst}, Caffarelli proved that the regular set is relatively open inside the free boundary, hence the singular set is closed.
Moreover, it follows from Lemma \ref{lem:HP}(a) that $p_{2,x_k,t_k} \to p_{2,x_\infty, t_\infty}$ whenever $(x_k,t_k)\to (x_\infty, t_\infty)$ (cp. \cite[Theorem 8.1]{Fig18b}).

(b) Fixed $\gamma>3$, this follows from the (almost) monotonicity of the frequency $\phi^{\gamma}$ and Corollary~\ref{corqwrtyu}.
\end{proof}

\begin{lemma}\label{lem:EG2B1bbP}
Let $u: B_1\times(-1,1)\to [0,\infty)$ be a bounded solution of  \eqref{eq:UPAR1},   $(x_k,t_k) \in \Sigma$,  and assume that $(x_k, t_k)\to (0,0)\in \Sigma$ .
Also, assume that $ t_k\ge -C' |x_k|^2$  for some $C'>0$. Then
\[
\textstyle {\rm dist}\left ( \frac{x_k}{|x_k|} , \{p_2=0\}\right ) \to 0 \qquad \mbox{as $k\to \infty$}.
\]
\end{lemma}

\begin{proof}
Let $r_k := |x_k|$, $y_k := x_k/r_k$,  and $s_k :=t_k/r_k^2$ . 
We know that $U_k :=r_k^{-2}u_{r_k}$ converges locally uniformly to $p_2$ as $k\to \infty$. Moreover, since $U_k$ has a singular point at $(s_k,t_k)$ and $U_k$ is increasing in time, it follows that $U_k(y_k, s)=0$ for all $s \le s_k$. In particular, thanks to our assumption, $U_k(y_k,  -C')=0$.
This implies that, for any accumulation point 
$y_\infty$ of $\{y_k\}$, we must have $p_2(y_\infty)=0$. The lemma follows.
\end{proof}

We can now prove Lemma \ref{lem627ygfryu4f8}.

\begin{proof}[Proof of Lemma \ref{lem627ygfryu4f8}]
The fact that $y_\infty^{(j)}\in \{x_n=0\}$ follows from Lemma \ref{lem:EG2B1bbP}. 
Also, since $(x_k^{(j)},t_k^{(j)})$ is a singular point, we have $u(x_k^{(j)},t_k^{(j)})=0$.
Hence, since $t_k\ge 0$ and $u$ is increasing in time, we deduce that  $(u-p_2)(x_k^{(j)}, t) \le 0$ for all $t\le 0$, and therefore
\[
\tilde w_{r_k}(y_k^{(j)}, t)\le 0 \quad \mbox{for all } t\le 0.
\]
Since $\tilde w_{r_k}\to q$ locally uniformly, it follows that $q(y_\infty^{(j)}, t)\le 0$, but since $y_\infty^{(j)}\in \{x_n=0\}$ (where $q$ is nonnegative, see \eqref{PTOP}), it must be 
\[
q(y_\infty^{(j)}, t)\equiv  0 \quad \mbox{for all } t\le 0.
\]
Now, the fact that  $q$  is homogeneous, monotone in $t$, and convex with respect to the first $n-1$ variables (see Lemma~\ref{haihoiaha}),  implies that the cone $K\subset \{x_n=0\}$ generated by $y_{\infty}^{(j)}$ is contained inside $\{q=0\}\cap \{x_n=0\}$. 

Assume now by contradiction that the rank of $\big\{y_\infty^{(j)}\big\}_j$ is $n-1$. Then $K$ has non-empty interior inside $\{x_n=0\}$. 
In particular, if we pick a direction  $e\in \mathbb S^{n-1}\cap\{x_n=0\}$ such that $-e \in {\rm int }(K)$, then for any point $z'\in\R^{n-1}$ we have 
$(z',0)-se \in K$ for $s>0$ large enough. 

Set $\hat q(x): =q(x,-1)$. We claim that 
\[\partial_e \hat q\ge 0  \qquad \textrm{inside}\quad \R^n\]
for any such direction $e$. 
Indeed, suppose by contradiction that $ \partial_e \hat q(z,-1)<0$ at some point $z=(z',z_n)$ with (with no loss of generality) $z_n>0$. 
Then, by convexity of $q$ in the direction $e= (e',0)$, we have 
\[
\hat q(z'-se',z_n) \ge \hat q(z',z_n) - \partial_e \hat q(z',z_n)s.
\]
Hence for all $s>s_\circ$  (with $s_\circ$ large enough depending on $z$) we have 
\begin{equation}\label{haihaiha1}
\hat q(z'-se',z_n)\ge c_\circ s \qquad \mbox{where }c_\circ>0,
\end{equation}  
and
\begin{equation}\label{haihaiha2}
(z'-se',0) \in K\subset \{\hat q=0\}\cap \{x_n=0\}.
\end{equation}
However, since $q$ is convex in the first $n-1$ variables, we have
\begin{equation}\label{haihaiha67t27}
\partial_{nn} q\leq \Delta q \le \partial_t q \le C_\delta R^{\lambda^{2nd}-2+\delta} \quad \mbox{in }B_R\times (-R^2,0).
\end{equation}
Therefore, since $\hat q(z'-se',0)=0$ and $\partial_n \hat q(z'-se',0^+)\ge 0$ (by \eqref{haihaiha2} and \eqref{PTOP}), integrating \eqref{haihaiha67t27} along the segment joining $(z'-se',0)$ and $(z'-se',z_n)$  we obtain
\[
\hat q(z'-se',z_n) \le C_\delta R^{\lambda^{2nd}-2+\delta} \, \frac{|z_n|^2}{2}\qquad \mbox{whenever}  \quad \big|(z'-se',z_n)\big|\le R.
\]
This implies that $\hat q(z'-se',z_n) \le C_\delta s^{\lambda^{2nd}-2+\delta}|z_n|^2$ for  $s$ sufficiently large. However, since $\lambda^{2nd}<3$, choosing $\delta>0$ such that  $\lambda^{2nd}-2+\delta<1$, this contradicts \eqref{haihaiha1}.

Thus, we have shown that $\partial_e \hat q\ge 0$. Then, differentiating \eqref{PTOP} with respect to $e$ and using that $q$ is homogeneous, we deduce that $\psi : =\partial_e \hat q$ is a nonnegative solution of
\[
\mathcal L_{OU}\psi  = \lambda \psi  \quad \mbox{in }\R^n\setminus \big(\{x_n=0\}\cap \{\hat q =0\}\big), \qquad \psi = 0 \quad \mbox{on }\{x_n=0\}\cap \{\hat q =0\}.
\]
Also, since $\partial_e q$ is $(\lambda^{2nd}-1)$-homogeneous, its associated eigenvalue is $\lambda=\frac{\lambda^{2nd}-1}{2}$. 
On the other hand, we showed in the proof of Lemma \ref{haihoiaha}(v) that the positive function $\varphi_1(x) :=\partial_t q(x,-1)$ is the first eigenfunction of this problem and has eigenvalue $\frac{\lambda^{2nd}-2}{2}$.
Since the first eigenfunction is the only one which does not change sign, this provides the desired contradiction.
\end{proof}

We can now prove that $\pi_x\big(\Sigma_{n-1}^{<3}\big)$ has Hausdorff dimension at most $n-2$ (recall \eqref{eq:proj}). Note that this is not standard, since Lemma \ref{lem627ygfryu4f8} only shows that,
around any point $(x_\circ, t_\circ) \in \Sigma_{n-1}^{<3}$, the set $\pi_x(\Sigma_{n-1}^{<3} \cap\{t\ge t_\circ\}\big)\cap B_{r}(x_\circ)$ is contained in a $\ep$-neighborhood of some $(n-2)$-plane. The fact that we can only control points ``in the future'' $\{t\ge t_\circ\}$ forces us to prove some new appropriate GMT result that incorporates this feature.
We state it as an abstract result.

\begin{proposition}\label{prop:GMTfut}
Let $E\subset \R^n\times \R$, and suppose that for all $(x,t)\in E$ and $\ep>0$ there exists $\varrho_{x,t,\ep}>0$ such that, for all  $r\in (0, \varrho_{x,t,\ep})$,
\begin{equation}
  \pi_x\big( E\cap ( \overline{B_{r}} (x)\times [t,\infty]) \big) \subset  (x + L + \overline {B_{r \ep}}) \quad \mbox{for some $L\subset\R^n$ linear space, with ${\rm dim}(L)  = m$}.
\end{equation}
Then
\[
{\rm dim}_{\HH} \big(\pi_x(E)\big) \le m.
\]
\end{proposition}

To prove this proposition, we need the following classical GMT lemma, whose proof can be found for instance in \cite[Lemma 9]{Ros-CIME}.

\begin{lemma}\label{GMTstd}
Given $\alpha>0$ there exists $\ep>0$, depending only on $n$ and $\alpha$, such that the following holds:
Let $F\subset \R^n$, and suppose that  there exists $\varrho>0$ such that, for all $x\in F$ and $r\in (0,\varrho)$,
\[
 F\cap  \overline{B_{r}} (x)   \subset  (x + L + \overline {B_{r \ep}}) \quad \mbox{for some $L\subset\R^n$ linear space, with ${\rm dim}(L)  = m$}.
\]
Then $\HH^{m+\alpha}(F)=0$.
\end{lemma}

To do this, for any arbitrarily small $\delta>0$ we will construct a cover of $\pi_x(E)$ by countably many balls $B_{r_i}(x_i)$ of diameter $\le \delta$ such that $\sum_i  r_i^{m+\alpha}$
can be made arbitrarily small. 
This implies that $\HH^{m+\alpha}(\pi_x(E)) =0$  for all $\alpha>0$ and hence  ${\rm dim}_\HH \big(\pi_x(E)\big) \le m$.

We can now prove Proposition \ref{prop:GMTfut}.

\begin{proof}[Proof of Proposition \ref{prop:GMTfut}]
Up to taking countable unions, we may assume that $E\subset \overline{B_{1/2}}(z)\times [-1,1]$.
Also, since 
\[E = \bigcup_\ell E_\ell, \qquad \textrm{where} \quad E_\ell = \big\{ (x,t) \in E : \rho_{x,t,\ep} > {\textstyle \frac 1 \ell}\big\},\]
it suffices to prove that ${\rm dim}_{\HH} \big(\pi_x(E_\ell)\big) \le m$ for each $\ell$.
To this aim, given $\alpha>0$ arbitrarily small, we shall prove that $\HH^{m+\alpha} \big(\pi_x(E_\ell)\big) =0$.

Fix $\ep>0$ (depending only on $n$ and $\alpha$, to be chosen later). 
By assumption, for all $(x,t)\in E_\ell$ and for all $r\in (0, 1/\ell]$,
\begin{equation}\label{ahihiahiw}
  \pi_x\big( E_\ell \cap ( \overline{B_{r}} (x)\times [t,\infty]) \big) \subset  (x + L_{x,t,r} + \overline {B_{r \ep}}) \quad \mbox{for some $L_{x,t,r}$ linear and $m$-dimensional}.
\end{equation}
Note that \eqref{ahihiahiw}  holds also for all $(\bar x, \bar t)$ in the closure $\overline{E_{\ell}}$ of $E_{\ell}$, provided we define  $L_{\bar x, \bar t,r}$   as an arbitrary limit of the hyperplanes $L_{x,t,\ep}$ as $(x,t)\in E_{\ell}$ converges to $(\bar x, \bar t)$. 

Now, given $(x,t)\in E_{\ell}$ and $r\leq  1/\ell$,
by compactness there exists $(\bar x, \bar t)\in \overline{B_{r}}(x)\times [-1,1] \cap \overline{E_{\ell}} $ such that
$\bar t \leq \pi_t\big(\overline{B_{r}}(x)\times [-1,1] \cap \overline{E_{\ell}} \big)\subset \R$.
Thus, applying \eqref{ahihiahiw} at the point $(\bar x, \bar t)$, we get\[
 \pi_x\big( E_{\ell} \big) \cap B_{r}(x)  \subset  \pi_x\left( \overline{E_{\ell}} \cap ( \overline{B_{r}} (\bar x)\times [\bar t,\infty)) \right) \subset  \bar x + L_{\bar x,\bar t,r} + \overline{B_{r\ep}}.
 \]
Thus, choosing $\ep$ sufficiently small, Lemma \ref{GMTstd} applied with $F=\pi_x(E_\ell)$ implies that $\HH^{m+\alpha}\big(\pi_x(E_\ell)\big) =0$, as desired.
\end{proof}

Finally, as an immediate consequence of Lemma \ref{lem627ygfryu4f8} and Proposition \ref{prop:GMTfut}, we deduce the following:

\begin{proposition} \label{prop-sigma<3}
The set $\Sigma_{n-1}^{< 3}$ satisfies ${\rm dim}_{\mathcal H}\big(\pi_x(\Sigma_{n-1}^{< 3})\big)\leq n-2$.
\end{proposition}

\section{Quadratic cleaning of the singular set and proof of Theorem \ref{thm-Stefan-intro-0}}
\label{sec-8}

In this section we prove  that, for any given $\ep>0$, there exists $C_\ep$ such that, for any singular point $(x_\circ, t_\circ)\in \Sigma$, it holds
\begin{equation}
\label{eq:almost quadr clean}
\Sigma \cap \big\{|x-x_\circ|\le  r, \ t\ge t_\circ+ Cr^{2-\ep} \big\}  = \emptyset.
\end{equation}
As we shall see, this estimate will allow us to prove Theorem \ref{thm-Stefan-intro-0}.

We begin with the following:

\begin{lemma}\label{cleaning1}
Let $u:B_1\times(-1,1) \to [0,\infty)$ be a solution of  \eqref{eq:UPAR1} and $(0,0)$ a singular point.
 Assume  that  $H(r,u-p_2)^{1/2} \le \omega(r)$, where $\omega(r) =o(r^2)$ as $r\downarrow 0$. 
 Suppose in addition that $\boldsymbol e_n$ is an eigenvector $D^2p_{2}$  with maximal eigenvalue, and that there exists $c>0$ such that
 \begin{equation} \label{whiowehohw127}
 \ave_{B_r \cap \{|x_n|\ge\frac{r}{10} \}}  \partial_tu(\cdot ,-r^2) \ge c\, r^{\beta} \quad \mbox{ for all $r\in (0,1)$, for some $\beta \in(0,1]$}. 
\end{equation}
Then there exists $C>0$ such that
\[
\{u=0\}\cap \big( B_{r/2}\times [C\omega(r)r^{-\beta} + r^2,1) \big) = \emptyset \qquad \forall \,r\in (0,1).
\]
\end{lemma}

\begin{proof}
Since  $\heatop(u-p_2) = -\chi_{\{u=0\}}\le 0$, the function $u-p_2$ is supercaloric.
Also, since $\partial_tu \geq 0$, then  $u-p_2$  is nondecreasing in time. Thus, thanks to the bound $H(r,u-p_2)^{1/2} \le \omega(r)$ we get
\begin{equation}\label{avguobajkf0}
u \ge p_2- C_1 \omega(r) \quad \textrm{in} \quad  B_{r}\times[-r^2/2,1) .
\end{equation}
In particular, since $\omega(r) =o(r^2)$ and $\boldsymbol e_n$ is an eigenvector $D^2p_{2}$  with maximal eigenvalue, for any fixed $\delta>0$ small, we obtain 
\begin{equation}\label{ahiahiah}
\{u=0\}\cap\big(B_r\times[-r^2/2,1)\big) \subset  \{|x_n|\le r \delta^2\} \qquad \forall\,r \ll 1.
\end{equation}
Thus $\heatop ( \partial_t u) =0$ inside $\big(B_r\cap\{|x_n|>r\delta^2\}\big)\times[-r^2/2,1)$,
and therefore  \eqref{whiowehohw127} and Harnack inequality imply that  $\partial_t u(\cdot ,-r^2/4)\ge 2c_2 r^{\beta}$ inside $B_r\cap\{|x_n|>r\delta\}$, for some $c_2=c_2(n,\delta)>0$.

Combining this bound with the estimate $\partial_{tt} u\ge -C$ (see Proposition \ref{prop.semtime}), we get 
\begin{equation}\label{avguobajkf1}
\partial_t u\ge c_2 r^{\beta} \quad \mbox{in }\big(B_r\cap\{|x_n|>r\delta\}\big)\times[-r^2/4, c_3 r],
\end{equation}
for some $c_3>0$ (recall that $\beta \leq 1$).
In particular, combining \eqref{avguobajkf0} and \eqref{avguobajkf1}, we obtain 
\begin{equation}\label{avguobajkf21}
u(\cdot , -r^2/4 + h)\ge p_2- C_1 \omega(r) +c_2 r^{\beta}h \qquad \mbox{in } B_r\cap\{|x_n|>r\delta\},
\end{equation}
for all $h\in [0, c_3 r]$.
Choosing $h \ge r^2/4 +2C_1c_2^{-1}\omega(r) r^{-\beta}$, and using again \eqref{avguobajkf0}, since $u$ is nondecreasing we obtain 
\begin{equation}\label{avguobajkf2}
u \ge p_2+ C_1 \omega(r) \bigl(-1 +2\chi_{\{|x_n|>r\delta\}}\bigr) \qquad  \forall \,(x,t)\in B_r \times[ 2C_1c_2^{-1}\omega(r) r^{-\beta}, 1).
\end{equation}
Now, let $h^\delta$ be the solution to 
\[
\begin{cases}
\heatop h^\delta = 0\quad &\mbox{in } B_{1}\times(0,\infty)\\
h^\delta  = 2&\mbox{on } \big( \partial B_{1}\cap\{|x_n|> \delta\}\big)\times[0,\infty)\\
h^\delta =  0&\mbox{on } \big( \partial B_{1}\cap\{|x_n|<\delta\}\big)\times[0,\infty)\\
h^\delta =0 &\mbox{at } t=0.
\end{cases}
\]
Since $h^\delta \to 2$ as $\delta \to 0$, it follows that $h^{\delta}\ge \frac 3 2 $ inside $B_{1/2}$ for all $t\ge 1$, provided $\delta$ is small enough.
Now we can observe that
\[\psi(x,t): = p_2(x) + C_1 \omega(r) \bigg(-1 + h^\delta\Big(\frac{x}{r}, \frac{t - 2C_1c_2^{-1}\omega(r) r^{-\beta}}{r^2}\Big)\bigg)\]
 satisfies $\heatop \psi= 1$ in $ B_r \times[ 2C_1c_2^{-1}\omega(r) r^{-\beta}, \infty)$  and, by \eqref{avguobajkf2}, we have $u\ge \psi$ on the parabolic boundary $\partial_{par} \bigl(B_r \times[ 2C_1c_2^{-1}\omega(r) r^{-\beta}, 1)\bigr)$. Hence, by the maximum principle,
\[
u \ge \psi  \quad \mbox{in  }  B_r, \mbox{ for } t\ge 2C_1c_2^{-1}\omega(r) r^{-\beta} .
\]
Evaluating at $t= 2C_1c_2^{-1}\omega(r) r^{-\beta}  + r^2$  (and using that $h^{\delta}\geq\frac 3 2$ in $B_{1/2}$ for all $t\ge 1$) we obtain 
\[
u \ge \psi = p_2  + \frac{C_1}{2} \omega(r) >0 \qquad \mbox{in  }  B_{r/2}, \mbox{ for } t\ge 2C_1c_2^{-1}\omega(r) r^{-\beta}+ r^2,
\]
and the result follows.
\end{proof}

The (almost) quadratic cleaning \eqref{eq:almost quadr clean} will be proved by applying Lemma \ref{cleaning1}, with different $\omega$ and $\beta$, in each of the following cases (recall \eqref{ahoiah0P}, \eqref{eq:Sigma n-1 <3}, and \eqref{eq:Sigma n-1 =3}):
\begin{itemize}
\item[(i)]  If $(x_\circ, t_\circ)\in \Sigma_m$  for some $m\le n-2$, we will use $\omega(r) =o(r^2)$ and   $\beta=\ep$;
\item[(ii)] if $(x_\circ, t_\circ)\in \Sigma_{n-1}^{<3}$, we will use $\omega(r) =r^{\lambda^{2nd}}$ and  $\beta=\lambda^{2nd}-2+\ep$;
\item[(iii)] if $(x_\circ, t_\circ)\in \Sigma_{n-1}^{=3}$, we will use $\omega(r) =r^3$ and $\beta=1$ (in this case we can prove the exact quadratic cleaning without loosing $\ep$ in the exponent)
\end{itemize}
We start with the easiest case (i).

\begin{proposition}\label{cleaning.i}
Let $(0, 0)\in \Sigma_m$. Then, for any $\ep>0$ there exists $r_\ep>0$ such that 
\[
\{u=0\}\cap \bigl(B_r\times [r^{2-\ep},1)\bigr) = \emptyset\qquad \forall \,r\in(0,r_\ep).
\]
\end{proposition}
\begin{proof}
As shown in the proof of Lemma \ref{lem:freq2and3}(a) (see \eqref{0+eps}), we have $\ave_{B_r}  \partial_tu(\cdot ,-r^2) \ge c_\ep r^{\ep}$ in $B_r$, for all $r$ sufficiently small. Then, the result follows using  Lemma \ref{cleaning1} with  $\omega(r) =o(r^2)$ and  $\beta=\ep$. 
\end{proof}

We now consider the case (iii), which is also easier than (ii) .
\begin{proposition}\label{cleaning.iii}
Let $\varrho>0$, and let $(x_\circ, t_\circ)\in \Sigma_{n-1}^{=3} \cap \bigl(B_{1-\varrho}\times (-1+\varrho^2,1)\bigr)$. Then
\[
\{u=0\}  \cap \big( B_r(x_\circ)\times [t_\circ+ Cr^{2},1)\big) = \emptyset\qquad \forall \,r\in(0,\varrho),
\]
where $C$  is independent of $(x_\circ, t_\circ)$.
\end{proposition}

To prove Proposition \ref{cleaning.iii}, we will need the following result.

\begin{lemma}\label{lem67ytg4}
Let $(x_\circ, t_\circ)\in \Sigma_{n-1}^{=3} \cap B_{1-\varrho}\times (-1+\varrho^2,1)$,
and assume (with no loss of generality, up to a rotation in space) that $p_{2,x_\circ,t_\circ} = \frac 12  x_n^2$.
There exist positive $\hat c$ and $\hat r$, independent of $(x_\circ,t_\circ)$, such that
\begin{equation}\label{haiohoiwh}
 \ave_{B_r \cap \{|x_n|\ge\frac{r}{10} \}}  \partial_tu(x_\circ+ \cdot ,t_\circ -r^2)  \ge \hat c\, r \qquad \forall \,r\in (0,\hat r).
\end{equation}
\end{lemma}

\begin{proof}
Proposition \ref{prop:E2B3P} applied  with $\lambda^{2nd}=3$ to the function $\varrho^{-2}u(x_\circ +\varrho\,\cdot\,, t_{\circ}+\varrho^2\,\cdot\,,)$ 
implies that 
\begin{equation}\label{abuaguuwg}
\{u(x_\circ +\,\cdot\,, t_{\circ}+\,\cdot\,) =0\}\cap B_{\varrho} \subset \bigl\{|x_n| \le C_1(|x'|^2 -t)\bigr\}\qquad \forall\, t \in [-\varrho^2,0].
\end{equation}
where $C_1$ depends only on $n$, $\|u\|_{L^\infty}$, and $\varrho$. 
We now consider a barrier of the form
\[
\phi(x, t)  :=g\big( x_n - C_1(|x'|^2 -t)\big).
\]
Note that
\[
\heatop \phi =  \big(1+ 4C_1^2|x'|^2\big) g''-C_1(2n-1)  g'.
\]
Hence, choosing $g(s) =\left(e^{As}-1\right)\chi_{(0,1)}(s) + \left(e^{A}-1\right)\chi_{(1,\infty)}(s) $ with $A$ large (as in the standard barrier from Hopf's Lemma), we see that $\phi\geq 0$  is subcaloric  inside the domain 
$
D :=\{x_n - C_1(|x'|^2 -t)<1\}.
$
Hence, since $\partial_t u$ is positive and caloric inside $\{u>0\}$, it follows by  \eqref{abuaguuwg} and the maximum principle that
\[
\partial_t u \ge \bar c\, \phi \qquad \text{in }B_{\hat r}(x_\circ)\times [t_\circ-\hat r^2,t_\circ],
\]
where $\bar c,\hat r>0$ may depend on $u$ and $\varrho$, but may be chosen independently of $(x_\circ, t_\circ)$.
Using the explicit formula for $\phi$, \eqref{haiohoiwh} follows easily. 
\end{proof}

As a consequence, we get:

\begin{corollary}\label{rsizercube}
Let $(x_\circ, t_\circ)\in \Sigma_{n-1}^{=3} \cap B_{1-\varrho}\times (-1+\varrho,1)$. Then
\[
\frac {r^3} C  \le H\big(r, (u(x_\circ+ \,\cdot\,, t_\circ+\,\cdot\,)-p_2)\big)^{1/2} \le Cr^3
\]
for all $r\in (0,\varrho)$, where $C>0$  may depend on $u$ and $\varrho$, but is independent of $(x_\circ, t_\circ)$.
\end{corollary}

\begin{proof}
Choosing $\gamma=4$, 
Corollary \ref{corqwrtyu} and Lemma \ref{lem:HP}(b) applied with $R=\varrho$ imply that
\[
H(r, u(x_\circ+\,\cdot\,,t_\circ+\,\cdot\,)-p_{2,x_\circ, t_\circ})\le Cr^6\qquad \forall \,r \in (0,\varrho).
\] 
Viceversa, the opposite inequality follows easily by integrating in time the estimate in  Lemma \ref{lem67ytg4}, using that $\partial_t (u-p_2)= \partial_t u \ge 0$.\end{proof}

We can now prove Proposition \ref{cleaning.iii}.

\begin{proof}[Proof of Proposition \ref{cleaning.iii}]
In view of Lemma \ref{lem67ytg4},  the result follows from  Lemma \ref{cleaning1} with  $\omega(r) =Cr^3$ and  $\beta=1$. 
\end{proof}

Finally, we consider the remaining case (ii).

\begin{proposition}\label{cleaning.ii}
Let $(0,0)\in \Sigma_{n-1}^{<3}$. 
Then, for any $\ep>0$ there exists $r_\ep>0$ such that 
\[
\{u=0\}\cap \big(B_r\times [r^{2-\ep},1)\big) = \emptyset\qquad \textrm{for all}\quad r\in(0,r_\ep).
\]
\end{proposition}

To prove Proposition \ref{cleaning.ii} we will need the following:

\begin{lemma}\label{haihfiaha}
Let $(0, 0)\in \Sigma_{n-1}^{<3}$, and assume that $p_2=\frac 1 2 x_n^2$. Then, for any $\ep>0$ there exists $r_\ep>0$ such that 
\[
\ave_{B_r\cap \{|x_n|\ge \frac{r}{10}\}}  \partial_tu(\,\cdot\, ,-r^2) \ge  r^{\lambda^{2nd}-2+\ep}\quad\mbox{  in }B_r,\qquad \textrm{for all}\quad r\in(0,r_\ep).
\]
\end{lemma}
\begin{proof}
Let $w = u-p_2$. For any $r_k\downarrow 0$ there exists a subsequence $r_{k_\ell}\downarrow 0$ such that 
\[
\tilde w_{r_{k_\ell}} = \frac{w_{r_{k_\ell}} }{H(r_{k_\ell}, w\cutoff)^{1/2}  } \rightarrow   q,
\]
where  $q$ satisfies the properties stated in Lemma \ref{haihoiaha}.
In particular, for any such limit $q$, we have  $\partial_tq\ge 0$ and $\partial_tq\not \equiv 0$. This implies the existence of a constant $c>0$ such that
\[
\int_{B_1\times[-2,-3/2] \cap \{|x_n|\ge 1/5\}} \frac{\partial_t w_{r_{k_\ell}} }{H(r_{k_\ell}, w\cutoff)^{1/2}  } \ge c.
\]
Also, by a compactness argument, the constant $c$ can be chosen to be independent of any subsequence, therefore
\[
\int_{B_1\times[-2,-3/2] \cap \{|x_n|\ge 1/5\}} \frac{\partial_t w_{r} }{H(r, w\cutoff)^{1/2}  } \ge c>0\qquad \forall\,r>0.
\]
Hence, since $w_r$ is caloric in $\big(B_2\times(-2,0)\big)\cap\big\{|x_n|\ge {\textstyle \frac{1}{20}}\big\}$ for $r \ll 1$, the classical Harnack inequality for the heat equation implies that
\[
\ave_{(B_{r}\cap\{|x_n|\ge \frac r {10}\})\times \{-r^2\}} \frac{\partial_t u }{H(r, w\cutoff)^{1/2} }  = \ave_{(B_1\cap\{|x_n|\ge \frac 1 {10}\})\times \{-1\}}  \frac{\partial_t w_{r} }{H(r, w\cutoff)^{1/2}  } \ge c>0.
\]
Recalling that $H(r, w\cutoff)^{1/2} \ge c_\ep r^{\lambda^{2nd}+\ep/2}$ (this follows from Lemma \ref{lem:HP}(b) with $\gamma>3$), the result follows. 
\end{proof}
We can now prove Proposition \ref{cleaning.ii}.

\begin{proof}[Proof of Proposition \ref{cleaning.ii}]
Arguing as in the proof of Lemma \ref{rsizercube}, we get $H(r,u - p_2)\leq Cr^{2\lambda^{2nd}}$.
Hence, in view of Lemma \ref{haihfiaha},  the result follows from  Lemma \ref{cleaning1} with  $\omega(r) =Cr^{\lambda^{2nd}}$ and  $\beta=\lambda^{2nd}-2+\ep$. 
\end{proof}

Combining Propositions \ref{cleaning.i},  \ref{cleaning.iii}, and  \ref{cleaning.ii}, we immediately deduce the following:

\begin{corollary}\label{corhaohowih}
For any $(x_\circ, t_\circ) \in  \Sigma$ and $\ep>0$, there exists $\rho=\rho(x_\circ,t_\circ,\ep)\in (0,1)$ such that
\[
\{u=0\}\cap \big(B_r(x_\circ)\times [t_\circ+ r^{2-\ep},1)\big) = \emptyset\qquad \textrm{for all} \quad r\in(0,\rho).
\]
\end{corollary}

Finally, to prove Theorem \ref{thm-Stefan-intro-0}, we will also need the following simple GMT lemma:

\begin{lemma}\label{prop:GMT4}
Let $E\subset \R^n\times (-1,1)$ with  
\[\dim_\HH\big(\pi_x(E)\big) \le \beta.\]
Assume that for any $\ep>0$ and $(x_\circ,t_\circ) \in E$ there exists  $\rho = \rho(\ep, x_\circ,t_\circ)>0$   such that
\[
\big\{ (x,t)\in B_\rho(x_\circ)\times(-1,1) \ :\   t-t_\circ> |x-x_\circ|^{2-\ep} \big\}\cap E = \emptyset.
\]
Then $\dim_{\rm par}(E) \le \beta$.
\end{lemma}

\begin{proof}
Fix $\beta'>\beta$. 
We need to  show that, for any given $\delta>0$, the set $E$ can be covered by countably many of cylinders $B_{r_i}(x_i) \times (t_i-r_i^2, t_i+r_i^2)$ so that 
\begin{equation}\label{aisohoiwhhw}
\sum_i r_i^{\beta'} \le \delta.
\end{equation}
Choose $\ep := (\beta'-\beta)/3$, $\beta'':=\beta+\ep$, and decompose
\[
E= \cup_{\ell\ge 1} E_{\ell},  \quad \mbox{where} \quad E_{\ell} :=  \big\{ (x,t)\in E\ : \  2^{-\ell+1}\geq \rho(\ep, x,t)>2^{-\ell}\big\}.
\]
Now, given a pair of points $(x,t)$ and $(x',t')$ belonging to $E_\ell$, with $|x-x'|<2^{-\ell}$,
by applying Corollary \ref{corhaohowih} at both points we deduce that 
$|t-t'|\le|x-x'|^{2-\ep}$. 
This proves that, for all $(x,t)\in E_\ell$ and $r\in (0,2^{-\ell})$, we have
\begin{equation}\label{whiohoiwh}
E_\ell \cap \pi^{-1}_x( B_r(x)) \subset B_r\times(t-r^{2-\ep}, t+r^{2-\ep}) .
\end{equation}
Now, the assumption  $\dim_\HH\big(\pi_x(E)\big) \le \beta$ implies  $H^{\beta''}(\pi_x(E_\ell)) =0$. 
Thus, for any given $\ell$,  there exists a family of balls $\{B_{r_{k}} (x_{k})\}_k$,
with $r_{k} \in (0, 2^{-\ell})$, such that
\[
\pi_x(E_\ell) \subset \bigcup_k B_{r_{k}} (x_{k}) \quad \mbox{ and }  \quad\sum_k r_{k}^{\beta''}\le \delta 2^{-\ell} .
\]
Noting that $(-r^{2-\ep}, r^{2+\ep})$ can be covered by $r^{-\ep}$ many intervals of length $2r^2$,  \eqref{whiohoiwh} implies that 
$E_\ell\cap \pi^{-1}_x( B_{r_k}(x_k))$ can be covered by cylinders of the form 
$\{B_{r_k}(x_k) \times (t_{j,k}-r_k^2,t_{j,k}+r_k^2 )\}_{j \in I_k}$, with $\# I_k \leq r_k^{-\ep}$.
This gives a covering of $E_\ell$  such that 
\[
 \sum_{k}\sum_{j \in I_k} r_k^{\beta'} \le r_k^{-\ep} \sum_k r_k^{\beta'}=  \sum_k r_k^{\beta''} \le \delta 2^{-\ell} .
\]
Taking the union over $\ell \geq 1$ of all these coverings, since $\sum_{\ell\ge 1} \delta 2^{-\ell} =\delta$, we obtain a covering for $E$ satisfying   \eqref{aisohoiwhhw}.
\end{proof}

We are now in position to prove one of our main results.

\begin{proof}[Proof of Theorem \ref{thm-Stefan-intro-0}]
The result follows from Proposition \ref{LinMon}, Corollary \ref{corhaohowih}, and Lemma \ref{prop:GMT4}.
\end{proof}

\section{Cubic blow-ups} \label{sec:EG2BP}

The following lemma classifies possible 2nd blow-ups at points of $\Sigma_{n-1}^{=3}$.

\begin{lemma}\label{haioha78h}
Let $q$ be a 3-homogeneous solution of \eqref{PTOP}, with $\{p_2=0\}=\{x_n=0\}$.
Then 
\begin{equation}\label{eguiegiu36yg}
q(x,t) =  a |x_n|\bigg( x_n^2 + 6bt -3\sum_{\alpha,\beta=1}^{n-1} b_{\alpha\beta} x_\alpha x_\beta\bigg)  + \bar a x_n\bigg( x_n^2 + 6\bar bt -3\sum_{\alpha,\beta=1}^{n-1} \bar b_{\alpha\beta} x_\alpha x_\beta\bigg) \,, 
\end{equation}
where $a\ge 0$, $(b_{\alpha\beta}) \in \R^{(n-1)\times (n-1)}$ is nonnegative definite, and  $b = {\rm trace}(b_{\alpha\beta})$, $\bar b = {\rm trace}(\bar b_{\alpha\beta})$.
\end{lemma}

As we shall see below, Lemma \ref{haioha78h} follows easily from the following result.

\begin{lemma}\label{lem:pareven3}
Let $q:\R^n\times (-\infty,0) \rightarrow \R$ be a continuous and  $\lambda$-homogeneous function with polynomial growth, such  that $q|_{\{x_n =0\}}\ge 0$ and $\heatop q$ is a locally bounded signed measure concentrated on $\{x_n =0\}$. 
If $\lambda>0$ is an odd integer, then  $q\equiv 0$ on $\{x_n=0\}$.
\end{lemma}

\begin{proof}
For simplicity we give the proof in the case $\lambda=3$, which anyhow is the only one relevant for our purposes. The interested reader will notice that the proof for $\lambda= 5,7,9...$ is identical, but using the functions $Q$ constructed in \cite[Lemma B.5]{FRS}.

Let $x= (x',x_n)$ and define
\[
Q(x) := |x_n|\big( 3|x'|^2 -(n-1) x_n^2\big)  .
\]
Note that, since $\heatop q=\mu$ and $q$ is $\lambda$-homonegenous, $q(x, -1)$  satisfies 
\[
\mathcal L_{OU} q  +\frac{\lambda}{2} q  = \mu \quad \mbox{in }\R^n
\]
with $\lambda=3$, where $\mathcal L_{OU}$ is the operator defined in \eqref{auiwgugwar}.
On the other hand, an explicit computation shows that
\[
\mathcal L_{OU} Q +\frac{\lambda}{2} Q  =   6|x'|^2 \HH^{n-1}|_{\{x_n=0\}}\quad \mbox{in }\R^n.
\]
Hence, since $\mu$ is concentrated on $\{x_n=0\}$,
using integration by parts (which is justified by the polynomial growth of $q$ and exponential decay of the Gaussian kernel) and denoting $dm= G(x, -1)dx$ the Gaussian density,
 we obtain 
\begin{align*}
6\int_{\{x_n=0\}}  q  |x'|^2G(x, -1)\, d\HH^{n-1} = \int_{\R^n}  q\,\big(\mathcal L_{OU} +{\textstyle \frac{\lambda}{2}}\big) Q \,dm =  \int_{\R^n} \big(\mathcal L_{OU} +{\textstyle \frac{\lambda}{2}}\big)q \, Q dm
= \int_{\R^n} G(x, -1)\, Q \,d\mu = 0,
\end{align*}
where in the last equality we used that $\mu$ is supported on $\{x_n=0\}$, where $Q$ vanishes.

Thus  $\int_{\{x_n=0\}}  q  |x'|^2G(x, -1)\, d\HH^{n-1} =0$, and since $q\ge 0$ on $\{x_n=0\}$, this forces $q\equiv0$ on $\{x_n=0\}$, as wanted. 
\end{proof}

We can now prove Lemma \ref{haioha78h}.

\begin{proof}[Proof of Lemma \ref{haioha78h}]
Thanks to Lemma \ref{lem:pareven3}, we know that $q$ vanishes on $\{x_n=0\}$. 
Thus $q|_{\{x_n>0\}}$ and $q|_{\{x_n<0\}}$ are 3-homogenous caloric functions vanishing on $\{x_n=0\}$ and therefore, by Liouville Theorem, $q$ must be of the form \eqref{eguiegiu36yg} satisfying $b = {\rm trace}(b_{\alpha\beta})$ and $\bar b = {\rm trace}(\bar b_{\alpha\beta})$. 
Recalling that $q$ solves \eqref{PTOP}, and hence it is a supercaloric function, we obtain the extra conditions that $a\ge 0$ and that $(b_{\alpha\beta})$ is nonnegative definite.
\end{proof}

Our next result is the following monotonicity formula.

\begin{lemma}\label{lem:yu3y6}
Let $( 0,0)\in \Sigma_{n-1}^{=3}$ with $p_2 = \frac 1 2 (x_n)^2$. Let $Q$ be a 3-homogeneous solution of \eqref{PTOP}.
Then
\begin{equation} \label{IRP}
\frac{d}{dr} \left( \frac{1}{r^6}\int_{\{t = -r^2\} } (u-p_2)\cutoff\,Q\, G\right)  \ge -C \|Q\|_{L^2(\CC_1)}\,,
\end{equation}
where $C$ depends only on  $n$ and $\|u\|_{L^\infty}$.
\end{lemma}

\begin{proof}
After scaling we have 
\[
\frac{1}{r^6}\int_{\{t = -r^2\} } (u-p_2)\cutoff \,Q\,  G   = \frac{1}{r^3}\int_{\{t = -1\}} \big((u-p_2)\cutoff\big)_r Q   G,
\]
therefore
\[
\frac{d}{dr}\left( \frac{1}{r^6}\int_{\{t = -r^2\}} (u-p_2)\cutoff \,Q\,  G\right) = \frac{1}{r^4}  \int_{\{t = -1\}} Z \big((u-p_2)\cutoff\big)_r \,Q  \, G   -  \frac{3}{r^4}  \int_{\{t = -1\}} \big((u-p_2)\cutoff\big)_r  \,Q \,  G,
\]
where $Z$ is defined in Section~\ref{sec:operators}.
Recall  the integration by parts identities from Lemma \ref{lem:PFF000},
$$
\int_{\{t=-1\}} Zf\, g\, G =  \int_{\{t=-1\}}  f\, Zg \,G  +2\int_{\{t=-1\}} (\heatop f\,g - f\, \heatop g)\, G .
$$
Since $ZQ =3Q$, $\heatop Q =  2At \,  \HH^{n-1}|_{\{x_n=0\}} \le 0$, $u-p_2= u \ge 0$ on $\{x_n=0\}$, 
and $\heatop (u-p_2)=\chi_{\{u=0\}}$, this yields
\[
\begin{split}
\frac{d}{dr} \left(\frac{1}{r^6}\int_{\{t = -r^2\}} (u-p_2)\cutoff \,Q\,  G\right) &= \frac{1}{r^4}  \int_{\{t = -1\}} \big\{ \heatop \big((u-p_2)\cutoff\big)_r\, Q- \big((u-p_2)\cutoff\big)_r\,\heatop Q \big\}\,  G
\\
&\ge\frac{1}{r^4}\int_{\{t = -1\}}  \heatop \big((u-p_2)\cutoff\big)_r \,Q \,G
\\
& =\frac{1}{r^7} \int_{\{x\in B_{1/2}, \ t = -r^2\}}  r^2 \chi_{\{u =0\} }\,Q \,G   + O(e^{-1/r}).
\end{split}
\]
Note that $Q(x) \le C_1|x_n| |x|^2$ with $C_1 : = C_n \| Q\|_{L^2(\CC_1)}$ (see \eqref{eguiegiu36yg}).
Also, since $\lambda^{2nd}=3$, it follows by Proposition \ref{prop:E2B3P}  that
\[
\{u=0\}\cap \{ |x|<1/2, t= -r^2\} \subset \{|x_n| \le C_2(|x'| + r)^2 \},
\]
where $C_2$ depends only on $n$ and $\|u\|_{L^\infty}$. Hence, 
\[
\begin{split}
\frac{1}{r^7} \int_{\{x\in B_{1/2}, \ t = -r^2\}}  r^2 \chi_{\{u =0\} }\,Q\, G  \ge - \frac{C_1}{r^5} \int_{\{x\in B_{1/2}, \ t = -r^2\}}    \chi_{\{|x_n| \le C_2(|x'| + r)^2\}} |x_n| |x|^2 G  
\ge -C,
\end{split}
\]
where the constant $C$ depends only on $n$ and $\|u\|_{L^\infty}$, and the lemma follows.
\end{proof}

As a consequence, we find:

\begin{corollary}
\label{cor:Sigma3}
Let  $u: B_1\times(-1,1)\to [0,\infty)$ be  a bounded solution of  \eqref{eq:UPAR1}, and assume that $(0,0)\in\Sigma_{n-1}^{=3}$.
Then the limit
\[
 \lim _{r\downarrow 0}\frac{(u -p_{2})_r }{r^3} \qquad \mbox{in } \, W^{1,2}_{\rm loc}\big(\R^n \times (-\infty,0]\big)  
 \]
exists in the weak topology, and it is of the form \eqref{eguiegiu36yg}.
 \end{corollary}

\begin{proof}
Given any sequence $r_k\downarrow 0$, by Proposition \ref{prop:E2B2P} and Corollary \ref{rsizercube} there exists  subsequence such that 
\[
r_{k_{\ell}}^{-3} (u-p_2)_{r_{k_\ell}}  \rightharpoonup Q \qquad \mbox{in } \, W^{1,2}_{\rm loc}\big(\R^n \times (-\infty,0]\big) ,
\]
where $Q$ is some  3-homogeneous solution of \eqref{PTOP}.
Now, assume that we have two limits along different sequences:
\begin{equation}\label{whohwoi633tg2}
\big(r_k^{(1)}\big)^{-3} w_{r_k^{(1)}} \rightarrow Q^{(1)} \qquad \mbox{and}\qquad \big(r_k^{(1)}\big)^{-3} w_{r_k^{(2)}}\rightarrow Q^{(2)}.
\end{equation}
Then, by Lemma \ref{lem:yu3y6}, for fixed $j\in \{1,2\}$ we have that
\[
\left( \frac{1}{r^6}\int_{\{t = -1 \}} r^{-3} \big((u-p_2)\cutoff\big) Q^{(j)}  G\right)    +C \|Q\|_{L^2(\CC_1)}r
\]
is nondecreasing in $r$. 
Hence, using \eqref{whohwoi633tg2}, and letting $r_k^{(j)}\to 0$, we deduce that
\[
\int_{\{t = -1 \}} Q^{(1)} Q^{(j)}  G = \int_{\{t = -1 \}} Q^{(2)} Q^{(j)} G\quad\text{for $j=1,2$}\qquad \Rightarrow \qquad \int_{\{t = -1 \}} (Q^{(1)} -Q^{(2)})^2  G=0 . 
\]
This  implies that $Q^{(1)}\equiv Q^{(2)}$ on $\{t=-1\}$ and hence, by homogeneity, in all of $\R^n \times (-\infty,0]$.
\end{proof}

We now investigate the structure of the second blow-ups at at ``most points'' of  $\Sigma_{n-1}^{=3}$. 
For this, we need a new dimension reduction lemma.

\begin{lemma}\label{lem:PG2Bn-1}
Let $u: B_1\times(-1,1)\to [0,\infty)$ be a bounded solution of  \eqref{eq:UPAR1}, and let $(0,0)\in \Sigma_{n-1}^{=3}$.
Assume there exists a sequence  of singular points $(x_k, t_k) \in \Sigma_{n-1}^{=3}$, with $|x_k|\le r_k$ for some $r_k\downarrow 0$,  and $t_k\le 0$. Assume also that $\tilde w_{r_{k}} \rightarrow  q$, where  $w:= u-p_2$, and that $y_k :=\frac{x_k}{r_k} \to y_\infty\neq 0$,
and let $q^{even}$ denote the even symmetrisation of $q$ with respect to the hyperplane $\{p_2=0\}$.

Then $y_\infty\in \{p_2=0\}$,  and $q^{even}(x',x_n)$ is translation invariant in the direction of $y_\infty$.
\end{lemma}

\begin{proof}
Without loss of generality we assume $p_2 = \frac 1 2 x_n^2$.  
Since  $(0,0)\in \Sigma\subset\{u=0\}$, it follows from Lemma \ref{cleaning.iii} applied at $(x_k,t_k)$ that 
\[
(0,0)  \not\in B_{r}(x_k)\times [t_k+ Cr^{2},1) \qquad \forall\,r \in (0,1/2),
\]
where $C$ is independent of $k$. As a consequence, choosing $r=r_k$ and recalling that $|x_k|\le r_k$, we get
\begin{equation}\label{w3hiowh2iohe}
-Cr_k^2\le -C|x_k|^2 \le  t_k \le 0.
\end{equation}
Let $s_k:= \frac{t_k}{r_k^2}$. Up to taking a subsequence we may assume that $(y_k, s_k)\to (y_\infty, s_\infty)$, where $-C\le s_\infty\le 0$.
Also, by Lemma \ref{lem:EG2B1bbP} we obtain that $y_\infty\in\{p_2=0\}$.

Next, define $P_k :=  r_k^2\big( p_2(y_k +\cdot ) - p_{2,x_k,t_k}\big)$
(note that this is a harmonic polynomial of degree two), and set
\[
\begin{split}
w_{r_k}(y_k + \cdot, s_k +\cdot )  +P_k &=  u(x_k + r_k \cdot, t_k +r_k^2\, \cdot\, ) - r_k^2p_2(y_k +\cdot ) + r_k^2\big( p_2(y_k +\cdot ) - p_{2,x_k,t_k}\big)
\\
&=   u(x_k + r_k \cdot, t_k +r_k^2\, \cdot\, ) - p_{2,x_k,t_k}(r_k\,\cdot\,)  =: \bar w_k
\end{split}
\]
Since $(x_k, t_k)\in \Sigma_{n-1}^{=3}$  we have $\phi(0^+ \bar w_k)=3$ and therefore, by Lemma \ref{lem:HP}, 
\begin{equation}\label{8giurguiygr61}
c\bigg(\frac{\varrho'}{\varrho}\bigg)^{\negmedspace 6} 
\le 
\frac{H( \varrho',  \bar w_k \cutoff(r_k \cdot ))}{H( \varrho,  \bar w_k \cutoff(r_k \cdot ))} 
\le 
C_\delta\bigg(\frac{\varrho'}{\varrho}\bigg)^{\negmedspace 6+\delta}\qquad 
\negmedspace\negmedspace\mbox{and}\qquad  
c\bigg(\frac{\varrho'}{\varrho}\bigg)^{\negmedspace6} 
\le 
\frac{H( \varrho',   w_{r_k}\cutoff(r_k \cdot ))}{H( \varrho,  w_{r_k}\cutoff(r_k \cdot ))} 
\le 
C_\delta\bigg(\frac{\varrho'}{\varrho}\bigg)^{\negmedspace 6+\delta}, 
\end{equation}
whenever $0<\varrho <\varrho'\ll r_k^{-1}$.
Also, by Lemma \ref{lemvyg8g276187},  for all $0<\varrho  \ll r_k^{-1}$ we have
\begin{equation}\label{8giurguiygr62}
H( \varrho ,  \bar w_k \cutoff(r_k \cdot )) \asymp \ave_{\CC_\varrho} (\bar w_k)^2\qquad \mbox{and}\qquad
H( \varrho ,  w_{r_k} \cutoff(r_k \cdot )) \asymp \ave_{\CC_\varrho}   (w_{r_k})^2 ,
\end{equation}
where $X \asymp Y$ here means $X\le CY$ and $Y\le CX$ with $C$ depending only on $n$ and $\|u\|_{L^\infty}$.

Now, recall $0\ge s_k \to s_\infty >-\infty$ and $y_k \to y_\infty\in \overline B_1\cap \{p_2=0\}$,  define $a_k: = H( 1,  w_{r_k} \cutoff(r_k \cdot ))^{1/2}$ and  $b_k : =  \|P_k\|_{L^2(B_1)}$, and let us show that $b_k\le Ca_k$ as $k\to \infty$. 
Indeed,  using  \eqref{8giurguiygr61}-\eqref{8giurguiygr62} with $\varrho=1$ and $\varrho'=R \gg 1$  (we want $R^2$ to be larger than, say,  $-2s_\infty$), we get
 \[
\|\bar w_k\|_{L^2 (\CC_R)}  \le  \|w_{r_k}(y_k + \cdot, s_k +\cdot ) +P_k \|_{L^2 (\CC_R)} +  \le  \|w_{r_k} \|_{L^2 (\CC_{2R})}  + C(R)\|P_k \|_{L^2 (B_1)} \le C(R) (a_k +b_k).
\]
We now claim that $b_k \leq Ca_k.$ Indeed, if by contradiction $b_k\gg a_k$ as $k\to \infty$ then, up to subsequence (note that any sequence of quadratic polynomials bounded in $L^2$ is pre-compact),
\[
 \lim_k \frac{\tilde w_k}{a_k+b_k} =  \lim_k \frac{P_k}{b_k}  =: \tilde P  = \,[\mbox{second order harmonic polynomial}] \qquad \text{in $L^2_{\rm loc} \big(\R^n \times (-\infty, 0]\big)$. }
\]
On the other hand, it follows by \eqref{8giurguiygr61} and \eqref{8giurguiygr62}  that, for  $0<\varrho<\varrho' \ll r_k^{-1}$,
\begin{equation}\label{hwio2hufgw6}
c\bigg(\frac{\varrho'}{\varrho}\bigg)^{\negmedspace 6} \le \frac{\ave_{\CC_{\varrho'}} \left(\frac{\bar w_k}{a_k+b_k}\right)^2}{\ave_{\CC_{\varrho}}  \left(\frac{\bar w_k}{a_k+b_k}\right)^2} .
\end{equation}
Taking the limit this yields
\[
c\bigg(\frac{\varrho'}{\varrho}\bigg)^{\negmedspace 6} \le \frac{\ave_{\CC_{\varrho'}} (\tilde P)^2}{\ave_{\CC_{\varrho}} (\tilde P)^2} \qquad \forall\,0<\varrho<\varrho',
\]
a contradiction since  $\tilde P$ is quadratic (and hence cannot satisfy a cubic growth).

This proves that $b_k\le C a_k$,  hence
\[
\left\|\frac{P_k(\cdot-y_k )}{a_k}\right\|_{L^2(B_1)} \le C, \quad\mbox{ and therefore }\quad  \frac{P_k(\cdot-y_k )}{a_k}  =: \bar P_k \rightarrow \bar P_\infty
\] 
as $k\to \infty$ (up to a subsequence). 
Now, a simple computation  ---see the proof of (3.11) in \cite[Lemma 3.3]{AlessioJoaquim}--- shows that
\begin{equation}\label{anioahiohehw}
\bar P_\infty \mbox{ is odd with respect to $\{p_2=0\}$}.
\end{equation}
Therefore, since  $\frac{w_{r_k}}{a_k} = \tilde w_{r_k} \to q$  in $L^2_{\rm loc}(\R^n\times (-\infty,0])$, using again \eqref{hwio2hufgw6} (now with $a_k$ instead of $a_k +b_k$ in the denominator) we obtain
\[
c\bigg(\frac{\varrho'}{\varrho}\bigg)^{\negmedspace 6} \le \frac{\ave_{(y_k,s_k)+ \CC_{\varrho'}}  (\tilde w_{r_k} +\bar P_k)^2    }{\ave_{(y_k,s_k)+ \CC_{\varrho}}  (\tilde w_{r_k} + \bar P_k)^2} \qquad \Rightarrow \qquad c\bigg(\frac{\varrho'}{\varrho}\bigg)^{\negmedspace 6} \le \frac{\ave_{(y_\infty,s_\infty)+ \CC_{\varrho'}}  (q + \bar P_\infty)^2    }{\ave_{(y_\infty,s_\infty)+ \CC_{\varrho}}  (q+ \bar P_\infty)^2}.
\]
In particular,  choosing $\varrho'=1$ and taking even parts, \eqref{anioahiohehw} implies that 
\begin{equation}\label{whoiwhiohw}
\left( \ave_{(y_\infty,s_\infty)+ \CC_{\varrho}}  (q+ \bar P_\infty)^2\right)^{1/2} \le C\varrho^3 \quad \mbox{and}\quad \left( \ave_{(y_\infty,s_\infty)+ \CC_{\varrho}}  (q^{even})^2\right)^{1/2} \le C\varrho^3\qquad \forall\,\varrho \in (0,1).
\end{equation}
Note now that, thanks to Proposition~\ref{prop:E2B2P},
$q$ (and therefore also $q^{even}$) is a $3$-homogeneous solution of the parabolic Signorini problem in $\R^n\times (-\infty,0]$  with obstacle zero on $\{p_2=0\}$. Then, \eqref{whoiwhiohw} shows that 
$(y_\infty,s_\infty)$  is a singular point for $q^{even}$ of homogeneity at least three. In addition, since $q^{even}$ is 3-homogeenous we have 
\[
3 \le \phi\big(0^+, q^{even}(y_\infty+\,\cdot\,, s_\infty+\,\cdot\,)\big)  \le \phi\big(\infty , q^{even}(y_\infty+\,\cdot\,, s_\infty+\,\cdot\,)\big)  = \phi(\infty, q)=3 .
\]
This implies that $r\mapsto \phi\big(r, q^{even}(y_\infty+\,\cdot\,, s_\infty+\,\cdot\,)\big)$ is constantly equal to 3, therefore $q^{even}(y_\infty+\,\cdot\,, s_\infty+\,\cdot\,)$ is also 3-homogeneous. In particular
\[q^{even}(y_\infty+\,\cdot\,, s_\infty+\,\cdot\,)\equiv \lim_{R\to\infty} R^{-3}q^{even}(y_\infty+R\,\cdot\,, s_\infty+R^2\,\cdot\,) \equiv q.\]
This implies
that $q^{even}$ is translation invariant in the direction $y_\infty$,\footnote{A possible way to show this, is the following.
By homogeneity of $q^{even}(y_\infty+\,\cdot\,, s_\infty+\,\cdot\,)$ and $q^{even}$, it holds
$$
q^{even}( \delta y_\infty+ x, \delta^2 s_\infty + t)  =  \delta^{3} q^{even}( y_\infty  +(x/\delta),  s_\infty+ (t/\delta^2)) = \delta^{3} q^{even}(x/\delta,  t/\delta^2) = q^{even}(x,t)\qquad \forall\,\delta>0.
$$ Subtracting $q(x,t)$ to both sides, dividing by $\delta$, and sending $\delta\downarrow 0$ we obtain $y_\infty\cdot \nabla q(x,t) \equiv 0$.} concluding the proof of the lemma. 
\end{proof}

For the sequel, it will be useful to introduce the following:

\begin{definition}\label{defSigma*}
Given  $u: B_1\times(-1,1)\to [0,\infty)$ a bounded solution of  \eqref{eq:UPAR1},  we define $\Sigma^*\subset \Sigma_{n-1}^{=3}$ as the set of singular points $(x_\circ,t_\circ)$ such that,  in some coordinate system where $p_{2,x_\circ,t_\circ}= \frac{1}{2}(x_n)^2$, we have that
\[
p_{3,x_\circ, t_\circ}^* : = \lim_k \frac{u(x_\circ+r \,\cdot\,,t_\circ+ r^2\,\cdot\, )-r^2 \,p_{2,x_\circ,t_\circ} }{r^3}
\]
can be written as 
\[
a_{x_\circ,t_\circ} |x_n|\big(x_n^2+6t\big) + p_{3,x_\circ, t_\circ},
\]
for some $a_{x_\circ,t_\circ}\ge 0$ and $p_{3,x_\circ, t_\circ}$ an odd  3-homogeneous caloric polynomial.
\end{definition}

The motivation of the previous definition is given by the following lemma, which follows from Lemma \ref{lem:PG2Bn-1} and Proposition \ref{prop:GMTfut}.

\begin{lemma}\label{lem:P6784t78}
Let $u: B_1\times(-1,1)\to [0,\infty)$ be a bounded solution of  \eqref{eq:UPAR1}.
Then
\[\dim_{\HH}\big(\pi_x( \Sigma_{n-1}^{=3} \setminus \Sigma^*)\big)\le n-2.\]
\end{lemma}

\begin{proof}
We begin by observing the following: Let $Q=Q(x,t)$ be an 3-homogenous solution of the parabolic Signorini problem, which is even with respect to $\{x_n=0\}$ and it is invariant with respect to translations parallel to $\{x_n=0\}$. Then $Q$ depends on the variables $x_n$ and $t$, and hence it must be of the form
\begin{equation}\label{now21go}
Q(x,t) =  a|x_n|\bigg(\frac{x_n^2}{6}+t\bigg),
\end{equation}
for some $a\ge0$. 
We now show that this implies the following:

\smallskip

{\em
\noindent {\bf Claim.} Let $(x_\circ,t_\circ)\in \Sigma_{n-1}^{=3}\setminus  \Sigma^*$, and $(x_k^{(j)}, t_k^{(j)})$ a sequence of singular points in $\Sigma_{n-1}^{=3}$, with $t_k^{(j)}\le t_\circ$, such that $y_k^{(j)}: = \frac{x_k^{(j)}-x_\circ}{r_k} \to y_\infty^{(j)}\neq 0$. Then
$y_\infty^{(j)}\in \{x_n=0\}$, and the set $\big\{y_\infty^{(j)}\big\}_j$ has rank at most $n-2$.
}

\smallskip

\noindent
Indeed, applying Corollary~\ref{cor:Sigma3} and Lemma \ref{lem:PG2Bn-1} to the function $u(x_\circ+\,\cdot\,, t_\circ+\,\cdot\,)$, we deduce that 
$$\frac{u(x_\circ+r_k \,\cdot\,,t_\circ+ r_k^2\,\cdot\, )-r_k^2 \,p_{2,x_\circ,t_\circ} }{r_k^3} \to q$$ and $y_\infty^{(j)}\in \{x_n=0\}$.
Then, if the rank of $\big\{y_\infty^{(j)}\big\}_j$ was $n-1$, Lemma \ref{lem:PG2Bn-1} would imply that 
 $q^{even}$ is of the form \eqref{now21go}. 
Recalling Definition \ref{defSigma*}, this would mean that $(x_\circ,t_\circ) \in\Sigma^{*}$, a contradiction.

Thanks to the Claim, we can apply Proposition \ref{prop:GMTfut} to deduce the desired dimensional bound on $\pi_x(\Sigma_{n-1}^{=3}\setminus  \Sigma^*)$  (more precisely, since the control that we have now is for ``points from the past'' instead of on ``points in the future'', we need to apply
Proposition \ref{prop:GMTfut} to the set $E: = \{(x,-t)  \,:\, (x,t) \in \Sigma_{n-1}^{=3}\setminus  \Sigma^*\}$).
\end{proof}

We now show that, for any $(x_\circ,t_\circ)\in \Sigma^{*}$, the coefficient $a_{x_\circ,t_\circ}$ must be strictly positive.

\begin{lemma}\label{lem:84y6t3}
Let $u: B_1\times(-1,1)\to [0,\infty)$ be a bounded solution of  \eqref{eq:UPAR1}, and $\varrho \in (0,1)$. There exists $c_\varrho>0$ such that the following holds: For any
$(x_\circ,t_\circ)\in \Sigma^*\cap B_{1-\varrho}\times (-1+\varrho^2,1)$, let  $a_{x_\circ,t_\circ}$  be the coefficient of the non-caloric part of  $p_{3,x_\circ, t_\circ}^*$, as in Definition \ref{defSigma*}. Then $a_{x_\circ,t_\circ}\geq c_\varrho$.
\end{lemma}

\begin{proof}
Let $(x_\circ, t_\circ)\in \Sigma^*\cap B_{1-\varrho}\times (-1+\varrho^2,1)$.
By Lemma \ref{lem67ytg4}, there exist positive constants $c$ and $r_\circ$ such that 
\[
\ave_{B_r}  \partial_tu(x_\circ+\,\cdot\, ,t_\circ-r^2) \ge c r \qquad\mbox{  in }B_r, \quad \forall \,r\in (0,r_\circ).
\]
Equivalently,  defining $w:= u(x_\circ+\,\cdot\, ,t_\circ+\cdot) -p_{2,x_\circ,t_\circ}$ and noting that $\partial_t w = \partial_t u \ge 0$, we have
\[
\ave_{B_1\times\{-1\}}   \partial_t  (r^{-3}w_r) \ge c >0 \qquad\, \quad \forall \,r\in (0,r_\circ).
\]
Taking the limit as $r \to 0$,
we deduce that
\[
\ave_{B_1\times\{-1\}}   \partial_t   \big(a_{x_\circ,t_\circ} |x_n|\big(x_n^2+6t\big) + p_{3,x_\circ, t_\circ}\big) \ge c>0.
\]
Since $p_{3,x_\circ, t_\circ}$ is odd, so is $\partial_t p_{3,x_\circ, t_\circ}$ and therefore $\ave_{B_1\times\{-1\}}  \partial_t p_{3,x_\circ, t_\circ}=0$. The lemma follows.
\end{proof}

\section{A useful monotone quantity}
\label{sec-10}

Let us define a smoothed version of $H$ given by
\begin{equation}\label{ahoihaoihoia}
\tilde H(\varrho, w) :=   \int_{1}^2  H(\varrho\theta, w) d\theta
\end{equation}
The goal of this section is to establish the following monotonicity formula. This will be crucial in the next sections in order to establish higher order regularity at (most) singular points in $\Sigma^*$.

\begin{proposition}\label{srjklhgfd}
Let  $u: B_1\times(-1,1)\to [0,\infty)$ be a bounded solution of  \eqref{eq:UPAR1}, let $(0,0)\in \Sigma^*$, 
and assume that 
$p_2 = \frac{1}{2} x_n^2$ and $p_3^* = \overline a|x_n|(x_n^2+6t) +p_{3}$,.

Given $C_\circ>0$ and $M\in (6,7)$, there exist $\ep_\circ>0$ and $R_\circ \ge 1$ such that the following holds:
Assume that there exists $r\in (0, \ep_\circ)$ such that
\begin{equation}\label{assumption-10.2}
\|u-p_2 -p_3^*\|_{L^\infty(B_{R_\circ \rho}\times (-2\varrho^2, \varrho^2/2))} \le \ep_\circ \varrho^3\quad \mbox{ for all }\varrho \in (0,r).
\end{equation}
Denote  $w := (u-p_2 -Q) \cutoff$ with $Q = x_n Q_2 + \bar a|x_n|(x_n^2+6t)$, where $\bar a \in [0, C_\circ]$ and  $Q_2$ is any 2-homogeneous polynomial satisfying $\heatop(x_nQ_2)=0$ and $\|Q_2\|_{L^2(\CC_1)}\le C_\circ$.
Then
\begin{equation} \label{h5/2}
\hhh(\varrho, w) : = \max_{R\in [1,1/\varrho]} R^{-M} \varrho^{-5}   \tilde H(R\varrho, w)
\end{equation}
is monotone nondecreasing for  $\varrho\in (0,r)$.
\end{proposition}

The idea behind the quantity $\hhh(\varrho,w)$ is the following.
Ideally, we would like to have that $\varrho^{-6}H(\varrho,w)$ is monotone at every point.
Unfortunately this is false at points in $\Sigma^*$ because $p_3^*$ is not caloric.
This means that we must look for a monotone quantity of lower order, and this is why we look at a quantity of the form $\varrho^{-5}H(\varrho,w)$. 
Our monotone quantity $\hhh$ is basically a modified version of this, in which we first consider an averaged version of $H$, given by \eqref{ahoihaoihoia}, and then define $\hhh$ as in \eqref{h5/2} in order to control the growth between different scales (see \eqref{abohaohaio2}).
It is thanks to these (small but crucial) modifications that we can prove that the quantity is monotone.

Before proving the monotonicity formula, we need the following:

\begin{lemma}\label{lem_barriermon}
Let  $u: B_1\times(-1,1)\to [0,\infty)$ be a bounded solution of  \eqref{eq:UPAR1}. 
Given positive constants $c_\circ$, $r_\circ$, $C_\circ$,  and $R_\circ$, there exists $\delta>0$ such that the following holds:
Let $Q_{2}$ be any 2-homogeneous polynomial such that $\heatop (x_n Q_2) =0$ and $\|Q_2\|_{L^2(\CC_1)}\le C_\circ$.
Assume that
\[
 \frac{1} {r^3} \bigg(u( r\,\cdot\,,  r^2\,\cdot\,)  - r^2 \frac{x_n^2}{2}\bigg)     \le c_\circ |x_n|t   + C_\circ |x_n|^3 +   x_n Q_2(x')  +   \delta  \quad \mbox{inside } \big\{|x|\le 2R_\circ, -4\le t\le -1\big\}.
 \]
 for all $r\in( 0,  r_\circ)$. 
 Then
 \[
u( r\,\cdot\,, r^2\,\cdot\,)    \le C r^4 \quad \mbox{ on } \ \{x_n= 0\} \cap  \{|x|\le R_\circ, t=-1\},
\] 
for all $r\in(0,r_\circ)$.
\end{lemma}

\begin{proof}
Fix $z\in  \{x_n=0\}\cap \{|x|\le R_\circ, \,t=-1\}$, and define  
\[
\phi(x,t) :=  -nx_n^2 + |x'- z' |^2 - 2(1+t) +  x_n Q_2(x')
\]
and
\[
v(x,t) : = \frac{1} {r^3} \bigg(u( rx,  r^2t)  - r^2 \frac{x_n^2}{2}\bigg).
\] 
Also, consider the set $U : =  \{ |x'-z'| \le 1, \,|x_n| \le \varrho, \,-2\le t \le -1\}$, with $\varrho>0$ to be fixed. 
We choose first  $\varrho>0$ and then $\delta >0$, both  conveniently small,  so that  for all $x\in \partial_{par}  U $  we have
\[
\begin{split}
v(x) &\le  c_\circ |x_n| t  + C_\circ |x_n|^3  + x_n Q_2(x)   +   \delta  
\\
&\le  
\begin{cases}
-c_\circ \varrho + C_\circ \varrho^3 +x_n Q_2(x) + \delta  \le -2n\varrho^2  + x_n Q_2(x) \le \phi(x) \quad &\mbox{for } |x_n|=\varrho 
\\
-c_\circ |x_n|+ C_\circ \varrho^3 +x_n Q_2(x)+ \delta  \le -2n\varrho^2  + x_n Q_2(x)  + 1 \le \phi(x)&\mbox{for } |x_n|\le \varrho,  |x'-z'| = 1
\\
C_\circ \varrho^3 +x_n Q_2(x)  + \delta \le -2n\varrho^2+ x_n Q_2(x) + 2  \le \phi(x)&\mbox{for } |x_n|\le \varrho,  t=-2.
\end{cases}
\end{split}
\]
Now, consider the function
 \[
 \psi(x) : = \frac {1}{2 r} x_n ^2  + \phi(x) + Cr,
 \] 
where $C$ is a large enough constant, chosen so that $\psi$ is nonnegative in $U$ (note that $\frac {1}{2 r} x_n ^2+Cr \geq \sqrt{C}|x_n|$, so such a constant $C$ exists independently of $r$).  
Since $\heatop \phi = 0$,  we have  $\heatop \psi \le \frac 1 r $ in $U$.    Also, by the argument above, $\frac{1} {r^3} u( r\,\cdot\,,  r^2\,\cdot\,)  \leq \psi$ on $\partial_{par}  U $.
Therefore, since $\psi\ge 0$ in $U$ and $\heatop \big(\frac{1} {r^3} u( r\,\cdot\,,  r^2\,\cdot\,) \big)     \ge \frac 1{r}$ whenever $\frac{1} {r^3} u( r\,\cdot\,,  r^2\,\cdot\,) >0$, it follows from the comparison principle that
\[
\frac{1} {r^3} u( r\,\cdot\,,  r^2\,\cdot\,)    \le \psi   \quad \mbox{in }U.
\]
In particular, choosing $x=z$ and $t=-1$ we get (recall that $z_n=0$)
\[
\frac{1} {r^3} u( rz,  -r^2)   \le \psi(z,-1)  = Cr,
\]
and the lemma follows.
\end{proof}

We can now prove Proposition \ref{srjklhgfd}.

\begin{proof}[Proof of Proposition \ref{srjklhgfd}]
Note that $\hhh(\cdot , w) $ is a Lipschitz function, so it suffices to prove that its derivative is nonnegative at every differentiability point.
Thus, we compute 
$\frac{d}{d\varrho}\big|_{\varrho = \varrho_\circ} \hhh(\varrho , w) $ and we note that two possible alternatives arise:

(a) $\max_{R\in [1,1/\varrho_\circ]} R^{-M}  \tilde H(R\varrho_\circ, w) >  \tilde H(\varrho_\circ, w)$;

(b)  $\max_{R\in [1,1/\varrho_\circ]} R^{-M}  \tilde H(R\varrho_\circ, w)  =  \tilde H(\varrho_\circ, w)$.
\vspace{4pt}

If $(a)$ holds, consider $R_\circ \in (1,1/\varrho_\circ]$ such that $\max_{R\in [1,1/\varrho_\circ]} R^{-M}  \tilde H(R\varrho_\circ, w)  =  R_\circ^{-M} \tilde H( R_\circ \varrho_\circ , w)$.
By continuity, for $s\in (0,1)$ sufficiently close to $1$ we have 
\begin{multline*}
\varrho_\circ^{5} \hhh(\varrho_\circ , w)=(R_\circ)^{-M} \tilde H( R_\circ \varrho_\circ , w)  = \max_{R'\in [s,1/\varrho_\circ]} (R')^{-M}  \tilde H(R'\varrho_\circ, w) \\
=  s^{-M}\max_{R\in [1,1/(s\varrho_\circ)]}  R^{-M} \tilde H(Rs\varrho_\circ, w)=s^{-M} (s\varrho_\circ)^{5} \hhh(s\varrho_\circ , w) .
\end{multline*}
Hence
\[
\hhh(s\varrho_\circ , w) = s^{M-5} \hhh(\varrho_\circ, w) \qquad \text{for $s \in (0,1)$ sufficiently close to $1$,}
\]
and therefore $\frac{d}{d\varrho}\big|_{\varrho = \varrho_\circ} \hhh(\varrho , w) = \frac{M-5}{\varrho_\circ} \hhh(\varrho_\circ, w) \ge 0$. This proves the result whenever $(a)$ holds.

\vspace{4pt}

Instead, when (b) holds, we note that $\hhh(\varrho,w) \geq \varrho^{-5}\tilde H( \varrho, w)$ with equality for $\varrho=\varrho_\circ$. Therefore
$$
 \frac{\hhh(\varrho,w)-\hhh(\varrho_\circ,w)}{\varrho-\varrho_\circ} \geq \frac{\varrho^{-5}\tilde H( \varrho, w)-\varrho_\circ^{-5}\tilde H( \varrho_\circ , w)}{\varrho-\varrho_\circ} \qquad \Rightarrow \qquad \frac{d}{d\varrho}\bigg|_{\varrho = \varrho_\circ} \hhh(\cdot , w)=\frac{d}{d\varrho}\bigg|_{\varrho = \varrho_\circ} \big( \varrho^{-5}\tilde H( \rho, w) \big).
$$
In addition, if (b) holds then we have the extra growth information 
\begin{equation}\label{abohaohaio}
  \tilde H(R\varrho_\circ, w)   \le \tilde H(\varrho_\circ, w) R^{M} \quad \mbox{for all } R\in [1, 1/\varrho_\circ],
\end{equation}
and in particular,  taking $R= 1/\varrho_\circ$, we get
\begin{equation}\label{abohaohaio2}
\varrho_\circ^{M}  \tilde H(1, w)   \le \tilde H(\varrho_\circ, w).
\end{equation}
Therefore, to conclude the proof of the lemma, it is enough to show that $\frac{d}{d\varrho}\big|_{\varrho = \varrho_\circ} \big( \varrho^{-5}\tilde H( \varrho , w) \big)\ge 0  $ at a every scale $\varrho_\circ$ where \eqref{abohaohaio} holds. We divide the proof of this fact in four steps.

\smallskip

\noindent $\bullet$ {\em  Step 1}. Denote 
\[v = (u - p _2) \cutoff,\qquad h = (x_nQ_2 +\overline a \Psi)\cutoff, \qquad \Psi(x, t): = |x_n|(x_n^2+6t),\]
where $\cutoff$ is the usual spatial cut-off. 
Recalling that $w=v-h$ we first show that, at any small enough scale $\varrho_\circ$ at which \eqref{abohaohaio} holds, we have
\begin{equation}\label{gtrdfoijoyutrse}
\frac{d}{d\varrho} \bigg|_{\varrho = \varrho_\circ} \big( \varrho^{-5} \tilde H( \varrho , w) \big)\ge  \frac{1}{\varrho_\circ^6} \bigg(  \frac{\tilde H(\varrho_\circ, v-h)}{2}  + 4\int_{-2\varrho_\circ^2}^{-\varrho_\circ^2}\biggl[ \int_{\R^n} \big\{ v\heatop v-   (v-h) \heatop (v-h)   \big\} G\biggr]\,dt\bigg)
\end{equation}
Observe that, by \eqref{ahoihaoihoia}, we have 
\begin{equation}
\frac{d}{d\varrho} \tilde H(\varrho, w) :=   \int_{1}^2 \frac{dH}{d\varrho} (\varrho\theta, w)\, \theta \,d \theta, 
\end{equation}
so we need to compute $\frac{d}{d\varrho} H(\varrho\theta, w)$.
Using the identity from Lemma \ref{lem:PFF000}
\[
\langle h, Zv \rangle_\varrho  =   \langle Zh, v \rangle_\varrho  + 2\varrho^2 \big( \langle h, \heatop v \rangle_\varrho  -  \langle \heatop h, v \rangle_\varrho   \big) ,
\]
 we obtain 
\[
\begin{split}
\frac{d}{d\varrho} H(\varrho, v-h)& = \frac {2}{\varrho}  \langle  v-h , Z(v-h) \rangle_\varrho
\\
&=\frac {2}{\varrho} \big(  \langle  v , Zv \rangle_\varrho  -   \langle  h , Zv \rangle_\varrho - \langle  v , Zh \rangle_\varrho +  \langle  h , Zh\rangle_\varrho\big) 
\\
&=\frac {2}{\varrho} \big(  \langle  v , Zv \rangle_\varrho  - 2\langle  v , Zh \rangle_\varrho  - 2\varrho^2 \big( \langle h, \heatop v \rangle_\varrho  -  \langle \heatop h, v \rangle_\varrho   \big) +   \langle  h  , Zh\rangle_\varrho\big) 
\\
& \ge \frac {2}{\varrho} \big(  3\langle  v , v \rangle_\varrho  +2\varrho^2\langle  v , \heatop v \rangle_\varrho  - 6\langle  v , h \rangle_\varrho  - 2\varrho^2 \big( \langle h, \heatop v \rangle_\varrho  -  \langle \heatop h, v \rangle_\varrho   \big) +  3 \langle  h  , h\rangle_\varrho\big)-Ce^{-1/\varrho}
\\ & =  \frac {2}{\varrho} \big(3  H(\varrho, v-h)  + 2\varrho^2 \big(  \langle v, \heatop v \rangle_\varrho  -  \langle h, \heatop v \rangle_\varrho  +  \langle \heatop h, v \rangle_\varrho   \big) \big)-Ce^{-1/\varrho} \\
& \geq \frac {2}{\varrho} \big(3  H(\varrho, v-h)  + 2\varrho^2 \big( 2\langle v, \heatop v \rangle_\varrho  - \langle v-h, \heatop( v-h) \rangle_\varrho    -  2\langle h, \heatop v \rangle_\varrho   \big) \big)  - C e^{-1/\varrho}.
\end{split}
\]
Here, we used that $Zh=3h$, that $\langle  v , Zv \rangle_\varrho \geq 3\langle  v , v \rangle_\varrho + 2\varrho^2 \langle  v , \heatop v \rangle_\varrho-Ce^{-1/\varrho}$ (this follows from \eqref{idD}, Proposition \ref{prop:EMF1P}, and the fact that the frequency at $(0,0)$ is $3$),
and that $ \big| \langle h, \heatop h\rangle_\varrho   \big|\le Ce^{-1/\varrho}$.
Therefore
\[
\begin{split}
\frac{d}{d\varrho}  \big( \varrho^{-5}H( \varrho , w) \big)&= \frac{1}{\varrho^6} \big( \varrho {\textstyle \frac{d}{d\varrho}} H( \varrho , w) - 5 H( \varrho , w) \big) 
\\
&\ge  \frac{1}{\varrho^6} \big(  H(\varrho, v-h)  + 4\varrho^2 \big(  2\langle v, \heatop v \rangle_\varrho  - \langle v-h, \heatop (v-h) \rangle_\varrho    -  2\langle h, \heatop v \rangle_\varrho   \big) \big)  - C e^{-1/\varrho}.
\end{split}
\]
Now recall that $v=(u-p_2)\cutoff$ and hence
\[(u-p_2)\heatop(u-p_2)\ge 0  \qquad \Rightarrow \qquad \langle  v , \heatop v \rangle_\varrho  \ge - C e^{-1/\varrho}.\]
On the other hand, since $\heatop(u-p_2) =-\chi_{\{u=0\}}$, $h= (x_nQ_2 +\bar a \psi)\cutoff$, and $|x_nQ_2 +\bar a \psi|\le C|x_n||x|^2$, we obtain
\[ 
\big|\langle h, \heatop v \rangle_\varrho\big| \le \int_{\{t = -\varrho^2\}}   |h|\, \chi_{\{ u=0\} } G  + Ce^{-1/\varrho}    \le  C \int_{\{t = -\varrho^2\}}   C|x_n|\, |x|^2 \chi_{\{ u=0\} } G + Ce^{-1/\varrho}.
\]
Also, using that (as a consequence of Proposition \ref{prop:E2B3P})
\begin{equation}\label{ehiehihw73}
\{u=0\}\cap \{ |x|<1/2, t= -\varrho^2\} \subset \{|x_n| \le C(|x'| + \varrho)^2 \},
\end{equation}
we obtain 
\begin{equation}\label{rdjk;kgftdr}
\big| \langle h, \heatop v \rangle_\varrho  \big|  \le    \int_{\{t = -\varrho^2\}}  C (\varrho+ |x|) ^4  \chi_{ \{ |x_n| \le C(|x'| + \varrho)^2 \}} G   +  Ce^{-1/\varrho}  \le   C\varrho^5.
\end{equation}
Therefore, we have shown that
\[
\frac{d}{d\varrho}  \big( \varrho^{-5}H( \varrho , w) \big)\ge  \frac{1}{\varrho^6} \bigg(  H(\varrho, v-h)  +4\varrho^2  \big( \langle v, \heatop v\rangle_\varrho -    \langle v-h, \heatop (v-h) \rangle_\varrho \big) -C\varrho^7 \bigg)    - C e^{-1/\varrho}
\]
Applying this inequality with $H(\varrho\theta,w)$ in place of $H(\varrho,w)$, and integrating with respect to $\theta\,d\theta$ over $[1,2]$, 
we obtain 
\[
\frac{d}{d\varrho}  \big( \varrho^{-5}\tilde H( \varrho , w) \big)\ge  \frac{1}{\varrho^6} \bigg(  \tilde H(\varrho, v-h)  + 4 \int_{-2\varrho^2}^{-\varrho^2} \bigg[\int_{\R^n} \big\{ v\heatop v - (v-h) \heatop (v-h) \big\} G\bigg]\,dt   -C\varrho^7   \bigg)  - C e^{-1/ \varrho}.
\]
Since $\tilde H(\varrho_\circ , v-h) \ge  c \varrho_\circ^M  \gg \varrho_\circ^7 \gg C e^{-1/ \varrho_\circ}$ (recall \eqref{abohaohaio2}), \eqref{gtrdfoijoyutrse} follows.

\smallskip

\noindent $\bullet$ {\em  Step 2}. Let
 \[w_{\varrho_\circ}  :=  (v-h)(\varrho_\circ \,\cdot\,, \varrho_\circ ^2\,\cdot\,) =  \big((u-p_2-p_3-\bar a \Psi) \zeta \big)(\varrho_\circ \,\cdot\,, \varrho_\circ ^2\,\cdot\,). \] 
We now show that for any $\ep>0$ there exists $\delta>0$ such that, if $\varrho_\circ$ is sufficiently small, 
\begin{equation}\label{srtdyuijhdzszfgjhkl}
\int_{-2}^{-1}\bigg(\int_{\R^n} |w_{\varrho_\circ}\heatop w_{\varrho_\circ}| \min\{ G, \delta\}\bigg)\,dt  \le \ep \tilde H (\varrho_\circ,  w_{\varrho_\circ} ).
\end{equation}
This will be used in order to bound the integral in \eqref{gtrdfoijoyutrse} outside a large ball, as we will see in Step 4 below.

We first note that, inside $B_{1/2}\times(-1,0)$, it holds 
$$
v\heatop v  = (u-p_2) \heatop(u-p_2)\ge 0,\quad h\heatop h=0, \quad \text{and}\quad -v\heatop h  = -u\heatop h \ge 0.
$$
Therefore, using \eqref{ehiehihw73} similarly to the argument in the previous step, we get
\[\begin{split}
(u-p_2-x_nQ_2-\bar a \Psi)\heatop (u-p_2-x_nQ_2-\bar a \Psi) &\ge - (x_nQ_2+\bar a \Psi)\heatop(u-p_2) \\
& \ge   - C (\varrho+ |x|) ^4  \chi_{ \{ |x_n| \le C(|x'| + \varrho)^2\}}  \quad \mbox{in }B_{1/2}\times(-\varrho^2,0).
\end{split} \]
Hence, after scaling and taking into account the error introduced by the cut-off $\cutoff$, for all $\varrho\in (0, 1/2)$ it holds
\[
w_{\varrho_\circ}   \heatop  w_{\varrho_\circ} \ge  -C\varrho_\circ ^2 \min \big\{  (\varrho+ \varrho_\circ|x|)^4  \chi_{\{ \varrho_\circ |x_n| \le C ( \varrho +\varrho_\circ|x'|)^2\}}  , 1\big\} \quad  \mbox{ in }\R^n\times(-\varrho^2/\varrho_\circ^2, 0).
\]
Equivalently, setting $\varrho =R\varrho_\circ$,  for all  $R \in \big(0, \frac{1}{2\varrho_\circ}\big)$ we have
\[w_{\varrho_\circ}   \heatop  w_{\varrho_\circ} \ge  -C\varrho_\circ ^6 \min \big\{  (R+|x|)^4  \chi_{\{  |x_n| \le C\varrho_\circ ( R +|x'|)^2\}}  , \varrho_\circ ^{-4}\big\} \quad  \mbox{ in }\R^n \times(-R^2, 0).\]
Integrating this bound for $R \in (1,2)$, we obtain 
\begin{equation}\label{ytdyroifxgdz2}
 \int_{-4}^{-1}\int_{\R^n} \big(w_{\varrho_\circ}   \heatop  w_{\varrho_\circ} \big)_- G\le  \int_{\R^n}  C\varrho_\circ ^6   (4+|x|)^4  \chi_{\{ |x_n| \le C\varrho_\circ (4 +|x|)^2\}}G \le C\varrho_\circ^7.
\end{equation}

Now, note  that the following identity holds for any function $\phi = \phi(x,t) \in C^1_tC^{1,1}_x$ with compact support and for any smooth $\xi = \xi(x,t)$:
\begin{equation}\label{drtkldfghjfdsall}
2\int_{\R^n \times\{t\}} \phi\,   \heatop  \phi \,\xi  + 2 \int_{\R^n\times\{t\} } |\nabla \phi| ^2\xi  + \frac{d}{dt} \int_{\R^n\times\{t\}} \phi^2\xi   =  \int_{\R^n\times\{t\}} \phi^2 (\Delta+\partial_t) \xi.   
\end{equation}
In particular, choosing $\phi=w_{\varrho_\circ}$ and $\xi=G$, we get
\begin{equation}\label{drtkldfghjfds}
2\int_{\R^n \times\{t\}} w_{\varrho_\circ} \,  \heatop  w_{\varrho_\circ} \,G + 2 \int_{\R^n\times\{t\} } |\nabla w_{\varrho_\circ}| ^2G  + \frac{d}{dt} \int_{\R^n\times\{t\}} w_{\varrho_\circ}^2G   =0.   
\end{equation}
Integrating between $t=-\theta$ and $t=-1$, we obtain
\begin{multline}
2 \int_{-\theta}^{-1}\biggl( \int_{\R^n} w_{\varrho_\circ}   \heatop  w_{\varrho_\circ} \,G\biggr)\,dt + 2 \int_{-\theta}^{-1} \biggl(\int_{\R^n} |\nabla w_{\varrho_\circ}|^2 \,G\biggr)\,dt + \sup_{t\in (-\theta, -1)}  \int_{\R^n\times\{t\}} w_{\varrho_\circ}^2G
\\ \le  \int_{\{t=-\theta \}} w_{\varrho_\circ}^2G   \le \int_{-4}^{-3} \int_{\R^n}  w_{\varrho_\circ}^2G,
\end{multline}
where $\theta\in(3,4)$ is chosen (thanks to Fubini) so  that $\int_{\{t=-\theta \}} w_{\varrho_\circ}^2G   \le \int_{-4}^{-3} \int_{ B_{2R}}  w_{\varrho_\circ}^2$. 
Hence, since $\int_{-4}^{-3} \int_{\R^n}  w_{\varrho_\circ}^2G \le \tilde H(2,w_{\varrho_\circ})\le C\tilde H(1,w_{\varrho_\circ})$ (by the definition of $\tilde H$ and using \eqref{abohaohaio} with $R=2$), recalling  \eqref{ytdyroifxgdz2} we deduce 
\[
2 \int_{-\theta}^{-1}\bigg( \int_{\R^n} |\nabla w_{\varrho_\circ}|^2 \,G\bigg)\,dt + \sup_{t\in (-\theta, -1)}  \int_{\R^n\times\{t\}} w_{\varrho_\circ}^2G 
\le   C\tilde H(1,w_{\varrho_\circ}) + C\varrho_\circ^7\le   C\tilde H(1,w_{\varrho_\circ}),
\]
where the last inequality follows from \eqref{abohaohaio2} (indeed, since $M<7$, $\tilde H(1,w_{\varrho_\circ})\gg \varrho_\circ^7$ for $\varrho_\circ>0$ small).

In particular, again by Fubini, there exists $\tilde \theta\in (2,3)$ such that 
\[
\int_{\R^n\times\{-\tilde \theta\}} |\nabla w_{\varrho_\circ}|^2 \,G +   \int_{\R^n\times\{-\tilde \theta\}} w_{\varrho_\circ}^2G  \le   C\tilde H(1,w_{\varrho_\circ}).
\]

We now claim that this implies, thanks to the log-Sobolev inequality, that 
\begin{equation}\label{whiohtewiohtweoi}
\int_{\R^n\times\{\tilde \theta\}} w_{\varrho_\circ}^2 \min\{G,\delta\}  \le   \ep(\delta)\tilde H(1,w_{\varrho_\circ}), 
\end{equation}
where $\ep(\delta)\to 0$ as $\delta\to 0$. 

Indeed, the function $f :=\frac{w_{\varrho_\circ}( \tilde \theta^{1/2} \,\cdot\,,\tilde \theta \,\cdot )}{C\tilde H(1,w_{\varrho_\circ})^{1/2}}$ satisfies 
\[
\int_{\R^n} |\nabla f|^2\,dm  \le 1 \quad \mbox{and}\quad \int_{\R^n} |f|^2\,dm  \le 1,
\]
where $dm = G(\,\cdot\,, -1)dx$ is the Gaussian measure. 
Hence, given $\lambda>1$, we can apply the Gaussian log-Sobolev inequality
to $F:=\max\{|f|,1\}$ (cp. Lemma~\ref{lem877r6edfcv}) to get
\begin{multline*}
\log\lambda   \int_{\R^n} f^2 \chi_{\{f^2>\lambda\}} dm  \leq \int_{\R^n} F^2 \log F^2  dm \leq \int_{\R^n} |\nabla F|^2   dm+\biggl(\int_{\R^n} F^2 dm\biggr)\, \log\biggl(\int_{\R^n} F^2 dm\biggr)\\
\leq \int_{\R^n} |\nabla f|^2dm +\biggl(1+\int_{\R^n} f^2 dm\biggr)\, \log\biggl(1+\int_{\R^n} f^2 dm\biggr)
\leq 1+ 2\log 2<3.
\end{multline*}
Thus, for any $R>0$ we have
\[
\int_{\R^n\setminus B_{R}} f^2 dm \leq  \int_{\R^n \setminus B_{R}} f^2 \chi_{\{f^2\le \lambda\}} dm + \int_{\R^n \setminus B_{R}} f^2 \chi_{\{f^2>\lambda\}} dm 
\le \lambda\,\int_{\R^n \setminus B_{R}}  dm + \frac{3}{\log\lambda}.
\]
Choosing $\lambda=\lambda_\varepsilon$ large enough so that $3/\log\lambda_\varepsilon<\varepsilon/4$, and then $R_\varepsilon>0$ so that $\lambda_\varepsilon\,\int_{\R^n \setminus B_{R_\varepsilon}}  dm<\varepsilon/4$, we deduce that
\[\int_{\R^n\setminus B_{R_\varepsilon}} f^2 dm < \frac{\varepsilon}{2}.\]
Therefore, \eqref{whiohtewiohtweoi} follows by taking $\delta$ small enough.

Now, since $(\Delta+\partial_t)\min\{G,\delta\} \leq 0$, we can use again \eqref{drtkldfghjfdsall} with $\phi=w_{\varrho_\circ}$ and $\xi=  \min\{G,\delta\}$ to get
\begin{equation}\label{drtkldfghjfdsbis}
2\int_{\R^n \times\{t\}} w_{\varrho_\circ} \,  \heatop  w_{\varrho_\circ} \,\min\{G,\delta\} + 2 \int_{\R^n\times\{t\} } |\nabla w_{\varrho_\circ}| ^2\min\{G,\delta\}  + \frac{d}{dt} \int_{\R^n\times\{t\}} w_{\varrho_\circ}^2\min\{G,\delta\}   \le0.
\end{equation}
Integrating between $t=- \tilde \theta$ and $t=-1$, exactly as before (but now taking advantage of \eqref{drtkldfghjfdsall}) we obtain 
\[\begin{split}
2 \int_{-\tilde \theta}^{-1} \biggl(\int_{\R^n} |w_{\varrho_\circ}   \heatop  w_{\varrho_\circ}| \,\min\{G,\delta\} \biggr)\,dt \le  \ep(\delta) \tilde H (1, w_{\varrho_\circ}) + C\varrho_\circ^7 \le 2\ep(\delta) \tilde H (1, w_{\varrho_\circ}),
\end{split}
\]
if $\varrho_\circ$ is sufficiently small. 
Since $\tilde\theta\in(2,3)$, this concludes the proof of \eqref{srtdyuijhdzszfgjhkl}.

\smallskip

\noindent $\bullet$ {\em  Step 3}. Let $u_\varrho := u(\varrho\,\cdot\,, \varrho^2\,\cdot\,)$. We now show that, for any fixed $R_\circ>0$, there exists $r>0$ such that 
\begin{equation}\label{dsdfuiokjiuftdr}
0 \le  -\int_{-2}^{-1}\biggl(\int_{B_{R_\circ}} u_\varrho \heatop ( \varrho^3 \Psi ) \biggr)\,dt \le C\varrho^7 \qquad \forall\, \varrho\in (0,r).
\end{equation}

Indeed, thanks to Lemma \ref{lem:84y6t3} (recall that $(p_3^*)^{even}$ denotes the even symmetrization with respect to $\{p_2=0\}=\{x_n=0\}$), it holds
\[ (p_3^*)^{even} = \bar a(t|x_n| + {\textstyle \frac16}|x_n|^3)\qquad \text{with }\bar a\geq c_\circ>0.\]
Therefore \eqref{assumption-10.2} implies that
\[
 \frac{1} {\varrho^3} \bigg(u( \varrho\,\cdot\,,  \varrho^2\,\cdot\,)  - r^2 \frac{x_n^2}{2}\bigg)     \le c_\circ |x_n|t   + C|x_n|^3 +   p_3^{odd}(x)  +   \varepsilon_\circ  \qquad \mbox{in } \{|x|\le 2R_\circ, -4\le t\le -1\},
 \]
 and so Lemma \ref{lem_barriermon}  implies
 \[
0\le u( \varrho\,\cdot\,,  \varrho^2\,\cdot\,) \le C\varrho^4 \quad \mbox{on } \{x_n=0\}\cap B_{R_{\circ}}.
 \]
Recalling that $\Psi(x,t) = t |x_n| +  |x_n|^3/6$ (in particular, its heat operator is concentrated on $\{x_n=0\}$), \eqref{dsdfuiokjiuftdr} follows easily.

\smallskip

\noindent $\bullet$  {\em  Step 4}.   We finally conclude the proof of the lemma by combining \eqref{gtrdfoijoyutrse}, \eqref{srtdyuijhdzszfgjhkl}, and \eqref{dsdfuiokjiuftdr}.

Indeed, recall that we want to show that $\frac{d}{d\varrho}|_{\varrho=\varrho_\circ} \big(\varrho^{-5} \tilde H(\varrho,w)\big)\geq 0$ at every scale $\varrho_\circ$ where \eqref{abohaohaio}-\eqref{abohaohaio2} holds.
In view of \eqref{gtrdfoijoyutrse}, it is enough to prove that 
\begin{equation}\label{fghjkjhgfd}
 \int_{-2\varrho_\circ^2}^{- \varrho_\circ ^2}\int_{\R^n} \big\{ v\heatop v - (v-h) \heatop (v-h) \big\} G \ge  -\frac {1} {10} \tilde H(\varrho_\circ, v-h).
 \end{equation}
 We split the integral
 \[
 \int_{-2\varrho_\circ^2}^{-\varrho_\circ ^2} \int_{\R^n} \big\{ v\heatop v - (v-h) \heatop (v-h) \big\} G = \int_{-2}^{-1}\int_{\R^n} \big\{v_{\varrho_\circ} \heatop v_{\varrho_\circ} -w_{\varrho_\circ} \heatop w_{\varrho_\circ}  \big\} G 
 \]
into two pieces, according to the partition $\R^n\times (-2,-1) =  B_{R_\circ}\times (-2,-1) \cup (\R^n \setminus  B_{R_\circ})\times (-2,-1)$. 

Note that, using Lemma \ref{lemerrors}, \eqref{srtdyuijhdzszfgjhkl}, and the bound $G\leq Ce^{-cR^2}$ in $(B_{2R}\setminus B_R)\times (-2,-1)$, we get
\[
\begin{split}
 \int_{-2}^{-1}\int_{(\R^n \setminus B_{R_\circ})} \big\{v_{\varrho_\circ} \heatop v_{\varrho_\circ} - w_{\varrho_\circ} \heatop w_{\varrho_\circ}  \big\} G & 
 \ge - \int_{-2}^{-1}\int_{(\R^n \setminus B_{R_\circ})} |w_{\varrho_\circ} \heatop w_{\varrho_\circ} | G - Ce^{-c\varrho_\circ}
 \ge- \frac 1 {20} \tilde H(\varrho_\circ, v-h), 
\end{split}
\]
provided that $R_\circ$ is taken sufficiently large.
(Here we used that, given $\delta_\circ>0$, we have  $\min\{G, \delta_\circ\}=G$ in $\R^n\setminus B_{R_\circ}$ if $R_\circ$ is sufficiently large.)

On the other hand, to estimate the integral  in $B_{R_\circ}\times (-2,-1)$ we use that $w_\varrho = v_\varrho + \varrho^3 (x_nQ_2+\bar a\Psi)$   inside the integration domain, that $\heatop(x_nQ_2)=0$, and  $\Psi\heatop \Psi \equiv 0$. This implies that
\[
 \int_{-2}^{-1}\int_{B_{R_\circ}} \big\{v_{\varrho_\circ} \heatop v_{\varrho_\circ} - w_{\varrho_\circ} \heatop w_{\varrho_\circ}  \big\} G  =   \int_{-2}^{-1}\int_{B_{R_\circ}}  \big\{v_{\varrho_\circ} \heatop( \varrho_\circ^3 \bar a\Psi) + ( \varrho_\circ^3 (x_nQ_2+\bar a\Psi))  \heatop v_{\varrho_\circ}   \big\} G
\]
Now, as in \eqref{rdjk;kgftdr} we have the bound 
\[
\bigg|\int_{-2}^{-1}\int_{B_{R_\circ}}  ( \varrho_\circ^3 (x_nQ_2+\bar a\Psi))  \heatop v_{\varrho_\circ}   G\bigg| \le C\varrho_\circ^7.
\]
Also, using  \eqref{dsdfuiokjiuftdr}, 
\[
 0 \le  -\int_{-2}^{-1}\int_{B_{R_\circ}} v_{\varrho_\circ} \heatop( \varrho_\circ^3 \Psi)    G =     -\int_{-2}^{-1}\int_{B_{R_\circ}} u_{\varrho_\circ} \heatop( \varrho_\circ^3 \Psi)     \le C\varrho_\circ^7,
\]
thus, recalling \eqref{abohaohaio2},
\[
 \int_{-2}^{-1}\int_{B_{R_\circ}} \big\{v_{\varrho_\circ} \heatop v_{\varrho_\circ} - w_{\varrho_\circ} \heatop w_{\varrho_\circ}  \big\} G  \ge - C\varrho_\circ^7  \gg -\frac{c}{20} \varrho_\circ^M \ge -\frac{1}{20} H(\varrho_\circ, v-h).\]
 This concludes the proof of \eqref{fghjkjhgfd} and therefore the proof of the proposition.
\end{proof}

\section{The set $\Sigma^{*}$: $\ep$-flatness vs accelerated decay} \label{sec:dichotomy}

To continue our analysis, we now make a finite partition of $\Sigma^*$ as follows: given $\delta>0$, we consider
\begin{equation}
\label{eq:partition Sigma*}
\Sigma^{*} = \bigcup_{\ell\in \mathbb N}    \Sigma^*_{\delta,\ell},
\qquad\text{with}
\qquad
\Sigma^*_{\delta, \ell} : = \bigg\{ (x_\circ, t_\circ)\in \Sigma^* \ :   H( 1,  p^*_{3, x_\circ, t_\circ})  \in  \frac{\delta}{2} [\ell, \ell+1) \bigg\}.
\end{equation}
Since $ 1/C_{\varrho} \le H( 1,  p^*_{3, x_\circ, t_\circ})\le C_\varrho$ for all $(x_\circ, t_\circ)\in \Sigma^*\cap \big( B_{1-\varrho}\times (-1+\varrho, 1)\big)$
(by Corollary \ref{rsizercube}), for any given $\delta>0$ the number of sets in the previous partition restricted to compact subsets of $B_1\times(-1,1)$ is bounded (their number is roughly $C/\delta$). 

As we will see, the advantage of considering this partition is the following key property:

\begin{lemma}\label{lem:barrier111}
Let $u: B_1\times(-1,1)\to [0,\infty)$ be a bounded solution of  \eqref{eq:UPAR1}.
For any $\ep>0$, $R_\circ>1$, and $\varrho\in (0,1)$, there exists $\bar \delta>0$ such that the following holds provided $\delta \leq \bar\delta$:\\
Given $\ell \in \mathbb N$ and $(x_\circ, t_\circ) \in \Sigma_{\delta, \ell}^*\cap (B_{1-\varrho}\times(-1+\varrho,1))$, there exists $r_\circ = r_\circ(x_\circ, t_\circ)>0$ such that 
\[
\big\|u(x_1+ \cdot, t_1 + \cdot ) - p_{2, x_1, t_1} - p_{3, x_1, t_1} ^*\big\|_{L^\infty (B_{R_\circ r}\times (-2r^2,-r^2/2))} \le \ep  r^3 \qquad \forall\,r \in (0,r_\circ)
\]
for all  $(x_1, t_1)\in \Sigma^*_{\delta, \ell}$  such that  $|x_1-x_\circ|\le r_{\circ}$.
\end{lemma}

\begin{proof}
Let $\bar\ep>0$ be a small constant to be fixed later, and let $(x_1,t_1)$ be as in the statement.
 Define 
\[
\boldsymbol v^r: = \frac{ u(x_1 + r \cdot,  t_1+ r^2\cdot )-  r^2p_{2,x_1, t_1}} {r^3}.
\]  
We divide the proof in two steps. 

\smallskip

\noindent $\bullet$  {\em Step  1}. We show that
\begin{equation}\label{fyguhjlk;lkjgfx}
\|  \boldsymbol v^r -p^*_{3,x_1, t_1} \|^2_{L^2(B_{2R_\circ}\times(-4,-1/2))}   \le C(R_\circ) \bar  \ep\qquad   \forall \,r\in (0, r_\circ).
\end{equation}
Indeed, by definition of $\Sigma_{\delta,\ell}^*$, since $(x_\circ, t_\circ),(x_1, t_1) \in \Sigma_{\delta, \ell}^*$ it holds
\begin{equation}\label{mlojuy3}
\bigg|\int_{\{t=-1\}} (p_{3,x_1, t_1}^*)^2 G  -\int_{\{t = -1 \}}  (p_{3,x_\circ,t_\circ}^*)^2 G\bigg|\le \delta.
\end{equation}
Also, by Lemma \ref{lem:yu3y6}, the quantity  
\[
\mathcal I_{x_1,   t_1}(r,Q) : = \frac{1}{r^6}\int_{\{t = -r^2\}} \big(u(x_1 +\,\cdot\,, t_1+ \, \cdot\,)-p_{2,x_1, t_1}\big) \cutoff \, Q \, G
\]
satisfies 
\[
\frac{d}{dr} \mathcal I_{x_1,   t_1}(r,Q)  \ge -C \| Q\|_{L^2(\CC_1)}
\]
for any $( x_1,   t_1)\in \Sigma^{*}_{\delta,\ell}\cap B_{1-\varrho}\times (-1+\varrho, 1)$ and  for any  given 3-homogenous solution $Q$ of \eqref{PTOP}.

Hence, on the one hand we have 
\begin{equation} \label{yfuthihfgx3}
\begin{split}
 \int_{\{t = -1 \}} \boldsymbol v^r\, p^*_{3,x_1,t_1} G = \mathcal I_{x_1,t_1}(r,  p^*_{3,x_1, t_1} ) 
 \ge  \mathcal I_{x_1,t_1}(0^+,  p^*_{3,x_1, t_1})  - Cr
=  \int_{\{t = -1 \}}  (p_{3,x_1, t_1})^2  G - Cr.
\end{split}
\end{equation}
On the other hand, given $\ep'>0$, we can fix $s_\circ>0$ ---depending on $(x_\circ, t_\circ)$ and $\ep'$--- such that 
\[
\int_{\{t = -1 \}} (p_{3,x_\circ,t_\circ}^*)^2 G   + \ep' \ge s_\circ^{-6}H\big(s_\circ, (u(x_\circ+\,\cdot\,, t_\circ+ \, \cdot\,)- p_{2,x_\circ, t_\circ})\cutoff\big).
\] 
Note that, since
\[
u(x_1 +\,\cdot, t_1+ \cdot\big)- p_{2,x_1, t_1}  \longrightarrow u(x_\circ + \,\cdot, t_\circ+ \cdot\big)- p_{2,x_\circ, t_\circ} \quad \mbox{as }(x_1, t_1)\to (x_\circ, t_\circ),
\] 
then for  $(x_1,t_1)$ sufficiently close to $(x_\circ, t_\circ)$ and for $r\in (0,s_\circ)$ we will have 
\[
\begin{split}
\int_{\{t = -1 \}} (p_{3,x_\circ,t_\circ}^*)^2 G  +2\ep' &\ge  s_{\circ}^{-6}H\big(s_\circ, ( u(x_\circ+\,\cdot\,, t_\circ+ \,\cdot\, )- p_{2, x_\circ, t_\circ})\cutoff\big) +\ep'
\\
& =   s_{\circ}^{-6} H\big(1,  (u( x_\circ + s_\circ\,\cdot\,, t_\circ+ s_\circ^2\,\cdot\, )- s_{\circ}^2p_{2, x_\circ, t_\circ})\cutoff\big) +\ep'
\\
& \ge  s_{\circ}^{-6} H\big(1,  (u(x_1 + s_\circ\,\cdot\,,  t_1+ s_\circ^2 \,\cdot\, )- s_\circ^2 p_{2,x_1, t_1})\cutoff\big)
\\
& =  s_{\circ}^{-6} H\big(s_\circ,  (u(x_1 + \cdot,  t_1+ \cdot )-  p_{2,x_1, t_1})\cutoff\big)
\\
& \ge r^{-6} H\big(r,  (u(x_1 + \cdot,  t_1+ \cdot )-  p_{2,x_1, t_1})\cutoff\big) + O(e^{-c/s_{\circ}})
\\
&= \int_{\{t = -1 \}} \big(  \boldsymbol v^r\big)^2 G+ O(e^{-c/s_\circ}).
\end{split}
\]
Combining this estimate with \eqref{yfuthihfgx3} and \eqref{mlojuy3}, and choosing first $s_\circ$ small enough and then $(x_1,t_1)$ sufficiently close to $(x_\circ, t_\circ)$, we obtain 
\[
\begin{split}
 2\int_{\{t = -1 \}}  \boldsymbol v^r\,p^*_{3,x_1,t_1} G \ge \int_{\{t = -1 \}} \big( (\boldsymbol v^r)^2 +  (p^*_{3,x_1,t_1})^2 \big) G -\bar \ep\qquad \forall\, r \in (0,s_\circ).
\end{split}
\]
In other words,
\[
\int_{\{t = -1 \}}  (\boldsymbol v^r - p^*_{3,x_1, t_1} )^2 G \le \bar \ep \qquad \forall \,r\in (0, s_\circ) ,
\]
and \eqref{fyguhjlk;lkjgfx} follows.

\smallskip

\noindent $\bullet$  {\em Step  2}. We upgrade the $L^2$ bound from \eqref{fyguhjlk;lkjgfx} to an $L^\infty$ bound.

For this, recall that $(p^*_{3,x_1, t_1})^2$ is always divisible by $p_{2, x_1, t_1}$,therefore
\[
\mathcal P_1 : =p_{2, x_1, t_1 p_3}  + p^*_{3,x_1, t_1}  + {\textstyle \frac{(p^*_{3,x_1, t_1})^2} {4p_{2, x_1, t_1}}   }  \ge 0
\]
satisfies 
\[
\heatop  \mathcal  P_1 \le   1 + \heatop \left({ \textstyle \frac{(p^*_{3,x_1, t_1})^2}{4p_{2, x_1, t_1}} }  \right ) \le   1+ C(|x|^2 +t).
\]
Thus, since $\heatop u=1$ in $\{u>0\}$,
\[ 
\heatop \big( u(x_1 + r \cdot,  t_1+ r^2\cdot ) - \mathcal P_1(r\cdot, r^2\cdot ) \big) _+\ge -C r^4 \qquad \mbox{in } B_{2R_\circ r}  \times (-3r^2, -r^2/2).
\]
On the other hand since $\heatop \mathcal P_1 = 1+ O(|x|^2 +t)$  inside $\{ \mathcal P_1>0\}$, while  $\heatop u\le 1$ and  $u\ge 0$,  we have
\[
\heatop  \big( u(x_1 + r \cdot,  t_1+ r^2\cdot ) -  \mathcal P_1(r\cdot, r^2\cdot ) \big) _- \ge -Cr^4 \qquad \mbox{in } B_{2R_\circ r} \times(-3r^2,-r^2/2).
\]
This proves that 
\[
\heatop  \big| u(x_1 + r \cdot,  t_1+ r^2\cdot ) -  \mathcal P_1(r\cdot, r^2\cdot ) \big| \ge -Cr^4 \qquad \mbox{in } B_{2R_\circ r} \times(-3r^2,-r^2/2).
\]
Also, \eqref{fyguhjlk;lkjgfx} implies that
\[
 \big\|  u(x_1 + r \,\cdot\,,  t_1+ r^2\,\cdot\, ) -  \mathcal P_1(r\,\cdot\,, r^2\,\cdot \,)  \big\|_{L^2( B_{2R_\circ r} \times(-3r^2,-r^2/2)) } \le \bar \ep r^3 + Cr^4.
\]
Therefore, it follows from the one-sided Harnack inequality for the heat equation that
\[
\big\| u(x_1 + r \,\cdot\,,  t_1+ r^2\,\cdot\, ) -  \mathcal P_1(r\,\cdot\,, r^2\,\cdot \,) \big\|_{L^\infty( B_{R_\circ r} \times(-2r^2,-r^2/2)) } \le C \bar \ep r^3 + Cr^4.
\]
This  implies the conclusion of the lemma by taking $\bar\ep$ and $r$ sufficienlty small.
\end{proof}

We have the following:

\begin{corollary}\label{cor-monot}
Let  $u: B_1\times(-1,1)\to [0,\infty)$ be a bounded solution of  \eqref{eq:UPAR1}. 
Given $\varrho\in (0,1)$, $C_\circ$, and $M\in(6,7)$, $\bar \delta>0$ such that the following holds provided $\delta \leq \bar\delta$:\\
Let $\hhh$ be as in \eqref{h5/2}, and recall \eqref{eq:partition Sigma*}.
For any given $\ell\in \N$ and $(x_\circ, t_\circ) \in \Sigma_{\delta, \ell}^*\cap \big(B_{1-\varrho}\times(-1+\varrho,1)\big)$, there exists $r_\circ = r_\circ(x_\circ, t_\circ)>0$ such that 
\[
\hhh\big(r, (u(x_1+ \cdot, t_1 + \cdot ) - p_{2, x_1, t_1} - Q_{x_1,t_1})\cutoff\big)\qquad \text{is monotone nondecreasing for $r\in (0, r_\circ)$}
\]
for all  $(x_1, t_1)\in \Sigma^*_{\delta, \ell}$ with $|x_1-x_\circ|\le r_{\circ}$. Here $Q_{x_1,t_1}=x_nQ_2 +\bar a |x_n|(x_n^2 +6t)$ in coordinates such that $p_{2,x_1,t_1} = \frac 1 2 x_n^2$,
where  $\bar a\in [0,C_\circ]$ and $Q_2$ is any 2-homogeneous polynomial satisfying $\heatop(x_nQ_2)=0$ and $\|Q_2\|_{L^2(\CC_1)}\leq C_\circ$.
\end{corollary}

\begin{proof}
It follows from Proposition \ref{srjklhgfd} and Lemma \ref{lem:barrier111}.
\end{proof}

\vspace{3mm}

The next proposition will be crucial in our argument. It provides us with the following powerful decay property:
if 
$\Sigma^*_{\delta,\ell}$ is not $\ep$-flat at some small scale $r$, then
we have an accelerated decay for $u$ at such scale.
This result will be at the core of the argument in the next section, where we will prove that we have an accelerated decay at $\Sigma^*$, outside a set of Hausdorff dimension at most $n-2$.

Recall that, given a function $q$, we denote by $q^{even}$ the even symmetrization of $q$ with respect to the hyperplane $\{p_2=0\}$.

\begin{proposition} \label{prop_dicho}
 Let $u:B_1\times(-1,1) \to [0,\infty)$ be a solution of  \eqref{eq:UPAR1}, let $\varrho>0$, and let $\delta>0$ be given by Corollary \ref{cor-monot}.
Let $(x_\circ,t_\circ)\in \Sigma^*_{\delta,\ell}\cap \big(B_{1-\varrho}\times(-1+\varrho,1)\big)$ for some $\ell\in \N$, and 
choose coordinates so that  $p_{2,x_\circ, t_\circ} = \frac 12 x_n^2$ and $(p_{3,x_\circ, t_\circ}^*)^{\rm even} = a_{x_\circ, t_\circ} \Psi$, where 
 \[\Psi(x, t) : = |x_n|( x_n^2 + 6t).\] 
Let $\hat{\alpha} : = 1/3$,  and for $r\in (0,\varrho]$ define 
\[
A(r) :=  \min_{a\in \R} \big\| (u(x_\circ+\,\cdot, t_\circ+\,\cdot\,)-p_{2,x_\circ, t_\circ} -p_{3,x_\circ, t_\circ}^* -a \Psi)_r  \big\|_{L^2(\CC_1)} 
\qquad
\text{and}
\qquad  
\Theta(r) : = \max_{s\in [r,\varrho]}  (r/s)^{3+\hat{\alpha}} A(s)\,.
\]
\noindent Then the following two properties hold.
\vspace{3pt} 

\noindent {\rm (1)} Fix $\ep>0$. For any $\ep_1$ there exists 
$r_{\ep_1}>0$ such that, for every $r\in (0,r_{\ep_1})$,  the following holds:\\
Consider the ``$\ep$-flatness property''
\begin{equation}
\label{eq:Pr}
\pi_x(\Sigma^*_{\delta,\ell} \cap \{t\le t_\circ \}) \cap B_r(x_\circ) \subset L + \ep B_r \qquad \text{for some linear space $L$ with ${\rm dim}(L)\le n-2$.}
\end{equation}
If \eqref{eq:Pr} fails, then $$A(r) \le \ep_1 \,\Theta(r).$$

\noindent {\rm (2)} For any $\ep_2>0$ there exists $\lambda_{\ep_2}\in (0,1)$ such that,  for all $r\in(0,\varrho)$  and $\lambda \in (0, \lambda_{\ep_2}]$,
\[
\Theta(\lambda r) \le   \lambda^{3-\ep_2}\, \Theta(r).
\]
\end{proposition}

\begin{proof}
Assume for notational simplicity that $(x_\circ,t_\circ) =(0,0)$ and $\varrho=1$. We begin by proving (1).
\vspace{5pt}

\noindent
{\bf $\bullet$  Proof of (1).} Set
\[
\overline{a}_r : = {\rm arg\,min}_{a\in \R} \big\| (u-p_2 -p_3^*-a\Psi)_r  \big\|^2_{L^2(\CC_1)},\qquad
w^r : = \frac{(u-p_2-p_{3}^*-\overline{a}_r  \Psi)_r}{ \Theta(r) }.
\]
Note that, by construction, for any $s>0$  we have  $\Theta(2s)\le 2^{3+\hat{\alpha}} \Theta(s)$. Hence
\[
 \big\| (u-p_2-p_{3}^*-\overline{a}_{2^j r}\Psi)_{2^j r}  \big\|_{L^2(\CC_1)} = A(2^j r) \le \Theta(2^j r) \le (2^j)^{3+\hat{\alpha}} \Theta(r) \qquad \forall\,j \geq 1.
\]
Rescaling and using the  triangle inequality between two consecutive dyadic scales, this gives
\[
\begin{split}
c \big|\overline{a}_{2^{j+1} r}-\overline{a}_{2^{j} r}\big| 2^{3j}  &= \big\| (\overline{a}_{2^{j+1} r}\Psi -\overline{a}_{2^{j} r}\Psi)_{2^jr}  \big\|_{L^2(\CC_1)}  
\\ 
&\le \big\| \big(u-p_2-p_{3}^* -\overline{a}_{2^{j+1} r}\Psi\big)_{2^{j}r}  \big\|_{L^2(\CC_2)} + \big\| \big(u-p_2-p_{3}^* -\overline{a}_{2^{j} r}\Psi\big)_{2^{j}r}   \big\|_{L^2(\CC_1)}   
\\
&= 2^{\frac{n+2}{2}}A(2^{j+1}r) + A(2^j r) \le C2^{(3+\hat{\alpha})j} \Theta(r).
\end{split}
\]
Hence, summing the geometric series and using again the triangle inequality, we obtain $|\overline{a}_{2^{j} r}-\overline{a}_{r}| \leq C2^{j\alpha\circ} \Theta(r)$,
which implies the growth control
\begin{equation}\label{rstdhjlkhgf}
\big\|w^r (R\,\cdot\,, R^2\, \cdot\, ) \big\|_{L^2 (\CC_1)} \le   CR^{3+\hat{\alpha}}  \qquad \forall R\in (1,1/r).
\end{equation}
Now, assume by contradiction that there exist a constant $\ep_1>0$ and a sequence $r_k \downarrow 0$ for which \eqref{eq:Pr} fails at the scales $r=r_k$ but $A(r_r) \ge \ep_1 \,\Theta(r_k)$, or equivalently
\begin{equation}\label{drj;ufzrhjkl}
\| w^{r_k} \|_{L^2(\CC_{1})} > \ep_1,
\end{equation}
while the negation of \eqref{eq:Pr} gives  the existence of $x_{k}^{(1)}, x_{k}^{(2)}, \dots, x_{k}^{(n-1)} \in \pi_x(\Sigma^*\cap \{t\leq0\}) \cap B_{r_k} $ such that $y_{k}^{(i)}: = x_{k}^{(i)}/r_k \in \overline{B_1}$ satisfy
\[
  {\rm dist}\big( \{ y_{k}^{(i)} \}_ { 1\le i\le n-1} , L \big) \ge \ep \,  \quad  \ \forall \,L\subset \R^n \mbox{ linear space with }{\rm dim}(L)= n-2.
\]
The outline of the proof will go as follows.

In {\em Step 1}, we show that $w^{r_k}$ converges (up to a subsequence) strongly in $L^2_{\rm loc} (\R^n\times(-\infty,0])$ to a some function $w^0 \not\equiv 0$
which satisfying the growth condition 
\begin{equation}\label{drklkjhgfdzs}
\| w^{0} (R\,\cdot\,, R^2\, \cdot\, ) \|_{L^2 (\CC_1)} \le   R^{3+\hat{\alpha}},  \qquad \forall R\in (1,\infty)
\end{equation}
and solves
\begin{equation}\label{drklkjhgfdzs2}
\begin{cases}
\heatop w^0 = 0  \quad &  \mbox{ in } \{ x_n \neq 0\}
\\
w^0 = 0  & \mbox{ on } \{ x_n =0\}.
\end{cases}
\end{equation}
This will imply that $w^0$ is a caloric polynomial degree 3 in each side of $\{x_n=0\}$.

In {\em Step 2}, we use Proposition \ref{srjklhgfd} to show that 
\begin{equation}
\label{drklkjhgfdzs30}
\mbox{  $w^0$ is homogeneous of degree  $3$}.
\end{equation}

In {\em Step 3}, we use the monotonicity formulae from  Lemma \ref{lem:yu3y6} to show that 
\begin{equation}
\label{drklkjhgfdzs31}
\mbox{ $w^0$  is even in the variable $x_n$,}
\end{equation}
so that, up to a rotation fixing $\boldsymbol e_n$,  $w^0$ must be of the form
\begin{equation}
\label{drklkjhgfdzs32}
w^0(x,t) \equiv  a \Psi +\sum_{i=1}^{n-1} a_i \,|x_n|( -3 x_i^2  + x_n^2).
\end{equation}
In addition we  show that, thanks to the choice of $\bar a_r$, $w^0$ satisfies 
\begin{equation}
\label{drklkjhgfdzs33}
\int_{\CC_1} w^0 \Psi =0
\end{equation}
and hence $a=0$.

Finally,  in {\em Step 4}  we exploit again the monotonicity formula from Proposition  \ref{srjklhgfd}, but now at points $(x_k^{(i)} , t_k^{(i)})$,  in order to reach a contradiction with \eqref{drklkjhgfdzs32}-\eqref{drklkjhgfdzs33} to complete the proof. 

We now develop Steps 1-4.  

\smallskip
 
$\bullet$ {\em Step 1}. We prove first  the compactness of the functions $w^{r_k}$ in $C^{0}_{\rm loc} (\R^n \times(-\infty,0])$. 
To prove this we show that, for any given  $R>0$, there  exists $k_\circ$ such that, for $k\geq k_\circ$, we have
\begin{equation}\label{goal1}
\int_{B_R \times(-R^2, 0)}  \bigl(|\nabla w^{r_k}|^2 + |\partial_t w|^2\bigr)  \le C_R .
\end{equation}
This will give compactness $w^{r_k}$  in $L^2_{\rm loc}$, and we will later upgrade this convergence to a locally uniform one.

Since $\Theta(r_k)\ge c r_k^{3+\hat{\alpha}}$ (thanks to \eqref{drj;ufzrhjkl}), arguing similarly to  \eqref{ytdyroifxgdz2} we get 
\begin{equation}\label{fyujlkgfdsdfg}
w^{r_k}   \heatop  w^{r_k} \ge  -C(r_k) ^{-2\hat{\alpha}}  R^4 \chi_{\{  |x_n| \le Cr_k R^2\}}   \quad  \mbox{ in }\CC_r .
\end{equation}
Using \eqref{drtkldfghjfdsall} with $\phi=w^{r_k}$ and with $\xi$ replaced by  $\xi_R(x)  := \xi(x/R)$, where $\xi\in C^\infty_c(B_2)$ is a nonnegative cut-off function satisfying $\xi\equiv 1 $ in $B_1$, we obtain 
\begin{equation}\label{iojsjfjoskbsi}
\int_{B_{2R}\times\{t\}}\big( |w^{r_k}  \heatop  w^{r_k}| + |\nabla w^{r_k} |^2\big) \xi_R   +  \frac{d}{dt} \int_{B_{2R}\times\{t\}} (w^{r_k})^2 \xi_R   \le  \frac{C_n}{R^2}\int_{B_{2R}\times\{t\}} (w^{r_k})^2  +  Cr _k^{1-2\hat{\alpha}} R^{n+5}.
\end{equation}
for $t \in (-R^2, 0)$ and  $R \in (0, 1/2r_k)$. 
Integrating  \eqref{iojsjfjoskbsi} in time and using \eqref{rstdhjlkhgf}, this yields
\begin{equation} \label{hwoihwohw}
\frac{1}{R^2}\int_{\CC_R}  |w^{r_k}|^2 +  \int_{\CC_R}  | \nabla w^{r_k}|^2 \le C(R).
\end{equation}
Also, since $\partial_{tt}(u-p_2-p_{3,r}^*) = \partial_{tt}u \ge -C$ (see Proposition \ref{prop.semtime}), it follows by scaling that $\partial_{tt} w_{r_k} \ge -Cr _k^{1-\hat{\alpha}}$. 

We now use the following elementary bound: for $f: (0,R)\to \R$,
\[
\ave_0^R R^2| f'|^2\le  C\bigg(  { \ave_0^R|f|^2} + \sup_{(0,R)} R^4 (f'')^2_-  \bigg),
\]
Applying this with $f(\,\cdot\,)= u(x, \,\cdot\, )$ for each fixed $x\in B_{R}$, and then integrating in the  $x$ variable, it follows by \eqref{hwoihwohw} that 
\[
\int_{\CC_R}  |\partial_t w^{r_k}|^2 dx\,dt   \le C(R)+o_k(1), 
\]
where $o_k(1)\to 0$ as $k\to \infty$.
This gives \eqref{goal1}, and thus $w^{r_k}\rightarrow  w^0$ in $L^2_{\rm loc}(\R^n\times (-\infty,0])$.
Notice also that $w^0$ is nonzero by \eqref{drj;ufzrhjkl}.
Moreover, reasoning exactly as in Step 2 in the proof of Lemma \ref{lem:barrier111}, such $L^2_{\rm loc}$ convergence implies $C^0_{\rm loc}$ convergence. 

The growth condition \eqref{drklkjhgfdzs} follows from \eqref{rstdhjlkhgf}, so it remains to show that $w^0$ is satisfies \eqref{drklkjhgfdzs2}.

First, using Lemma \ref{lem_barriermon},
\[
w^{r_k} |_{\{x_n = 0\}} \rightarrow 0 \quad \mbox{ in }C^0_{\rm loc}( \{x_n=0\}).
\]
On the other hand, since $\heatop w^{r_k}  = 0$ in  $\R^n \setminus (\{u_{r_k}=0\}  \cup \{x_n=0\})$ and the Hausdorff distance of $\{u_{r_k}=0\}$ from $\{x_n=0\}$ converges locally to zero, we obtain that 
$\heatop w^{r_k} \to 0$ in every compact subset of $\{x_n \neq 0\}$.
It follows that $w^0$ solves \eqref{drklkjhgfdzs2} (in the weak sense).

Finally, by the growth condition \eqref{drklkjhgfdzs} and the Liouville theorem for the heat equation in a half-space, we deduce that $w^0$ is a caloric polynomial of degree 3 on each side of $\{x_n=0\}$.
Furthermore, since $w^{0} |_{\{x_n = 0\}}=0$, both such polynomials are divisible by $x_n$.

\smallskip
 
$\bullet$ {\em Step 2}. We show that $w^0$ must be  homogeneous of degree three, i.e., it cannot have linear or quadratic terms. 
 This will follow by using that, by Proposition \ref{srjklhgfd}, we have that
\begin{equation}\label{hwiohewiohw000}
\rho \mapsto \hhh\big(\rho, (u -p_2 -p_{3}^* -\overline a_{r_k}\Psi)   \cutoff \big)   \quad \mbox{is monotone increasing}.
\end{equation}
Indeed, 
rescaling \eqref{hwiohewiohw000}, we have
\begin{equation} \label{shoehohr2}
\rho \mapsto \max_{R\in [1,1/(r_k\rho)]} R^{-M} \rho^{-5}   \tilde H\big(R\rho, w^{r_k} \cutoff(\,\cdot\,/r_k) \big)    \quad \mbox{is monotone increasing}.
\end{equation}
Now, taking  $2(3+\hat{\alpha})<M<7$ (recall that $\hat{\alpha}=1/3$), thanks to the growth control \eqref{rstdhjlkhgf}
we have 
\[
\max_{R\ge1}  R^{-M}    \tilde H\big(R, w^{r_k} \cutoff(\,\cdot\,/r_k) \big) \le C_1,
\] 
with $C_1$ independent of $k$.
Hence, by monotonicity, for all $\rho\in (0,1)$ we have 
\[
\rho^{-5}   \tilde H\big(\rho, w^{r_k} \cutoff(\,\cdot\,/r_k) \big)   \le \max_{R\in [1,1/(r_k\rho)]} R^{-M} \rho^{-5}   \tilde H\big(R\rho, w^{r_k} \cutoff(\,\cdot\,/r_k) \big)   \le C_1,
\]
and letting $k\to \infty$ we get
\begin{equation}\label{wnbowow}
 \tilde H\big(\rho, w^0\big)  \le C_1 \rho^5.
\end{equation}
Since we already know that $w^0$ must be a degree three polynomial when restricted to each side of $\{x_n=0\}$, using \eqref{wnbowow} we conclude that it must be 3-homogenous.

\smallskip
 
$\bullet$ {\em Step 3}. We now show that $w^0$ must be  of the form 
 \begin{equation}\label{drklkjhgfdzsy3895t4789t5}
w^0(x,t) \equiv  \sum_{i=1}^{n-1} a_i \,|x_n|( -3 x_i^2  + x_n^2).
\end{equation}
We first show that  $w^0$ satisfies \eqref{drklkjhgfdzs31}. Indeed, by Lemma \ref{lem:yu3y6}, we have 
 \[
\frac{1}{r^3} \int_{\{t=-r^2\}} (u-p_2) \cutoff  \,P \,G  \ge  \int_{\{t=1\}} p_3^* \,P \,G  - C(P)r
 \]
 for all $P$  odd 3-homogeneous caloric polynomial.
Therefore, using that $\Psi^{odd} = 0$, we have
\[
\int_{\{t = -1 \}} w^{r_k}  \cutoff(r_k\,\cdot\, )\,P \,G =  \int_{\{t=-r_k^2\}} \frac{u-p_2-p_{3}^*-\overline{a}_{r} \Psi}{\Theta(r_k)} \,\cutoff \,P \,G -Cr_k^4 /\Theta(r_k).
\] 
Since $\Theta(r_k)\geq c r_k^{3+\hat{\alpha}}$, it follows that $r_k^4/\Theta(r_k) \to 0$  as $k \to \infty$, and therefore
\[
\int_{\{t = -1 \}} w^{0} \,P \,G \ge0 , \quad  \mbox{for every 3-homogeneous odd caloric polynomial $P$}.
\]
This implies that $w^0$ must be even, and therefore must be of the form \eqref{drklkjhgfdzs32}.

In addition, by definition of $\overline a_r$ and $w^r$ we obtain, for all $r>0$, the following orthogonality condition holds:
\[ 
\int_{\CC_1} w_r  \,\Psi  =0.
\]
Taking the limit as $r= r_k\downarrow 0$ we obtain  \eqref{drklkjhgfdzs33}, and \eqref{drklkjhgfdzsy3895t4789t5} follows.

 \smallskip
 
\noindent $\bullet$ {\em Step 4}.
Let $\boldsymbol R_k^{(i)}\in SO(n)$ denote the rotation mapping $\{p_{2,x_k^{(i)}, t_k^{(i)}} = 0\}$ to $ \{x_n=0\}=\{p_2=0\}$  such that $\boldsymbol R_k^{(i)}-{\rm Id}$ has minimal Hilbert-Schmidt norm.
Since $p_{2,x_k^{(i)}, t_k^{(i)}} - p_2(x_k^{(i)} + r_k\cdot, t_k^{(i)}+r^2\cdot) = O(r_k^3)$ (this follows, for instance, using  Corollary \ref{rsizercube} and the triangle inequality), we have that
$|\boldsymbol R_k^{(i)}-{\rm Id}|\le Cr_k$.
Also, since $(x_k^{(i)},t_k^{(i)})\in \Sigma_{n-1}$, we have
$p_{2,x_k^{(i)},t_k^{(i)}} = p_2 \circ \boldsymbol R_k^{(i)}$.
Furthermore, thanks to Proposition \ref{srjklhgfd},
we know that 
\begin{equation}\label{hwiohewiohw}
\hhh\big(r, (u(x_{k} ^{(i)} +\cdot, t_{k}^{(i)}+ \cdot)- (p_2-p_{3}^* -\overline a_{r_k}\Psi)  \circ \boldsymbol R_k )\cutoff \big)   \quad \mbox{is monotone increasing}.
\end{equation}
Note that, since $(x_k^{(i)},t_k^{(i)})\in \Sigma_{n-1}$, we have
$p_{2,x_k^{(i)},t_k^{(i)}} = p_2 \circ \boldsymbol R_k^{(i)}$.

Define now
\[D_k^{(i)} : =  \big(r_k^2 p_2 +r_k^3( p_{3}^* +\overline a_{r_k}\Psi) \big)  (x_k^{(i)} +\,\cdot\,, t_k^{(i)} + \,\cdot\,)  - \big(r_k^2  p_{2} +r_k^3 ( p_{3}^* +\overline a_{r_k}\Psi)  \big)\circ \boldsymbol R_k^{(i)}.\]
Then
\[
\begin{split}
W_k^{(i)}(x) : &=  u(x_k^{(i)} +r_k\,\cdot\,,t_k^{(i)} + r_k\,\cdot\,)-  r_k^2 p_{2,x_k^{(i)}, t_k^{(i)}}-r_k^3(p_{3}^* +\overline a_{r_k}\Psi)\circ \boldsymbol R_k^{(i)}
\\& =  (u-p_2 -p_{3}^* -\overline a_{r_k}\Psi ) (x_k^{(i)} +r_k\,\cdot\,, t_k^{(i)} + r_k\,\cdot\,) +   D_k^{(i)}
\\
& = \Theta(r_k) w^{r_k}  (y_k^{(i)} +\,\cdot\,, s_k^{(i)} + \,\cdot\,)   + D_k^{(i)}.
\end{split} 
\]
We want to compute the limit $D_k^{(i)}/\Theta(r_k)$ for $i=1,\ldots,n-1$.
There are two alternatives:
\begin{itemize}
\item[(i)] either there exists $i \in \{1,\ldots,n-1\}$ such that $\Theta(r_k)=o\left(\|D_k^{(i)}\|_{L^2(\CC_1)}\right) $ as $k\to \infty$, 
\item[(ii)] or there exists $c>0$ such that $\|D_k^{(i)}\|_{L^2(\CC_1)} \le c\Theta(r_k)$ for all $i=1,\ldots,n-1$.
\end{itemize}
We want to reach a contradiction in both cases.

\smallskip

\noindent \emph{- Case (i)}. 
Fix $i \in \{1,\ldots,n-1\}$ such that $\Theta(r_k)=o\left(\|D_k^{(i)}\|_{L^2(\CC_1)}\right) $.
Then, up to a subsequence we have 
\[
\lim_k \frac{W_k^{(i)}}{\|D_k^{(i)}\|_{L^2(\CC_1)}}  =   \lim_k \frac{D_k^{(i)}}{\|D_k^{(i)}\|_{L^2(\CC_1)}}  =:   \tilde D^{(i)}_\infty.
\]
Note that, since $D_k^{(i)}$ belongs to a finite dimensional space, the convergence is strong and $\| \tilde D^{(i)}_\infty\|_{L^2(\CC_1)}=1$.

We want to understand the structure of $D_\infty^{(i)}$.
First, since $p_2 =\frac 1 2 x_n^2$ we have
\[
\frac{ r_k^2\, p_2 (x_k^{(i)} +\,\cdot\,, t_k^{(i)} + \,\cdot\,) - r_k^2\,  p_{2} \circ \boldsymbol R_k^{(i)}}{\|D_k^{(i)}\|_{L^2(\CC_1)}} \to  x_n (a' \cdot x') + c_1x_n,
\]
where $a'\in \R^{n-1}$ and $x' = (x_1, \dots, x_{n-1})$.  

On the other hand, 
\[
\begin{split}
&\frac{r_k^3\,(p_3^* +\overline a_{r_k}\Psi) (x_k^{(i)} +\,\cdot\,, t_k^{(i)} + \,\cdot\,)  - r_k^3 \,( p_3^* +\overline a_{r_k}\Psi) \circ \boldsymbol R_k^{(i)}}{\|D_k^{(i)}\|_{L^2(\CC_1)}} = \\
&= \frac{r_k^3\big((p_3^* +\overline a_{r_k}\Psi) (x_k^{(i)} +\,\cdot\,, t_k^{(i)} + \,\cdot\,)  - ( p_3^* +\overline a_{r_k}\Psi) \big)}{\|D_k^{(i)}\|_{L^2(\CC_1)}}+
\frac{r_k^3 \big( (p_3^* +\overline a_{r_k}\Psi) - ( p_3^* +\overline a_{r_k}\Psi) \circ \boldsymbol R_k^{(i)} \big)}{\|D_k^{(i)}\|_{L^2(\CC_1)}}
\end{split} 
\]
Hence, since $(p_3^* +\overline a_{r_k}\Psi)^{even}$ is a positive multiple of $|x_n|(x_n^2+6t)$, we get
\[
 \frac{r_k^3 \big((p_3^* +\overline a_{r_k}\Psi) (x_k^{(i)} +\,\cdot\,, t_k^{(i)} + \,\cdot\,)  -  ( p_3^* +\overline a_{r_k}\Psi) \big)}{\|D_k^{(i)}\|_{L^2(\CC_1)}} \to c_2 {\rm sign}(x_n)x_n^2  + c_3|x_n| + \big[\mbox{odd quadratic polynomial}\big].
\]
Also, recalling that $\|D_k^{(i)}\|_{L^2(\CC_1)}\ge r_k^{3+\hat{\alpha}} \gg r_k^4$ and $|\boldsymbol R_k^{(i)}-{\rm Id}|\le Cr_k$, we have
\[
\frac{r_k^3 \big((p_3^* +\overline a_{r_k}\Psi) -  ( p_3^* +\overline a_{r_k}\Psi) \circ \boldsymbol R_k^{(i)}\big)}{\|D_k^{(i)}\|_{L^2(\CC_1)}} \to 0.
\]
This proves that
\begin{equation}\label{vhkjlkjkhjghfgfx}
\tilde D_\infty^{(i)}  =   c_2 {\rm sign}(x_n)x_n^2  + c_3|x_n| + \big[\mbox{odd quadratic polynomial}\big].
\end{equation}
On the other hand, using \eqref{hwiohewiohw} and recalling the definition of $\hhh$ in \eqref{h5/2}, we obtain  that
\begin{equation} \label{shoehohr}
\rho \mapsto \max_{R\in [1,1/(r_k\rho)]} R^{-M} \rho^{-5}   \tilde H\big(R\rho, W_k^{(i)} \cutoff(\,\cdot\,/r_k) /\|D_k\|_{L^2(\CC_1)}\big)     \quad \mbox{is monotone increasing},
\end{equation}
and
by the same argument as in \eqref{wnbowow}, we get 
\[ \tilde H\big(\rho, W_k^{(i)} \cutoff(\,\cdot\,/r_k) /\|D_k\|_{L^2(\CC_1)}\big) \leq C\rho^5\qquad \Rightarrow\qquad \tilde H(\rho, \tilde D_\infty ) \leq C\rho^5.\]
Recalling \eqref{vhkjlkjkhjghfgfx} we deduce that $ \tilde D_\infty=0$, a contradiction.

\smallskip

\noindent \emph{- Case (ii)}. In this case, for each $i=1,\ldots,n-1$ we have
\[
\lim_k \frac{W_k^{(i)}}{\Theta(r_k)}  =   w^{r_k}(y_k^{(i)} + \,\cdot\,,  s_k^{(i)} + \cdot )   + \lim_k \frac{D_k^{(i)}}{\Theta(r_k)}  
\]
where $y_k^{(i)} : = x_k^{(i)}/r_k \in \overline{B_1}$ and $-C\le s_k^{(i)} = t_k^{(i)}/r_k^2\le 0$  (cp. \eqref{w3hiowh2iohe}).

Up to taking a subsequence we can define
$\tilde D^{(i)}_\infty:=\lim_k \frac{D_k^{(i)}}{\|D_k^{(i)}\|_{L^2(\CC_1)}}  $, so that $\lim_k \frac{D_k^{(i)}}{\Theta(r_k)}  = \lim_k \frac{\|D_k^{(i)}\|_{L^2(\CC_1)}}{\Theta(r_k)} \tilde D_\infty^{(i)}$ and therefore
\begin{equation}\label{hfiowhiowhiow2u4gu}
\lim_k \frac{W_k^{(i)}}{\Theta(r_k)}  =   w^{0}(y_\infty^{(i)} + \,\cdot\,,  s_\infty^{(i)} + \,\cdot \,)  + c^{(i)} \tilde D^{(i)}_\infty,
\end{equation}
 where the points $y_\infty^{(i)} \in \overline{B_1}\cap \{x_n=0\}$ are linearly independent,   $-C\le s_\infty^{(i)}\le 0$, $c^{(i)} \geq 0$, and $\tilde D^{(i)}_\infty$ is of the form~\eqref{vhkjlkjkhjghfgfx}
 (by the same argument as the one above, in Case (i)).
 
Also, thanks to \eqref{hwiohewiohw} and recalling the definition of $\hhh$ in \eqref{h5/2}, we have that
\begin{equation} \label{shoehohr2bis}
\rho \mapsto \max_{R\in [1,1/(r_k\rho)]} R^{-M} \rho^{-5}   \tilde H\big(R\rho, W_k^{(i)} \cutoff(\,\cdot\,/r_k) /\Theta(r_k)\big)     \quad \mbox{is monotone increasing},
\end{equation}
and as in Case (i) we get 
\begin{equation} \label{shoehohrbis}
\tilde H(\rho,  w^{0}(y_\infty^{(i)} + \,\cdot\,,  s_\infty^{(i)} + \,\cdot \,)  + c^{(i)} \tilde D^{(i)}_\infty) \le C\rho^{5}\qquad \forall\,i=1,\ldots,n-1. 
\end{equation}
Recall that $w^0$ is of the form \eqref{drklkjhgfdzsy3895t4789t5}, while each $\tilde D^{(i)}_\infty$ is of the form \eqref{vhkjlkjkhjghfgfx}.
Hence, since the points $y_\infty^{(i)}$ are $n-1$ linearly independent in $\{x_n=0\}$,  one can easily check that the only possibility is  $w^0\equiv 0$ and $c^{i}=0$ for all $i$.  
However, recalling \eqref{goal1}, taking the limit in \eqref{drj;ufzrhjkl} we obtain $\| w^{0} \|_{L^2(\CC_{1})} \ge \ep_1$, a contradiction.

 This completes the proof of  part (1) of the proposition.

\vspace{5pt}

\noindent
{\bf $\bullet$ Proof of (2).} This follows from a modified (and simpler) version of the argument given for part (1).  
We will prove that, for any given $\ep_2>0$, there exists $C_\circ\ge1$ such that
we have 
\begin{equation}\label{wheuihiehrie}
A(\lambda r) \le C_\circ   \lambda^{3-\ep_2/2}\, \Theta(r) \qquad \forall\,r,\lambda \in (0,1).
\end{equation}
First notice that \eqref{wheuihiehrie} implies the conclusion. Indeed, by the definition of $\Theta$ and \eqref{wheuihiehrie}, 
\[
\begin{split}
\Theta(\lambda r) &= \max_{s\in [\lambda r, 1]}  (\lambda r/s)^{3+\hat{\alpha}}A(s)  = \max\Big\{ \lambda^{3+\hat{\alpha}}\Theta (r) , \max_{s \in [\lambda r, r]} (\lambda r/s)^{3+\hat{\alpha}}A(s) \Big\}
\\&\le \max\left\{ \lambda^{3+\hat{\alpha}}\Theta (r) , C_\circ   \lambda^{3-\ep_2/2}\, \Theta(r) \right\} = C_\circ   \lambda^{3-\ep_2/2}\, \Theta(r) \leq \lambda^{3-\ep_2}\, \Theta(r),
\end{split}
\]
for $\lambda>0$ sufficiently small.

To prove  \eqref{wheuihiehrie} we reason by contradiction and compactness. 
Assume by contradiction that we have sequences $r_k\in (0,1)$ and $\lambda_k \in (0,1)$ such that
\begin{equation}\label{nwoiheoiht4}
A(\lambda_k r_k) \geq  k   \lambda_k^{3-\ep_2/2}\, \Theta(r_k),
\end{equation}
and define
\[ 
\Omega_k(\lambda) = \max_{\lambda' \in [\lambda,1]}  (\lambda/\lambda')^{3-\ep_2/2} \Theta(\lambda' r_k).
\]
Using  \eqref{nwoiheoiht4} and the definition of $\Theta$ we have
\begin{equation}\label{nwoiheoiht4b}
\lambda_k^{-3+\ep_2/2}\frac{\Theta(\lambda_k r_k)}{\Theta(r_k)} \to \infty  \qquad \text{as }k\to\infty.
\end{equation}
Note that
\begin{equation}\label{eq:theta doubling}
\text{for any $\theta_\circ\in (0,1)$, there exists a constant $C_{\theta_\circ}>0$ such that $\Theta(\theta_\circ r)\le C_{\theta_\circ} \Theta(r)$
for all $r \in (0,1)$.}
\end{equation} 
Also,
by \eqref{nwoiheoiht4b} and the definition of $\Omega_k$, there exists $\lambda'_k \in [\lambda_k,1]$ such that 
\[
\Omega_k(\lambda_k) =   (\lambda_k/\lambda'_k)^{3-\ep_2/2} \Theta(\lambda'_k r_k) = \max_{\lambda'\in [\lambda_k', 1] } (\lambda_k/\lambda')^{3-\ep_2/2} \Theta(\lambda' r_k) \ge (\lambda_k)^{3-\ep_2/2} \Theta(r_k).
\]
Thanks to these two facts we deduce that $\lambda'_k \to 0$.
Also, by the maximality of $\lambda_k'$, 
$$
\Omega_k(\tau) =   (\tau/\lambda'_k)^{3-\ep_2/2} \Theta(\lambda'_k r_k)  \qquad \forall\,\tau \in [\lambda_k,\lambda_{k}'],
$$
 and by the definitions of $\Theta$ and $\Omega_k$ we  deduce that
\begin{equation}\label{wneiohoiehte1}
A(\lambda'_k r_k) =  \Theta(\lambda'_k r_k)=\Omega_k(\lambda_k').
\end{equation}
We now define
\[
W^{k}: = \frac{(u-p_2-p_{3}^*-\overline{a}_{\lambda_k r_k}  \Psi)_{\lambda'_k r_k }}{ \Omega_k(\lambda'_k) }.
\]
By \eqref{wneiohoiehte1} we have
\[
\|W^{k}\|_{L^2(\CC_1)} \ge 1.
\]
Also, by construction,
\begin{equation}\label{rstdhjlkhgf2443}
\big\|W^{k} (R\,\cdot\,, R^2\, \cdot\, ) \big\|_{L^2 (\CC_1)} \le  
\begin{cases}
 CR^{3-\ep_2/2}  \quad &\mbox{for }R\in [1, 1/\lambda_k']\\
 C(1/\lambda_k')^{3-\ep_2/2} (R/\lambda_k')^{3+\alpha}  \quad &\mbox{for }R\in [1/\lambda_k' , 1/(\lambda_k' r_k)].\\
 \end{cases}
\end{equation}
Hence, repeating the same reasoning as in Step 1 above, we find that $W^{k} \to \widetilde W$ where
\[
\|\widetilde W\|_{L^2(\CC_1)} \ge 1,\qquad \heatop \widetilde W = 0  \quad   \mbox{in } \{ x_n \neq 0\},
\qquad 
\widetilde W = 0  
\quad  \mbox{on } \{ x_n =0\}.
\]
Furthermore, thanks to \eqref{rstdhjlkhgf2443} and the fact that $\lambda_k'\to 0$, $\widetilde W$ satisfies the  sub-cubic growth
\[
\big\| \widetilde W (R\,\cdot\,, R^2\, \cdot\, ) \big\|_{L^2 (\CC_1)} \le  CR^{3-\ep_2/2}  \quad \mbox{for all  }R\ge1.
\]
This implies that $\widetilde W$ is a caloric polynomial of degree 2 on each side of the hyperplane $\{x_n=0\}$.
On the other hand, using \eqref{rstdhjlkhgf2443} again, we can repeat the argument in Step 2 above to show that $ \tilde H\big(\rho, \widetilde W\big)  \le C_1 \rho^5$ for all $\rho>0$ (cp. \eqref{wnbowow}).
This implies that $\widetilde W\equiv 0$, a contradiction.
\end{proof}

\section{An estimate of order $3+\beta$ at ``most'' singular points} \label{sec:E3B}

In this section we prove that, at ``most'' points of $\Sigma^*$, we have a decay of order $3+\beta$ as in the following definition.

\begin{definition}
\label{def:sigma diamond}
Denote by $\Sigma^{\diamond}$ the set of points $(x_\circ, t_\circ) \in \Sigma^*$  such that, for some positive numbers $\beta$ and~$r_\circ$  (which may depend on the point), we have 
\begin{equation}\label{ejeijihhs}
\big\| \big(u(x_\circ+ \,\cdot\,, t_\circ+\,\cdot\,)-p_{2,x_\circ, t_\circ}-p_{3,x_\circ, t_\circ}^*\big)_r \big\|_{L^2(\CC_1)} \le  r^{3+\beta}  \quad \mbox{for all }r\in (0,r_\circ). 
\end{equation}
\end{definition} 

We want to show the following:

\begin{proposition}\label{prop:heiht4367}
Assume that  $u:B_1\times(-1,1) \to [0,\infty)$ is a solution of  \eqref{eq:UPAR1}. 
Then,
\[{\rm dim}_{\mathcal H}\big(\pi_x(\Sigma^*\setminus \Sigma^{\diamond})\big)\leq n-2.\]
\end{proposition}

We will prove Proposition \ref{prop:heiht4367} by carefully exploiting Proposition \ref{prop_dicho}, combined with some delicate covering arguments.

Let $\delta>0$ be given by Corollary \ref{cor-monot}, and $\ell\in \mathbb N$. Also, let $\ep>0$ to be fixed later.
 Given $(x_\circ ,t_\circ)\in \Sigma^{*}_{\delta,\ell}$,  let $\kappa_\circ =\lfloor \log_2 r_\circ \rfloor+ 1$. 
 For $j \ge k_{x_\circ, t_\circ} $ we define
\begin{equation}\label{defNbaihia}
N_{\ep}(x_\circ, t_\circ,  j) :=  \# \big\{ i \in \{\kappa_\circ,\ldots,j\} \ : \mbox{ \eqref{eq:Pr} applied to  $u(x_\circ +\, \cdot\,,\ t_\circ + \,\cdot\, )$ does \emph{not} hold at scale $r = 2^{-i}$}\big\},
\end{equation}
and
\begin{equation}\label{omega-ep}
\omega_\ep(x_\circ, t_\circ) : = \liminf_{j \to \infty} \frac {N_{\ep}(x_\circ, t_\circ,  j) }{j}.
\end{equation}
Notice that $\omega_\ep$ is nonincreasing in $\ep$.

The goal of this section will be to prove the following two properties:
\begin{enumerate}
\item[(i)] if $\omega_\ep(x_\circ, t_\circ) >0$ for some $\ep>0$, then $(x_\circ,t_\circ)\in \Sigma^{\diamond}$;

\vspace{2mm}

\item[(ii)] for any $\ell \in \mathbb N$, the $\pi_x$-projection of the set $\{(x_\circ, t_\circ)\in \Sigma^{*}_{\delta,\ell}\,:\, \omega_\ep(x_\circ, t_\circ) =0 \ \forall\,\ep>0\}$  has Hausdorff dimension at most $n-2$.
\end{enumerate}

The implication (i) will follow by Proposition \ref{prop_dicho}(2), thanks to the  definition of $\omega(x_\circ, t_\circ)$.  
It is the content of the following:

\begin{lemma}\label{whioehoirhe}
Assume that $\omega_\ep(x_\circ, t_\circ) >0$ for some $\ep>0$.
Then there exist $r_\circ,\beta>0$  such that \eqref{ejeijihhs} holds. 
\end{lemma}

\begin{proof}
For notational simplicity we assume that  $(x_\circ, t_\circ) = (0,0)$.
By assumption there exists $N_\circ \in \mathbb N$ such that 
\[
\omega_\ep(0,0) : = \liminf_{ j \to \infty} \frac {N_\ep(0,0, j) }{j} \ge \frac{2}{N_\circ}.
\]
Hence, there exists $j_\circ$ such that 
\[
\frac {N(0,0, j) }{j} \ge  \frac{1}{N_\circ}\qquad \forall\,j \geq j_\circ.
\]
This means that there exists a set $\mathcal I_\ep\subset \{\kappa_\circ,\ldots\}$ such that 
\begin{equation}\label{haihiahihavaugvuba}
\mbox{ \eqref{eq:Pr} does not hold at scales $r = 2^{-i}$   for all $i \in \mathcal I_{\ep}$,  }\qquad \#  \big\{ \mathcal I_{{\ep}}  \cap \{\kappa_\circ,\dots, j\}  \big\} \ge \frac{j}{N_\circ}
\quad \forall\,j \ge j_\circ.
\end{equation}
In particular, given $\ep_1>0$, up to taking $j_\circ$ larger if needed (depending on $\ep_1$), it follows from Proposition \ref{prop_dicho}(1) that 
\begin{equation}\label{ejiowhoih1}
A(2^{-i}) \le \ep_1 \Theta(2^{-i}), \qquad \mbox{whenever } i \in  \mathcal I _{{\ep}}.
\end{equation}
Also, given $\ep_2>0$, it follows from Proposition \ref{prop_dicho}(2) that
\begin{equation}\label{ejiowhoih2}
A(\lambda 2^{-i}) \le   \lambda^{3-\ep_2}\, \Theta(2^{-i}),\quad   \mbox{whenever }  i \not\in  \mathcal I _{{\ep}} \quad \mbox{and} \quad \lambda\in (0,\lambda_{\ep_2}).
\end{equation}
We now start to fix the different parameters.

First, we set $\beta:=\frac{1}{8N_\circ}$ and choose $\ep_2>0$ small enough so that
\begin{equation}\label{ehwiohwh}
(3-\ep_2) \bigg(1-\frac1{N_\circ}\bigg) +\frac{10}{3N_\circ} \ge 3+2\beta.
\end{equation}
Then, we fix $M\in \mathbb N$  such that $2^{-M}< \lambda_{\ep_2}$ and we note that, by the definition of $A$ (see Proposition~\ref{prop_dicho})
\[
i\in \mathcal I_\ep \qquad \Rightarrow \qquad A(\lambda 2^{-i})\le \lambda^{-(n+2)} A(2^{-i}) \le \lambda^{-(n+2)} \ep_1\Theta(2^{-i}).
\]
Hence, choosing $\ep_1>0$ small so that $2^{(n+2)M}\ep_1\le 2^{-\frac{10}{3}M}$, it follows from the bound above and the definition of $\Theta$ (see Proposition~\ref{prop_dicho} and recall that $\hat{\alpha}=1/3$) that
\begin{equation}\label{eq:I eps}
i\in \mathcal I_\ep \qquad \Rightarrow \qquad \Theta(2^{-i-M})\le  2^{-\frac{10}{3}M} \Theta(2^{-i}).
\end{equation}
Now, to conclude the proof, given $ k \in \mathbb N$ such that $j_k:=\kappa_\circ+(kN_\circ+1) M \geq j_\circ$, we note that there exists  $m_k \in\{1, \dots, M \}$   such that 
\[
\#  \big( \mathcal I _{{\ep}}  \cap   (M \mathbb Z +\kappa_\circ+m_j)  \cap \{\kappa_\circ,\dots, j_k\}   \big)  \ge k.
\]
Hence, if we set $ \tilde {\mathcal I}_k : = \big( \mathcal I _{{\ep}}  \cap   (M \mathbb Z +\kappa_\circ+m_k)   \cap \{\kappa_\circ,\ldots, j_k\}   \big) $,
it follows from \eqref{eq:I eps} that 
\[
i\in \tilde {\mathcal I}_k \qquad \Rightarrow \qquad \Theta(2^{-i-M})\le  2^{-4M} \Theta(2^{-i}).
\]
On the other hand, by Proposition \ref{prop_dicho}(2) we have
\[
i \in  \bigl((M \mathbb Z +\kappa_\circ+m_k)\cap \{\kappa_\circ,\ldots, j\}\bigr) \setminus \tilde {\mathcal I}_k\qquad \Rightarrow \qquad \Theta(2^{-i-M})\le    2^{-(3-\ep_2)M}\Theta(2^{-i}).
\]
Combining these two informations and recalling that $\#\tilde {\mathcal I}_k \ge  k$, it follows from \eqref{ehwiohwh} that 
\begin{multline*}
\Theta(2^{-(\kappa_\circ+m_k+kN_\circ M)}) \le 2^{-\left(\frac{10}3 M \,\#\tilde {\mathcal I}_k +(3-\ep_2)[kN_\circ M - \#\tilde {\mathcal I}_k] \right)}\Theta(2^{-(\kappa_\circ+m_k)})\\
\leq 2^{-[10/3 +(3-\ep_2)(N_0-1)]kM}\Theta(2^{-(\kappa_\circ+m_k)})\leq  2^{-(3+2\beta)kN_\circ M}\Theta(2^{-(\kappa_\circ+m_k)}).
\end{multline*}
Since $m_k \in \{0,\ldots,M-1\}$, using \eqref{eq:theta doubling}  we get
$$
\Theta(2^{-(\kappa_\circ+M+k N_\circ M)}) \le C_M 2^{-(3+2\beta)kN_\circ M}\Theta(2^{-\kappa_\circ}) \qquad \forall\,k \in \mathbb N,
$$
for some constant $C_M$ depending only on $M$.
Thus, by \eqref{eq:theta doubling} and the definition of $\Theta$ we get
\[
A(r)\leq \Theta(r) \le C_{N_0,M,\kappa_0}r^{3+2\beta} \quad \mbox{for }r\in (0,2^{-\kappa_\circ}),
\]
and the lemma follows by choosing $r_\circ$ sufficiently small.
\end{proof}

We now want to show property (ii).
This will follow from the following  GMT results.

\begin{lemma}\label{lemvygyguy}
Let $m\le n$ be positive integers. Given $\alpha\in (0,1)$ there exists $\rho_\circ>0$ so that,
for any $\rho\leq \rho_\circ$, there exist small constants $\ep_{\alpha,\rho},\omega_{\alpha,\rho}>0$ such that  the following statement holds for $\ep\leq \ep_{\alpha,\rho}$ and $\omega\leq \omega_{\alpha,\rho}$. 

Let $E\subset B_{1/2}(z)\subset \R^n$ for some $z \in \R^n$. 
For every $x\in E$ and $j\ge 1$, define
\begin{equation}\label{def-Nm}
\begin{split}
N^{\ep, \rho}_m(E,x, j) & : = \# \big\{   i \in \{0,\dots, j\} \ :\  E\cap B_{\rho^{i}} (x) \subset y+ L + B_{\rho^{i}\ep}
\\ &\hspace{2cm} \mbox { for some $y\in \R^n$ and $L\subset\R^n$ linear subspace with ${\rm dim}(L) \le m$}\big\}.
\end{split}
\end{equation}
Assume that, for some $j \ge 1$, we have 
\begin{equation}\label{hgiljkhugyf}
N^{\ep, \rho}_m(E, x, j) \ge (1-\omega) j \quad \mbox{ for all } x\in E.
\end{equation}
Then $E$ can be covered by $\rho^{-(m+\alpha)j}$ balls of radius $\rho^j$.
\end{lemma}

To prove it, we will need the following simple result (see {\cite[Lemma 7.2]{FRS}}):

\begin{lemma}
\label{LOC.lem:GMT2}
Let $B_r(x)\subset \R^n$ be an open ball, and $L$ be a $m$-dimensional linear subspace (not necessarily passing through $x$).
Let $\beta_1>m$. 
Then there exists $\hat\tau=\hat\tau(m,\beta_1)>0$ such that the following holds.

Let $F\subset \R^n$ satisfy
$$
F\subset B_r(x) \cap \{y\,:\,{\rm dist}(y,L)\leq \tau r\},\qquad \text{for some  $0<\tau\leq \hat \tau$, $x \in \R^n$, $r>0$.}
$$
Then $F$ be covered with $\gamma^{-\beta_1}$ balls of radius $\gamma r$ centered at points of $F$, where $\gamma:=5\tau$.
\end{lemma}

We now prove Lemma \ref{lemvygyguy}.

\begin{proof}[Proof of Lemma \ref{lemvygyguy}]
We will prove the following:

\vspace{2mm}

\noindent{\em Claim}.  Let $\ell\geq 1$ and $0 \leq k \leq \ell$. If $E\subset B_{1/2}(z)$ satisfies
\begin{equation}\label{huiahoiuago}
N^{\ep, \rho}_m(E, x; \ell)  \ge k \quad \mbox{ for all } x\in E,
\end{equation}
then $E$ can be covered by $\rho^{-(n+\alpha) (\ell-k)} \rho^{ -(m+\alpha) k}$ balls of radius $\rho^\ell$. 

\vspace{2mm}

To prove the Claim, we shall proceed by induction on $\ell \geq 1$.
However, we remark first that the case $k=0$ and $\ell$ arbitrary, we simply use that  $E \subset B_{1/2}$ to deduce that $E$ can be covered by  $C_n \rho^n \leq \rho^{-(n+\alpha)}$ balls of radius $\rho$ (for $\rho$ small).
Hence, in the induction procedure, we can assume that $k\geq 1$.

\smallskip

\noindent {\em $\bullet$ The case $\ell=1$.}
In this case the only option is $\ell = k =1$.
Then  the assumption $N^{\ep, \rho}_m(E,x, 1) \ge1$ implies 
 $E \cap B_{1/2}(x)\subset y+ L + B_{\ep}$, so Lemma \ref{LOC.lem:GMT2} with $\tau=\rho/5$ and $r=1/2$ yields that $E$ can be covered by $\rho^{-(m+\alpha)}$ balls of radius $\rho$ (provided $\rho$ is small enough so that $\tau\leq \hat\tau$, and then $\ep$ is chosen sufficiently small so that $\ep\leq \tau r$).

 \smallskip
 
 \noindent {\em $\bullet$ The inductive step.}
Assume that the claim is true for $\ell-1\ge 1$ and for all $k= 1,\ldots \ell-1$.
We now prove it for $\ell$, and $k =1,\ldots, \ell$. 
There are two cases to consider. 

\smallskip

\noindent \emph{-  Case 1.} Assume there exists $x\in E$ such that 
\begin{equation}\label{naihaioh}
E\cap B_{1}(x) \subset x+  L + B_{\ep }.
 \end{equation}
 Then, since $E\subset B_{1/2}(z)$, as before we can apply Lemma \ref{LOC.lem:GMT2} to deduce that $E$ can be covered  (provided $\rho$ and $\ep$ are chosen small enough) by $\rho^{-(m+\alpha)}$ balls of radius $\rho/2$. 
 Let us call these balls $\{B_{\rho/2}(z_q)\}_{q \in \mathcal I}$, $\#\mathcal I\leq \rho^{-(m+\alpha)}$. 
 
 Now, for any $q \in \mathcal I$ we define $E_q : = \frac1 \rho\big( E\cap B_{\rho/2}(z_q)\big) \subset B_{1/2}(z_q/\rho)$ and we observe that, thanks to \eqref{huiahoiuago}, 
 \[
N^{\ep, \rho}_m(E_q, x; \ell-1) \ge k-1\qquad \forall\, q \in \mathcal I.
 \]
Therefore, by induction hypothesis, each set $E_q$ can be covered by $ \rho^{-(n+\alpha) (\ell-k)} \rho^{ -(m+\alpha) (k-1)}$ balls of radius $\rho^{\ell-1}$. 
This implies that
the union of all these balls multiplied by a factor $\rho$ covers $E$, and  the total number of such balls is $\rho^{-(n+\alpha) (\ell-k)} \rho^{ -(m+\alpha) (k-1)}\#\mathcal I \leq  \rho^{-(n+\alpha) (\ell-k)} \rho^{ -(m+\alpha)}$, as desired.

\smallskip

\noindent \emph{-  Case 2.} Assume that for {\em none} of the points  $x\in E$  \eqref{naihaioh} holds. 
In particular this implies that $k\le \ell-1$ and therefore, thanks to  \eqref{huiahoiuago},  for all $x\in E$ we have
\begin{equation}\label{haoihaoiha}
k \le N^{\ep, \rho}_m(E, x; \ell)  = \# \big\{   i \in \{1,\dots, \ell\} \ :\  E\cap B_{\rho^{i}} (x) \subset x + L + B_{\rho^{i}\ep } \big\}.
\end{equation}
(Note that, in the formula above, $i\ge 1$ instead of $i\ge 0$ as in the definition of $N^\ep$.)

Now, let us cover the ball $B_{1/2}(z)$ by $\rho^{-(n+\alpha)}$ balls  $\{B_{\rho/2}(z_q)\}_{q \in \mathcal I}$ of radius $\rho/2$. 
In particular these balls cover $E$. 
As in Case 1 we define $E_q : = \frac1 \rho\big( E\cap B_{\rho/2}(z_q)\big) \subset B_{1/2}(z_q/\rho)$,
and it follows from \eqref{haoihaoiha}  that 
\[
N^{\ep, \rho}_m(E_q, x; \ell-1) \ge k\qquad \forall\, q \in \mathcal I.
\]
By induction hypothesis, each set $E_q$ can be covered by $\rho^{-(n+\alpha) (\ell-1-k)} \rho^{ -(m+\alpha) k}$ balls of radius $\rho^{\ell-1}$,
and therefore $E$ can be covered by $\rho^{-(n+\alpha) (\ell-1-k)} \rho^{ -(m+\alpha)k}\#\mathcal I \leq  \rho^{-(n+\alpha) (\ell-k)} \rho^{ -(m+\alpha)}$ balls of radius $\rho^\ell$.
Thus, the Claim is proved.

\smallskip

Using the Claim, we finally show the result.
Indeed, setting $\ell=j$ and $k=\lfloor (1-\omega)j\rfloor$, we know that for any $\hat\alpha>0$ the set $E$ can be covered by $\rho^{-(n+\hat\alpha)(j-\lfloor (1-\omega)j\rfloor)-(m+\hat\alpha)\lfloor (1-\omega)j\rfloor} \leq \rho^{-(m+\hat\alpha+(n-m)\omega)j-1}$ balls of radius $\rho^j$.
Choosing $\hat\alpha< \alpha-(n-m)\omega$, the result follows by ensuring that $\rho$ is sufficiently small.
\end{proof}

\begin{lemma}\label{prop:GMT5}
Let $E\subset \R^n$. 
Given $\ep,\omega>0$ small and $m\leq n$, 
assume that
\begin{equation}\label{hahhajhjha}
\limsup_{j \to \infty} \frac{N^{\ep,1/2}_{m}(E,x,j)}{j} \ge  1-\omega \qquad \forall\,x \in E,
\end{equation}
where $N^{\ep,1/2}_m$ is defined as in  \eqref{def-Nm}.
Then
\[
{\rm dim}_{\HH} (E) \le m+\alpha,
\]
where $\alpha(n, \ep, \omega)\to 0$ as $(\ep, \omega)\to(0,0)$.
\end{lemma}

\begin{proof}
Up to taking countable unions of we may assume that $E\subset B_{1/2}(z)$ for some $z \in \R^n$.

We being by observing that, as a consequence of \eqref{hahhajhjha} and the definition $N^{\ep,1/2}_m$, 
\begin{equation}\label{hahhajhjha2}
\limsup_{j \to \infty}  \frac{N^{\ep,2^{-M}}_{m}(E,x,j)}{j} \ge  1-M\omega \qquad \forall\,M \geq 1.
\end{equation}
Hence we fix $\rho:=2^{-M}$ with $M$ sufficiently large so that $2^{-M}\leq \rho_\circ$, where $\rho_\circ$ is given by Lemma \ref{lemvygyguy}.

Now, given $k \in \N$ large, we define
\[
j(x,k) : = \min \{\,  j \ge k \ : \ N^{\ep,\rho}_{m}(E,x, j) \ge (1-2M\omega)j\, \}.
\]
Thanks to \eqref{hahhajhjha2}, we can partition $E$ as 
\[
E = \bigcup_ {\ell =k} ^\infty E_\ell  \qquad \mbox{ with} \quad E_\ell : = \{x\in E \  : \ j(x,k) = \ell\}.
\]
Hence, given $\alpha>0$, provided $\ep,\omega$ are small enough we can apply 
Lemma \ref{lemvygyguy} to deduce that   $E_\ell\subset E \subset B_{1/2}(z)$ can be covered by $\rho^{-(m+\alpha/2) \ell}$ balls of radius $\rho^{\ell}$. 
Therefore, $E$   can be covered by balls of radius  $r_i = \rho^{\ell_i}$ with $\ell_i\ge k$ and $\#\{ i : \ \ell_i = \ell \}\le \rho^{-(m+\alpha/2) \ell}$. 
This implies that
\[
\mathcal H^{m+\alpha}_{2^{-k}}(E)\leq C_{m+\alpha}\sum_i  r_i^{m+\alpha}  \le  C_{m+\alpha}\sum_{\ell =k}^\infty     \rho^{-(m+\alpha/2) \ell} (\rho^\ell)^{m+\alpha} = C_{m+\alpha}\sum_{\ell=k}^\infty (\rho^{\alpha/2})^\ell ,
\]
and the right hand side can be made arbitrarily small by choosing $k$ large. This proves that
$\HH^{m+\alpha}(E) =0$,  and therefore  ${\rm dim}_\HH (E)\le m+\alpha$, as desired.

\end{proof}

We will also need the following modification of Lemma \ref{prop:GMT5}.

\begin{proposition}\label{prop:GMT6}
Let $E\subset \R^n\times \R$.
Given $\ep,\omega>0$ small, $m\leq n$, $(x,t)\in E$, and $j\geq1$, define
\begin{equation}
\begin{split}
\widetilde N^{\ep, \rho}_m(E,x,t, j) : = \# \big\{ &  i \in 0,1,\dots, j \ :\   \pi_x\big( E\cap ( \overline{B_{\rho^{i}}} (x)\times (-\infty, t]) \big) \subset  x + L + \overline {B_{\rho^{i}\ep}}
\\ &\hspace{3cm} \mbox { for some linear subspace $L\subset\R^n$ with ${\rm dim}(L)  = m$}\big\}.
\end{split}
\end{equation}
Assume that
\begin{equation}\label{hahhajhjha33}
\limsup_{j \to \infty} \frac{\widetilde N^{\ep,1/2}_{m}(E,x,t,j)}{j} \ge  1-\omega \qquad \forall\,(x,t)\in E.
\end{equation}
Then
\[
{\rm dim}_{\HH} \big(\pi_x(E)\big) \le m+\alpha,
\]
where $\alpha(n, \ep, \omega)\to 0$ as $(\ep, \omega)\to(0,0)$.
\end{proposition}

\begin{proof}
Fix $\alpha>0$, and let  $\rho=2^{-M}$  with $M$ chosen large enough so that $\rho^{\alpha/4} \le \frac 1 4$ and $\rho\leq \rho_\circ$, with $\rho_\circ$ given by Lemma \ref{lemvygyguy}.
Up to taking countable unions, we may assume that $\pi_x(E)\subset \overline{B_{1/4}}(z)\times [-1,1]$.
Also, as in the proof of the previous lemma, \eqref{hahhajhjha33} implies
\begin{equation}\label{hahhajhjha332}
\limsup_{j \to \infty} \widetilde N^{\ep,\rho}(E,x,j) \ge  1-M\omega.
\end{equation}
Hence, given $k \in \N$ large, we define
\[
j(x,k) : = \min \{\,  j \ge k \ : \widetilde N^{\ep,\rho}_m(E,x, j) \ge (1-2M\omega)j\, \},
\]
and we partition
\[
E = \bigcup_ {\ell =k} ^\infty  E_\ell \qquad \mbox{with} \quad E_\ell : = \{x\in E \  : \ j(x,k) = \ell\}.
\]
Since
\[
\widetilde N^{\ep,\rho}_m(E_\ell ,x, t,  \ell)\ge (1-2N\omega)\ell \qquad \forall\,(x,t)\in E_\ell,
\]
for any $(x,t)\in E_\ell$ there exist a subset $\sigma_{x,t} \subset \{0,\ldots, \ell\}$ satisfying 
\[
\#\sigma_{x,t} \ge (1-2N\omega)\ell
\qquad
\text{and}
\qquad
 \pi_x\big( E\cap ( \overline{B_{\rho^{i}}} (x)\times (-\infty, t]) \big) \subset x + L_{x,i} + \overline {B_{\rho^{i}\ep}} \quad \forall \,i \in \sigma_{x,t},
\]
where  $L_{x,i} \subset \R^n$ is some liner subspace with ${\rm dim}(L) = m$.

Let us further decompose each $E_\ell$ as follows:
\[
E_\ell = \bigcup_{\substack{\ell \in \N \\\sigma\subset \{0,\ldots, \ell\} }} E_{\ell, \sigma}\qquad \mbox{where}\quad E_{\ell, \sigma} : = \{ (x,t)\in E_\ell \ : \ \sigma_{x,t} = \sigma\},
\]
and let us show that 
\begin{equation}\label{gihgygigfydyf1}
N^{\ep/\rho,\rho}_m\left(\pi_x(E_{\ell,\sigma}) ,x, t, \ell\right)\ge (1-2M\omega)\ell -1 \ge (1-3M\omega)\ell \qquad \forall \,(x,t)\in E_{\ell, \sigma},
\end{equation}
provided that $\ell$ is large enough, where $N^{\ep,r}_m$ is defined as in  \eqref{def-Nm} .

Indeed, by the definition of $E_{\ell, \sigma}$ we have
\begin{equation}\label{hgycfesrdftgj7}
 \pi_x\big( E_{\ell, \sigma} \cap ( \overline{B_{\rho^{i}}} (x)\times (-\infty, t]) \big) \subset  x + L_{x,t,i} + \overline {B_{\rho^{i}\ep}} \quad \forall \, (x,t)\in E_{\ell, \sigma} , \ \forall  \,i \in \sigma.
\end{equation}
The first important observation is that \eqref{hgycfesrdftgj7} holds also for all $(\bar x, \bar t)$ in the closure $\overline{E_{\ell, \sigma}}$ of $E_{\ell, \sigma}$, provided that we define  $L_{\bar x, \bar t,i}$  for $(\bar x, \bar t) \in \overline{E_{\ell, \sigma}} \setminus E_{\ell, \sigma}$  as a limit of hyperplanes $L_{x,t,i}$ for $(x,t)\in E_{\ell, \sigma}$ converging to $(\bar x, \bar t)$. 

Now, given $(x,t)\in E_{\ell, \sigma}$ and $i\in \{1,\ldots, \ell\}$ such that $(i-1)\in \sigma$, we choose
$\bar t : =  \max \pi_t\big(\overline{B_{\rho^{i}}(x)}\times [-1,1] \cap \overline{E_{\ell, \sigma}} \big)$ and let $(\bar x, \bar t)\in \overline{B_{\rho^{i}}(x)}\times [-1,1] \cap \overline{E_{\ell, \sigma}} $ a corresponding point of ``maximal time''.
Since $(i-1)\in \sigma$, then \eqref{hgycfesrdftgj7} holds for the point $(\bar x, \bar t)$ with $i$ replaced by $i-1$.  Thus, using that $B_{\rho^i(x)}\subset B_{\rho^{i-1}}(\bar x)$, we get
\[
 \ \pi_x\big( E_{\ell, \sigma} \big) \cap B_{\rho^i}(x)  \subset  \pi_x\left( \overline{E_{\ell, \sigma}} \cap ( \overline{B_{\rho^{i-1}}} (\bar x)\times (-\infty, \bar t]) \right) \subset  \bar x + L_{\bar x,t,i} + \overline {B_{\rho^{i-1}\ep}}.
 \]
In other words, for any given $(x,t)\in E_{\ell, \sigma}$ and for all scales $i \in \{1,2,\dots, \ell\}$  such that  $(i-1)\in \ \sigma$ (a set of cardinality at least $(1-3M\omega)\ell$),  we have
\[
\pi_x\big( E_{\ell, \sigma} \big) \cap B_{\rho^i}(x) \subset y + L + \overline {B_{\rho^{i}\ep/\rho}},
\]
for some $y\in \R^n$ and a linear subspace $L\subset \R^n$ of dimension $m$.
Thus \eqref{gihgygigfydyf1} follows.

To conclude, choosing $\omega$ and $\ep$ small enough and applying Lemma \ref{lemvygyguy}, we deduce that   $\pi_x(E_{\ell,\sigma}) \subset B_{1/2}(z)$   can be covered by $\rho^{-(m+\alpha/4) \ell}$ balls of radius $\rho^{\ell}$. 
Therefore, since the set of all possible choices of $\sigma \in \{0, \ldots, \ell\}$ has $2^{\ell+1}$ elements, we  see that each set $\pi_x(E_{\ell})$ can be covered by $2^{\ell+1}\rho^{-(m+\alpha/4) \ell}$ balls of radius  $\rho^{\ell}$.
Since $\rho$ was chosen so that $\rho^{\alpha/4} \ge \frac 1 4$, this implies that  $\pi_x(E_{\ell})$ can be covered by $\rho^{-(m+\alpha/2) \ell}$ balls of radius $\varrho^\ell$. 

Adding these bounds over $\ell \geq k$ and letting $k \to \infty$, we conclude that $\pi_x(E_{\ell})$ has zero $\HH^{m+\alpha}$-measure (cp. proof of Lemma \ref{prop:GMT5}).
\end{proof}

We can finally prove Proposition \ref{prop:heiht4367}.

\begin{proof}[Proof of Proposition \ref{prop:heiht4367}]
Let $\delta>0$ be given by Corollary \ref{cor-monot}, and split $\Sigma^*=\cup_\ell \Sigma^*_{\delta,\ell}$. It suffices to show that ${\rm dim}_{\mathcal H}(\Sigma^*_{\delta,\ell}\setminus \Sigma^{\diamond})\leq n-2$.

For any $\ep>0$ and any $(x_\circ,t_\circ)\in \Sigma^*_{\delta,\ell}$, let $\omega_\ep$ be given by \eqref{omega-ep}.
By Lemma \ref{whioehoirhe}, if $\omega_\ep(x_\circ,t_\circ)>0$ then $(x_\circ,t_\circ)\in \Sigma^{\diamond}$.
On the other hand, Proposition \ref{prop:GMT6} applied with $m=n-2$ and $E=\Sigma^*_{\delta,\ell}$ implies that ${\rm dim}_{\mathcal H}(\{(x_\circ,t_\circ)\in \Sigma^*_{\delta,\ell}:\omega_\ep(x_\circ,t_\circ)=0\ \text{for all}\,\ep>0\}) \leq n-2$.
Thus ${\rm dim}_{\mathcal H}(\Sigma^*_{\delta,\ell}\setminus \Sigma^{\diamond})\leq n-2$.
\end{proof}

\section{Enhanced decay towards polynomial Ansatz} \label{sec-13}

In this section we will show that, once we have an estimate of order $3+\beta$ for some $\beta>0$ at a singular point, then we can actually prove a $C^\infty$ estimate.

Assume with not loss of generality that $(0,0)\in\Sigma^{\diamond}$ (see Definition~\ref{def:sigma diamond}).
Then there exist $\beta\in (0,\frac18)$ and $C_\circ>0$ such that
\begin{equation}\label{wiohehth}
\big\| (u-p_2-p_3^*)_r \big\|_{L^2(\CC_1)} \le C_\circ r^{3+2\beta}  \quad \mbox{for all }r\in (0,1). 
\end{equation}
Moreover, we may choose coordinates so that $p_2= \frac 1 2 x_n^2$ and $p_3^* = a|x_n| (x_n^2+6t) +p_3$.

In the following result, and throughout this section, we denote
\begin{equation}\label{def-OmegabetaR}
\Omega^\beta_{R}  : = \left\{ (x,t) \in \R^n\times \R\,:\,|x|< R (- t)^{\frac{1}{2+\beta}}, \   t \le0\right\}.
\end{equation}
Recall that $\CC_r: = B_r\times (-r^2,0)$ denotes a parabolic cylinder.
Also, , to avoid unnecessary  parentheses, given $z \in \R$ we denote $z_+^2=(z_+)^2$ (and analogously for $z_-^2$).

\begin{theorem}\label{thm:euihwbegs}
Let $u:B_1\times(-1,1) \to [0,\infty)$ be a solution of  \eqref{eq:UPAR1} and assume $(0,0)\in\Sigma^{\diamond}$.
Then there exists $r_\circ>0$, depending only on $n$, $\beta$, and $C_\circ$, such that inside 
\[V: =  \Omega^{\beta}_1\cap \CC_{r_\circ}\]
the positivity set $\{u>0\}$ has two connected components and   $u= u^{(1)} + u^{(2)}$ in $V$, where $u^{(i)}$ are two solutions of the parabolic obstacle problem inside $V$ with disjoint supports satisfying 
\begin{equation}
\label{eq:u12}
\frac{u^{(1)}(rx)}{r^2}\to \frac12 (x_n)_+^2 \quad \text{and}\quad \frac{u^{(2)}(rx)}{r^2}\to \frac12 (x_n)_-^2 \qquad \text{as }r \to 0.
\end{equation}
\end{theorem}


This result will allow us to break the function $u$ ---which has a singular point at the origin--- into two separate functions $u_1$ and $u_1$ that  behave \emph{as if} the origin was a \emph{regular} point for each of them.
This is because, with the parabolic scaling, the domain $\Omega^\beta_R$ near the origin is almost equivalent to a full cylinder.

To prove Theorem \ref{thm:euihwbegs} we need a barrier argument to show that the contact set $\{u=0\}$ splits $\Omega^{\beta}_1\cap \CC_{r_\circ}$ into two disconnected pieces. 
This barrier is constructed in the following:

\begin{lemma}\label{lem:784hbtr6}
Let $u:B_1\times(-1,1) \to [0,\infty)$ be a solution of  \eqref{eq:UPAR1}, and assume that $(0,0)\in\Sigma^{\diamond}$, i.e., $u$ satisfies \ref{wiohehth}.
Choose coordinates so that $p_2= \frac 1 2 x_n^2$ and $p_3^* = a|x_n| (x_n^2+6t) +p_3$.

Then there exists $r_\circ>0$, depending only on $n$, $\beta$, and  $C_\circ$, such that 
\[
 \{ x_n +p_3/x_n =0 \} \cap  \Omega^{\beta}_1\cap \CC_{r_\circ} \subset \{u=0\}. 
\]
(Recall that $p_3(x,t)$ is an odd  polynomial, hence it is divisible by $x_n$.)
\end{lemma}

\begin{proof}
First of all we note that, by exactly the same argument as the one given in Step 2 of the proof of Lemma \ref{lem:barrier111}, the $L^2$ bound from \eqref{wiohehth} can be upgraded to and $L^\infty$ (up to enlarging the constant $C_\circ$). 
Therefore, we may assume that
\begin{equation}\label{wiohehth2}
\big\| u-p_2-p_3^* \big\|_{L^\infty(\CC_r)} \le C_\circ r^{3+2\beta}  \quad \mbox{for all }r\in (0,1), 
\end{equation}
where $\beta\in (0,\frac18)$.

Now, let  $\boldsymbol P(x,t)  := \frac{1}{2} (x_n +p_3(x,t) /x_n)^2$ and fix $(\bar x, \bar t)\in \{\boldsymbol P=0\}\cap \Omega^{\beta}_1\cap \CC_{r_\circ}$, with $r_\circ$ small to be chosen.
We need to prove that $u(\bar x,\bar t)=0$. 

First observe that, since $p_3^* = a|x_n| (x_n^2+6t) +p_3$  with $a\geq c_1>0$ (see Lemma \ref{lem:84y6t3}), \eqref{wiohehth2} implies that
\begin{equation}\label{wjiowhow33t5f}
\frac {1 }{r^3} (u-\boldsymbol P)_r \le  c|x_n|t +C|x_n|^3 + C_\circ r^{2\beta} \quad \mbox{in } \CC_2.
\end{equation}
We will use a modification of the proof of Lemma \ref{lem_barriermon}, using as barrier 
\[
\Psi_{\bar y, \bar s} (x,t) := r^{-3}\boldsymbol P_r(x,t) +  r^{\beta} \big(|x'-\bar y'|^2 - (t-\bar s)- 2n \,r^{-2}\boldsymbol P_r(x,t)\big), 
\]
where $(\bar y,\bar s)\in  \{\boldsymbol P_r=0\}\cap \CC_1$ is any given point satisfying  $\bar s\le r^{\beta}$.

Note that  $\heatop \boldsymbol P_r = r^2+ O(r^4)$ inside $\CC_1$. 
Hence, for $r$ sufficiently small we have
\[
\heatop \Psi_{\bar y, \bar s}   \le  \frac{1}{r} +Cr  + Cr^{\beta}\big(  2(n-1) +1 -2n(1+ O(r^2))\big) \le \frac{1}{r}  .
\]
Consider now the domain $U_\varrho := \left\{ |x'-\bar y'| \le \frac{1}{10}, |x_n| \le  \varrho,  \frac{1}{10}\le  t-\bar s\le 0  \right\}$ with $\varrho:=r^{2\beta /3}$.
Using \eqref{wjiowhow33t5f} and the fact that $r^{-2}\boldsymbol P_r =\frac 1 2 x_n^2 + O(r)$ and  $\bar s\le r^{\beta}$, on $\partial_{par}U_\varrho \cap\{|x_n|=\varrho\}$ we have
\begin{multline}
r^{-3}u_r  - \Psi_{\bar y, \bar s}   \le 
2n \,r^{\beta} r^{-2}\boldsymbol P_r   + \frac{c}{2}  \bar s  \varrho + C\varrho^3 + C_\circ r^{2\beta}
\\ 
\le  2n\,r^{\beta}(\varrho^2 + Cr)  - \frac c 2  r^{\beta}\varrho + C\varrho^3 +  C_\circ r^{2\beta} \le Cr^{2\beta} -\frac c 2 r^{5\beta/3} < 0 
\end{multline}
for $r$ small enough.
Also it is easy to show that also $r^{-3}u_r  - \Psi_{\bar x, \bar t} \le 0$ on  the remaining pieces of $\partial_{par}U_\varrho$, thanks to the positive term $r^{\beta} \big(|x'-\bar y'| -(t-\bar s)\big)$ appearing in the definition of $\Psi_{\bar x, \bar t}$. 

Hence, since $r^{-3}\heatop u_r = \frac{1}{r} \ge \heatop \Psi_{\bar x, \bar t}$ inside $\{u_r >0\}$ and $\Psi_{\bar x, \bar t}$ is nonnegative, it follows from the maximum principle that
\[
0\le r^{-3}u_r \le \Psi_{\bar y, \bar s}\quad \mbox{in }U_\varrho.
\]
In particular, evaluating at $(\bar y,\bar s)$ we obtain 
 $ 0\le r^{-3}u_r(\bar y,\bar s) \leq \Psi_{\bar x, \bar t}(\bar x,\bar t)=0$. 
Since $(\bar y,\bar s)\in  \{\boldsymbol P_r=0\}\cap \CC_1$  was an arbitrary point satisfying  $\bar s\le r^{\beta}$, after rescaling we obtain
\[
u=0 \quad \mbox{on } \{\boldsymbol P=0\} \cap \CC_r \cap\left\{t \le -r^{2+\beta}\right\}
\]
whenever $r$ is sufficiently small, as desired.
\end{proof}

We can now prove Theorem \ref{thm:euihwbegs}.

\begin{proof}[Proof of Theorem \ref{thm:euihwbegs}]
Using Lemma \ref{lem:784hbtr6} we see that $\{u>0\}$ is split into  two connected components inside $\Omega^\beta_1\cap \CC_{r_\circ}$  --- the number of connected components cannot be larger since  $\partial_{nn} u\ge 0$ in $B_{r_\circ}\times (-r_\circ^2, r_\circ^2)$ thanks to  Lemma \ref{lemconvxn}.
Then $u = u^{(1)}+ u^{(2)}$, where $u^{(1)}$ and $u^{(2)}$ are respectively supported in $\{x_n +p_3(x,t) /x_n> 0\}$ and $\{x_n +p_3(x,t) /x_n< 0\}$ inside $\Omega^\beta_1\cap \CC_{r_\circ}$. 
Finally, \eqref{eq:u12} follows from the convergence $r^{-2}u(rx)\to \frac12 x_n^2$ and the fact that, after rescaling, the positivity set of $u^{(1)}$ (resp. $u^{(2)}$) converges to $\{x_n>0\}$ (resp. $\{x_n<0\}$).
\end{proof}

Once $u$ is split into two separate functions $u^{(1)}$ and $u^{(2)}$ in $\Omega^\beta_R\cap \CC_{r_\circ}$, we now look for a $C^\infty$ expansion for each of these two functions.
For this, we need first to construct a series of ansatz for the Taylor expansion of these functions.

\begin{definition}\label{ansatz1}
Let $k \ge 3$, and let $(Q_\ell)_{2\le \ell \le k-1}$ be a family of polynomials such that $Q_\ell=Q_\ell(x,t)$  is (parabolically) homogeneous of degree  $\ell$ and satisfies
$
\heatop ( x_nQ_\ell ) \equiv 0.
$
We define the polynomial $\mathscr A_k = \mathscr A_k[Q_2, \dots, Q_{k-1}]$ as follows.

For the base case $k =3$, we set
\[
 \mathscr A_{3} =  \mathscr A_3 [Q_2 ](x,t) : =   x_n +  Q_2(x,t) + x_nR_2(x,t),
\]
where $R_2$ is the unique 2-homogeneous  caloric polynomial satisfying\footnote{The fact that \eqref{ftdyuij;uytdrse} has exactly one solution can be shown as follows: write
\[
\heatop \bigg( \frac 1 2 (Q_2)^2\bigg) =:  T_{2,2}(x') +  x_n T_{2,1}(x') + x_n^2  T_{2,0},
\]
where $T_2^\ell$ is an $\ell$-homogeneous polynomial in the variables $x'$, and
\[
R_2 = S_{2,2}(x') +x_n S_{2,1}(x') +x_n^2 S_{x,0},
\]
where $S_{2,\ell}$ is a $\ell$-homogeneous polynomial in the variables $x'$.
Then  \eqref{ftdyuij;uytdrse} amounts to 
\[
2 S_{2,2}(x') + 6x_n S_{2,1}(x') + x_n^2 (12S_{2,0}+ \heatop S_{2,2}) = - T_{2,2}(x')  - x_n T_{2,1}(x') - x_n^2 T_{2,0},
\]
which leads to a linear system with triangular structure, for which the unique solution is given by
\[
 S_{2,2} = -\frac{1} 2T_{2,2}, \qquad 
 S_{2,1} =   - \frac{1} 6 T_{2,1}, \qquad
  S_{2,0}  =    -\frac{1}{12} (  T_{2,0} + \heatop S_{2,2}).
 \]
}
\begin{equation}\label{ftdyuij;uytdrse}
\heatop \bigg( \frac 1 2 (Q_2)^2 +    x_n^2R_2 \bigg) =0.
 \end{equation}
 Note that \eqref{ftdyuij;uytdrse} is equivalent to 
 \[
\heatop \bigg( \frac 1 2 \mathscr A_{3}^2  \bigg)  =  1 + O\big( (|x|+ |t|^{1/2} )^3\big).
 \]
For $k \ge 3$, we inductively define
\[
 \mathscr A_{k+1} \big[ (Q_\ell)_{2\le \ell \le k}  \big]  :=   \mathscr A_{k}\big[ (Q_\ell)_{2\le \ell \le k-1}  \big] + Q_k + x_n R_k,
\]
where  $R_k$ is  a $k$-homogeneous caloric polynomial characterized  by the identity
\begin{equation}\label{fuyghjlkgjhfgy}
\heatop \bigg( \frac 1 2 \mathscr A_{k+1}^2  \bigg)  =  1 + O\big( (|x|+ |t|^{1/2} )^{k+1}\big).
\end{equation}
As before, this equation has exactly one solution $R_k$. 
Indeed, since  $\mathscr A_k = x_n + O(|x|^2 + |t|)$, it follows by induction that
\[
\begin{split}
\heatop \bigg( \frac 1 2 \mathscr A_{k+1}^2 \bigg) = \heatop \bigg( \frac 1 2 \mathscr A_{k}^2  +  x_nQ_{k} + Q_2Q_{k-1} + x_n^2 R_k  \bigg)  + O\big( (|x|+ |t|^{1/2})^{k+1} \big)
\end{split}
\]
Hence, if we define
\[
\heatop \bigg( \frac 1 2 \mathscr A_{k}^2+x_nQ_k+Q_2Q_{k-1} \bigg) =: 1+ \sum_{j=0}^k x_n^{k-j} T_{k,j}(x') + O\big( (|x|+ |t|^{1/2} )^{k+1}\big) 
\]
and write
\[
R_k : =    \sum_{j=0}^k x_n^{k-j} S_{k,j}(x'),
\]
then \eqref{fuyghjlkgjhfgy} is equivalent to 
\[
\sum_{j=0}^{k}  (k-j+2)(k-j+1)x_n^{k-j} S_{k,j}(x')   +  \sum_{j=0}^{k-2} x_n^{k-j} \heatop S_{k,j+2}    = -  \sum_{j=0}^k x_n^{k-j} T_{k,j}(x').
\]
This leads to a triangular system, whose unique solution is given by
\[
 S_{k,k} = -\frac{1} {2}T_{k,k}, \quad 
 S_{k,k-1} =   - \frac{1} 6 T_{k,k-1}, \quad
  S_{k,j}  =  -  \frac{1}{(k-j+2)(k-j+1)} (  T_{k,j}  + \heatop S_{k,j+2}) \quad \text{for }j \in \{0,\ldots,k-2\}.
 \]
\end{definition}
The previous definition aims to construct a Taylor expansion which is compatible with the PDE satisfied by our solution.
However, the previous formulas do not take care of possible translations and rotation. This is the purpose of the next definition.

\begin{definition}\label{ansatz2}
Let $k \ge 3$, and let $(Q_\ell)_{2\le \ell \le k-1}$ be a family of parabolically homogeneous  polynomials of degree  $\ell$ satisfying 
$\heatop ( x_nQ_\ell ) \equiv 0.$
Then, given $\tau \in \R$ and a rotation $\boldsymbol R\in  SO(n)$, we define 
\[
 \anz_k = \anz_k \big[ Q_2,\dots, Q_{k-1},   \tau, \boldsymbol R\big] (x,t) 
\]
by
\[ 
 \anz_k (x,t):= \frac 1 2 \left(  \mathscr A_k \big[ Q_2, \dots, Q_{k-1}  \big] \right)_+^2   ( \boldsymbol R (x+\tau \boldsymbol e_n), t),
 \]
where $\mathscr A_k$ is given in Definition \ref{ansatz1}.
 \end{definition}

Our next goal will be to prove an $\varepsilon$-regularity result in the domain $\Omega^\beta_R$.
Notice that this result is for ``one-sided'' solutions, i.e., it will be applied separately to $u^{(1)}$ and $u^{(2)}$.

\begin{theorem}\label{tdrsdfghkljjhgfxd}
Given $\alpha\in (0,1)$,  $\beta>0$, $k \ge 4$, and $\delta\in (0,1)$, there exist positive constants $\bar r$ and $\ep_\circ$, depending only on $n$, $\beta$, and $k$, such that the following holds.

Let $u$ satisfy $\heatop u = \chi_{\{u>0\}}$, $u\ge 0$, and $\partial_t u\ge 0$ inside $\Omega^\beta_{1} \cap \CC_{1}$. Assume that $u(0,0)=0$ and
\[
\bigg\| u - \frac{1}{2} (x_n)_+^2 \bigg\|_{L^\infty(\Omega^\beta_{1} \cap \CC_{r_\circ})} \le \ep_\circ.
\]
Then, for some  Ansatz $\anz_k =  \anz_k \big[(Q_\ell)_{2\le\ell \le k-1},  \tau, \boldsymbol R\big]$ we have
 \[
\| u - \anz_k\|_{L^\infty( B_r \times (-r^2, -\delta r^{2}) )}\le r^{k+1+\alpha}
\]
for all $r\in (0, \bar r)$.
\end{theorem}

As a corollary of this result, it follows that the free boundary is $C^\infty$ near every \emph{regular} free boundary point.
However, thanks to the fact that our result looks at the solution only inside $\Omega^\beta_{1}$, thanks to Theorem \ref{thm:euihwbegs} it can be applied as well to the case of singular points in $\Sigma^{\diamond}$.

Before proving Theorem \ref{tdrsdfghkljjhgfxd}, we will need several ingredients.
We start with a compactness result.

\begin{lemma}\label{lem-comp}
Given $n$, $k$, $\beta$, $\alpha$, and $\ep_\circ$, there exists a positive constant $M_\circ$ such that the following holds. 
Assume that  $N_\circ \ge M_\circ$ and that $w : \Omega^\beta_{N_\circ} \cap \CC_{M_\circ} \to \R$ satisfies 
\begin{equation}\label{hyp1}
\bigg(\int_{\Omega^\beta_{N_\circ}\cap \CC_1} |w_{2^m}|^2 + |\nabla w_{2^m}|^2  + |\partial_t w_{2^m}|^2\, dx \,dt\bigg)^{1/2} \le  (2^m)^{k+1 +\alpha}  \quad \mbox{for  } 0\le m \le M_\circ, 
\end{equation}
where $w_{2^m}(x,t) := w(2^mx, 4^mt)$.
Assume in addition that
\begin{equation}\label{hyp2}
\begin{cases}
|\heatop w| \le\frac{1}{M_\circ} \  & \mbox{ in } \ \Omega^\beta_{N_\circ} \cap \big( B_{M_\circ} \times (-M_\circ, -{\textstyle \frac {1}{M_\circ}}) \big) \cap \big\{x_n \ge{\textstyle \frac{1}{M_\circ}} \big\}\\
 w =0  &\mbox{ in }\   \Omega^\beta_{N_\circ}\cap \big( B_{M_\circ} \times (-M_\circ, -{\textstyle \frac {1}{M_\circ}}) \big) \cap \big\{x_n \le -{\textstyle \frac{1}{M_\circ}} \big\}.
\end{cases}
\end{equation}
Finally, suppose that 
\begin{equation}\label{hyp3 Q}
\bigg| \int_{-2}^{-1}\int_{B_{M_\circ}}  (x_n)_+ w\, Q \, G \,dx \,dt  \bigg| \le \frac{1}{M_\circ}
\end{equation}
for any $Q= Q(x,t)$ parabolically  $\ell$-homogeneous polynomial for some $0\le \ell\le k$ which satisfies $\heatop(x_n Q)=0$.
Then
\[
\bigg( \int_{\Omega^\beta_{N_\circ}\cap \CC_1} |w|^2 \bigg)^{1/2} \le \ep_\circ.
\]
\end{lemma}

\begin{proof}
We split the proof in two steps.

\smallskip

\noindent $\bullet$ \emph{Step 1}. We show the following classification result for ancient solutions in the whole space. 
 Let $w: \R^n \times (-\infty, 0)\to \R$ be such that both $|w|^2$ and $|\nabla w|^2$  are locally integrable and satisfy
\begin{equation}\label{ftj;uyftdr1}
\bigg(\int_{\CC_1} |w_{2^m}|^2 + |\nabla w_{ 2^m}|^2 + |\partial_t w_{ 2^m}|^2 \, dx\,dt\bigg)^{1/2} \le   (2^m)^{k+1 +\alpha}  \quad \mbox{for  all } m\ge 1,
\end{equation}
and assume in addition that 
\begin{equation}\label{ftj;uyftdr2}
\heatop w =0 \quad \mbox{ in  }\{x_n >0\},\qquad   w|_{\{x_n\leq 0\}} \equiv 0,
\end{equation}
and 
\begin{equation}\label{ftj;uyftdr3}
\int_{-2}^{-1}\int_{\R^n} (x_n)_+ w \,Q \,G \,dx \,dt  =0
\end{equation}
for  any polynomial $Q= Q(x,t)$ which parabolically  $\ell$-homogeneous for some $0\le \ell\le k$ and satisfies $\heatop(x_n Q)=0$.
Then, 
\[
w\equiv 0.
\]
To show this we notice that, thanks to \eqref{ftj;uyftdr1} and \eqref{ftj;uyftdr2}, it follows from the Liouville Theorem for the heat equation that the odd extension of $w$ across $\{x_n=0\}$ is a caloric polynomial of degree at most $k+1$ vanishing on $x_n=0$. Since \eqref{ftj;uyftdr3} implies that $w$ must be orthogonal to all such polynomials (recall that $w$ vanished for $x_n\leq 0$), we conclude that $w\equiv 0$.

\smallskip

\noindent $\bullet$ \emph{Step 2}. Since $H^1_{\rm loc} ( \R^n \times (-\infty,0] )$ embeds compactly in $L^2_{\rm loc}(( \R^n \times (-\infty,0] )$, 
the desired result follows immediately by compactness as $M_\circ \to \infty$.
\end{proof}

We also need the following estimate.
Notice that when the integrals below are taken in $\CC_2$ and $\CC_1$, then the estimate is more standard ---and we proved it in Lemma \ref{lem:u-v}.
Now, we need such estimate in the new domains $\Omega^\beta_R$, and thus it becomes a bit more delicate.

\begin{lemma}\label{lem:u-v2}
Fix $\beta\in (0,1)$, and let $u_i: \CC_2\to  \R$, $i =1,2,$ be  two solutions of
\eqref{gakljhgaghk;}
inside  $\Omega^{\beta}_R \cap \CC_2$, with $|\ep_i(x,t)|\le \bar \ep<\frac{1}{100}$. Set $w := u_1-u_2$.
Then, for any $\theta\in (\frac12,1)$ there exist constants $R_\circ,C>0$, depending only on $n$, $\beta$, and $\theta$, such that 
\begin{equation}\label{hhhaiuguiag22}
\int_{\Omega^{\beta}_{\theta^2 R} \cap \CC_1} |\nabla w|^2   + |w\,\heatop w|  + |\partial_t w|^2 
\le C \bigg(\int_{\Omega^{\beta}_{R} \cap \CC_2}  w^2  +  \bar \ep ^2\bigg) \qquad\forall\,R \geq R_\circ.
\end{equation}
\end{lemma}

\begin{proof} As in the proof of Lemma \eqref{lem:u-v} we have
\[
w\,\heatop w \ge- 2\bar \ep |w| \quad \mbox{and} \quad  \partial_t w\,\heatop w \ge- 2\bar \ep |\partial_t w|.
\]
The idea is to follow the proof of Lemma \eqref{lem:u-v}, but multiplying the spatial cut-offs $\tilde\eta(x)$ and $\eta(x)$ by an extra appropriately chosen space-time cut-off in order to make sure that we only
evaluate the equation inside its domain $\Omega^\beta_{R} \cap \CC_2$. 

Set $\alpha: =\frac{1}{2+\beta}$ and note that $\alpha \in (1/3,1/2)$. 
 The new cut-off will be of the form 
\[
\xi^\beta_R = \xi^\beta_R(x,t) := \xi\left(\frac{|x|}{R(-t)^{\frac{1}{2+\beta}}}\right) \quad \mbox{for  some suitable cut-off $\xi\in C^{1,1}  ( [0,1])$ satisfying $\xi|_{[0,\theta]}\equiv 1$}. 
\] 
More precisely, we choose $\xi\in C^{1,1}([0,1])$ which satisfies the following properties:
\begin{equation}\label{gaiuhguiya}
\xi(s) :=  
\begin{cases} 
1 \quad  & \mbox{for }s\in[0,\theta]\\
\mbox{concave and decreasing}& \mbox{for }s\in\big[\theta, {\textstyle \frac{\theta+1}{2}}\big]\\
(1- s)^{2+p} & \mbox{for }s\in\big[{\textstyle \frac{\theta+1}{2}},1\big],
\end{cases}
\end{equation}
where  $p\in (2,\infty)$ satisfies $\frac{p}{p+1}=2\alpha$.
It is important to notice that, since $\xi'\le 0$,
\begin{equation}\label{guoaihgfyuafguai}
\partial_t  \xi^\beta_R(x,t)  \le 0.
\end{equation}
We claim that
\begin{equation}\label{guoaihgfyuafguai2}
\Delta\xi^\beta_R +\partial_t  \xi^\beta_R\le C_\circ, \quad \mbox{ for  }(x,t)\in \R^n\times (-4,0),
\end{equation}
where $C_\circ$ depends only on $n$, $\theta$,  $\beta$, but not on $R$.

Indeed,  defining $s = s(x,t) :=\frac{|x|}{R(-t)^{\alpha}}$, since the inequality is trivial for $s \leq \theta$ we have
\[
\Delta \xi^\beta_R +\partial_t  \xi^\beta_R   =  \frac{1}{R^2(-t)^{2\alpha}}\bigg( \xi''(s) + \frac{(n-1)}{s} \xi'(s)\bigg) +  \xi'(s) \frac{\alpha s}{(-t)}  \le  \frac{\xi''(s)}{R^2(-t)^{2\alpha}}   + \frac{ \xi'(s)}{6(-t)},
\]
where we used that $\xi'\le 0$ and that $\alpha s \geq 1/6$ for $s \geq \theta \geq 1/2$.
Since $\xi$ is concave and decreasing inside $\big[\theta, \frac{\theta+1}{2}\big]$, the right hand side above is negative  for $s \in \big[\theta, \frac{\theta+1}{2}\big]$. Hence, to prove \eqref{guoaihgfyuafguai2} we only need to bound the previous expression for  $s \in \big[\frac{\theta+1}{2}, 1\big]$, where we have $\xi(s) =(1-s)^{2+p}$. Recalling that  $2\alpha= \frac{p}{p+1}$ and writing $\tau: =(1-s)^{p+1}$,
for $R \geq 1$ we obtain
\[
\begin{split}
\frac{1}{R^2(-t)^{2\alpha}} \xi''(s)  + \frac{ \xi'(s)}{-6t} &= (2+p)  \bigg( \frac{(1+p)}{R^2 (-t)^{2\alpha}} (1-s)^{p} + \frac{1}{6t}  (1-s)^{p+1}\bigg)
\\
& =  \frac{(2+p)}{6}  \bigg( \frac{6(1+p)}{R^2} \bigg(\frac{\tau}{-t}\bigg)^{2\alpha} - \bigg( \frac{\tau}{-t}\bigg)\bigg) \le C_{p, \theta} \qquad \forall\,\tau \in (0,1),
\end{split}
\]
where we used that $2\alpha\in (0,1)$. This concludes the proof of \eqref{guoaihgfyuafguai2}. 

We also note that 
\begin{equation}\label{guoaihgfyuafguai3}
32|\nabla  \xi^\beta_R| +\partial_t  \xi^\beta_R\le 0 \quad \mbox{ for  }(x,t)\in \R^n\times (-4,0).
\end{equation}
Indeed, as before we only need  to check the inequality for $s \geq \theta$, and we have
\[
32|\nabla\xi_\beta| +\partial_t \xi_\beta   =  \frac{32}{R(-t)^{\alpha}} |\xi'(s)| -  \xi'(s) \frac{\alpha s}{t}  \le \frac{32}{R(-t)^{\alpha}} |\xi'(s)| +  \xi'(s) \frac{1}{6(-t)}  
\le   0,
\]
provided that $R\ge R_\circ$ is large enough.

We can now prove \eqref{hhhaiuguiag22}.  As in the proof of Lemma \ref{lem:u-v}, let
\begin{equation}\label{hgyfghilk1v}
\tilde \eta= \tilde \eta (x) \in C^\infty_{c}(B_{5/3})  \quad \mbox{such that  $\tilde \eta\equiv 1$ in $\overline B_{4/3}$, $0 \le \tilde \eta\le 1$ in $B_{5/3}$, and $|\nabla \tilde \eta|+|D^2 \tilde\eta| \le 10$.}
\end{equation}
Then
\[
\begin{split}
\frac{d}{dt} \int_{B_{5/3} \times \{t\}}  w^2 \tilde  \eta \, \xi^\beta_{R}  + 2\int_{B_{2} \times \{t\}}  (w\,\heatop w  + |\nabla w|^2)\, \eta \, \tilde \xi^\beta_{R}    
&=   \int_{B_{2}\times\{t\}} w^2 (\partial_t +  \Delta) \big(\tilde  \eta  \xi^\beta_{R}\big) 
\\
&= \int_{B_{2}\times\{t\}} w^2     \big(\tilde  \eta (\partial_t +   \Delta )  \xi^\beta_{R}  + 2 \nabla \tilde  \eta \cdot \nabla \xi^\beta_{R}+\Delta\tilde \eta\,\xi^\beta_R \big).
\end{split}
\]
Note that, thanks to \eqref{guoaihgfyuafguai2},   $\tilde  \eta (\partial_t +   \Delta )   \xi^\beta_{R} \le C_\circ \chi_{B_{5/3}}(x) \chi_{\Omega^\beta_R}$. Also, the term $\nabla \tilde  \eta \cdot \nabla  \xi^\beta_{R}$ is also bounded because $\frac1{R(-t)^\alpha}\sim 1$ on the intersection of the supports of  
$\nabla\tilde \eta$ and $\nabla  \xi^\beta_{R}$. 
Therefore, we obtain 
\begin{equation}\label{anioqioqbioboi21}
\frac{d}{dt} \int_{B_{5/3} \times \{t\}}  w^2 \tilde  \eta \, \xi^\beta_{R}  + 2\int_{B_{2} \times \{t\}}  (w\,\heatop w  + |\nabla w|^2) \,\tilde \eta \, \xi^\beta_{R}   \le C \int_{B_{5/3}\times\{t\}} w^2 \chi_{\Omega^\beta_{R}} . 
\end{equation}
Integrating  \eqref{anioqioqbioboi21} with respect to $t \in (-T_*,0)$ where $T_* \in (3/2,2)$ satisfies
\[
\int_{(B_{5/3} \times \{-T_*\}) \cap \Omega^\beta_{R}}  w^2  \xi^\beta_{R}  \le  10 \int_{(B_{5/3} \times (-2,0))\cap \Omega^\beta_R} w^2 \xi^\beta_{R},
\]
and recalling that $w\,\heatop w \ge- 2\bar \ep |w|$ and that, by construction, $\tilde\eta \equiv 1$ in $B_{4/3}$ and $\xi^\beta_R \equiv  1$ in $\Omega^\beta_{\theta R}$, we obtain
\begin{equation}\label{fyuilkhhkjjhgfx}
\begin{split}
 \int_{-3/2}^0\int_{B_{4/3}}  \big(|w\,\heatop w|  + |\nabla w|^2\big)  \chi_{\Omega^\beta_{\theta R}}  \le  C\bigg( \int_{(B_{5/3} \times (-2,0))\cap  \Omega^\beta_{R} } w^2  + \bar \ep  |w| \bigg)  
 \le C\bigg(\int_{(B_{5/3} \times (-2,0))\cap  \Omega^\beta_{ R}} w^2 + \bar \ep^2\bigg). 
\end{split}
\end{equation}
This proves the estimates for $|\nabla w|^2$ and $|w\,\heatop w|$ in \eqref{hhhaiuguiag22}.

To control $|\partial_t w|^2$, we choose a spatial cut-off $\eta\in C^\infty_{c}(B_{4/3}) $ such that $\eta\equiv 1$ in $\overline B_{1}$. Since $\partial_t  w \heatop w \ge -2|\partial_t w| \bar \ep$ we obtain (using the inequality $4ab \leq \frac{2}Ma^2+2M b^2$ and choosing $M=8$)
\[
\begin{split}
2\int_{B_2 \times \{t\}} &(\partial_t w)^2 \eta^2 (\xi^\beta_{\theta R})^2  +  \frac{d}{dt} \int_{B_2 \times \{t\}} |\nabla w|^2 \eta^2  (\xi^\beta_{\theta R})^2   -   \int_{B_2 \times \{t\}} |\nabla w|^2\partial_t\big[ \eta^2\,( \xi^\beta_{\theta R})^2 \big]   =
\\
&=  2\int_{B_2 \times \{t\}}  (\partial_t w)^2  \eta^2 (\xi^\beta_{\theta R})^2   +\int_{B_2\times\{t\} }2 \nabla w  \cdot \nabla \partial_ t w\,\eta^2  (\xi^\beta_{\theta R})^2
\\
& =   \int_{B_2\times\{t\} }  2(\partial_t w -\Delta w ) \,\partial_ t w\,\eta^2 (\xi^\beta_{\theta R})^2  - 4 \nabla w\cdot \nabla \eta\,\partial_t w \, \eta  \,(\xi^\beta_{\theta R})^2 -  4\partial_t w\, \nabla w\cdot \nabla  \xi^\beta_{\theta R}\, \eta^2\xi^\beta_{\theta R}
\\
&\le    \int_{B_2\times\{t\} }  4\bar \ep  | \partial_ t w|\eta^2 (\xi^\beta_{\theta R})^2 +   \int_{B_{2}\times\{t\}} 4 M | \nabla w|^2  \big(|\nabla \eta|^2(\xi^\beta_{\theta R})^2 + |\nabla  \xi^\beta_{\theta R}|^2 \eta^2\big)  + \frac{4}{M}   (\partial_ t w)^2 \eta^2 (\xi^\beta_{\theta R})^2
\\ 
& \le  \int_{B_2 \times \{t\}} (\partial_t w)^2 \eta^2 (\xi^\beta_{\theta R})^2   + C\bar\ep ^2 +  32 \int_{B_{2}\times\{t\}}  | \nabla w|^2  \big(|\nabla \eta|^2 (\xi^\beta_{\theta R})^2 + |\nabla \xi^\beta_{\theta R}|^2 \eta^2\big). 
\end{split}
\]
Hence, recalling \eqref{guoaihgfyuafguai3}, since $ \xi^\beta_{\theta R}$ is supported  in $\Omega^\beta_{\theta R}$, we obtain
\[
\int_{B_2 \times \{t\}} (\partial_t w)^2 \eta^2 (\xi^\beta_{\theta R})^2  +  \frac{d}{dt} \int_{B_2 \times \{t\}} |\nabla w|^2 \eta^2  (\xi^\beta_{\theta R})^2     \le C  \bigg( \int_{(B_{4/3}\times\{t\})\cap \Omega^\beta_{\theta R} }   | \nabla w|^2  + \bar \ep^2\bigg).
\]
Integrating with respect to $t \in (-T_{\diamond},0)$ where $T_{\diamond} \in(1,3/2)$ is chosen so to satisfy
\[
\int_{(B_{4/3} \times \{T_\diamond \}) \cap \Omega^\beta_{\theta R}}  |\nabla w|^2\eta^2  (\xi^\beta_{\theta R})^2   \le  10 \int_{(B_{4/3} \times (-3/2,1))\cap \Omega^\beta_{\theta R}} |\nabla w|^2\eta^2  (\xi^\beta_{\theta R})^2,
\]
and recalling that $\eta\equiv 1$  in $B_1$ and $ \xi^\beta_{\theta R}\equiv 1$ in $\Omega^\beta_{\theta^2 R}$, we finally get
\[
\int_{\Omega^{\beta}_{\theta^2 R} \cap \CC_1}  |\partial_t w|^2  \le C \bigg(\int_{\Omega^{\beta}_{R} \cap \CC_2} |\nabla w|^2  +  \bar \ep ^2\bigg).
\]
This completes the proof of  the estimate \eqref{hhhaiuguiag22}, and the Lemma follows.
\end{proof}

Recall now the set of Ansatz given  in Definition \ref{ansatz2}. We note that our parametrization of possible Ansatz
given in Definition \ref{ansatz2} is non-injective: for every ${\boldsymbol S} \in SO(n)$ that fixes $\{x_n=0\}$ we have 
\begin{equation}\label{fghjokjihugyftd}
\anz_k\big[     \tau, \boldsymbol R{\boldsymbol S^{-1}}, {\boldsymbol S}_*Q_2, \dots , {\boldsymbol S}_*Q_{k-1}  \big] =  \anz_k\big[\tau, \boldsymbol R, Q_2 , \dots, Q_{k-1}\big], \quad \mbox{where} \ {\boldsymbol S}_* P (x,t): = P( {\boldsymbol S}x,t).
\end{equation}
In Lemma \ref{tgbyhn} below, we will consider the set of Ansatz that are a ``small perturbation'' of $\frac{1}{2} (x_n)_+^2$,
and it will be convenient to provide an injective and smooth parametrization of such subset.
\begin{definition} \label{defi-N_delta}
Given $\delta>0$ small,
let
\[
\mathcal N_\delta : = \big\{ \mbox{$\anz_k\big[\tau, \boldsymbol R, Q_2, \dots, Q_{k-1} \big]$ :  $|\tau| \le \delta$, \,$ |\boldsymbol R  \boldsymbol e_n - \boldsymbol e_n|\le \delta$, \,$\max_\ell \| Q_\ell\|_{L^2(Q_1)} \le \delta$} \big\}.
\]
\end{definition}
Our goal is to find a smooth bijective parameterization of $\mathcal N_\delta$.
Keeping \eqref{fghjokjihugyftd} in mind, given  $e\in \mathbb S^{n-1}$ with $|e-\boldsymbol e_n|\le \delta$ we define 
\[
\boldsymbol R_e : = \arg\min \big\{ \|\boldsymbol R -{\rm Id}\|_{HS} \ : \ \boldsymbol R\in SO(n), \quad \boldsymbol R \boldsymbol e_n =e\big\}.
\]
Note that $\boldsymbol R_e$ is a rotation of angle $\angle( \boldsymbol e_n, e)$ in the plane generated by $\boldsymbol e_n$ and $e$ and leaves all the vectors orthogonal to this plane invariant (when $e=\boldsymbol e_n$ then $\boldsymbol R_e= {Id}$).
So, it makes sense to consider the map
\begin{equation}\label{gfyuiugyf}
\big(\tau, e, Q_2, \dots, Q_{k-1}\big) \ \longmapsto  \  \anz_k\big[ \tau, \boldsymbol R_e , Q_2, \dots, Q_{k-1}\big],
\end{equation}
and we want to show that it is a bijection for $\delta$ small.

For technical convenience, instead of considering $\anz_k$, it makes sense to replace it with the polynomial of  degree $2k$ that equals  $\anz_k$ inside $\{\anz_k >0\}$, namely
\begin{equation}\label{eq:Pk tilde Pk}
 \anz_k (x,t)= \frac 1 2 \left(  \mathscr A_k \big[ Q_2, \dots, Q_{k-1}  \big] \right)_+^2   ( \boldsymbol R (x+\tau \boldsymbol e_n), t)\quad 	\leadsto
 \quad \widetilde \anz_k (x,t)= \frac 1 2 \left(  \mathscr A_k \big[ Q_2, \dots, Q_{k-1}  \big] \right)^2   ( \boldsymbol R (x+\tau \boldsymbol e_n), t).
\end{equation}
In this way, the map \eqref{gfyuiugyf} takes value in the smooth manifold of $2k$ homogeneous polynomials.
In the next lemma we compute the differential of this map at $\tau =0$, $e=\boldsymbol e_n$,  $(Q_\ell) = 0$, and prove that it is a smooth diffeomorphism for $\delta$ small.
Also, we show that $\widetilde \anz_k$ is completely determined, in a continuous way, by its coefficients of degree $\le k$.

\begin{lemma}\label{tgbyhn}
Let $\anz_k, \mathcal N_\delta$ be as in Definitions \ref{ansatz2} and \ref{defi-N_delta}.
For $\anz_k\in \mathcal N_\delta$, let  $\widetilde \anz_k$ denote the polynomial of degree $2k$ which equals $\anz_k$ in $\{\anz_k > 0\}$, see \eqref{eq:Pk tilde Pk}.
Then, the differential of the map 
\begin{equation}\label{gauijiahiqhiq}
\big(\tau ,e,  Q_2, \dots, Q_{k-1}\big)\ \mapsto \widetilde \anz_k \big[\tau, \boldsymbol R_e , Q_2, \dots, Q_{k-1}\big]
\end{equation}
at the point $\mathcal O : = \big(\tau=0,e= \boldsymbol e_n,  Q_2 =0, \dots, Q_{k-1}=0 \big)$ has the following diagonal structure:
\begin{equation}\label{weoithoiwhw}
\begin{split}
\frac{d}{d\ep}\bigg|_{\ep=0}   \widetilde\anz_k \big[\ep\tau', \boldsymbol R_{\boldsymbol e_n}, 0, \dots, 0 \big] &=   x_n \tau' \\
\frac{d}{d\ep}\bigg|_{\ep=0}\widetilde\anz_k \big[0, \boldsymbol R_{\frac{\boldsymbol e_n+\ep e'}{|\boldsymbol e_n+\ep e'|}}, 0, \dots, 0 \big] &=    x_n(x\cdot e')\\
\frac{d}{d\ep }\bigg|_{\ep=0} \widetilde\anz_k \big[0, \boldsymbol R_{\boldsymbol e_n}, 0, \dots, \ep Q_{\ell}',\dots,0 \big]  &=    x_n Q_\ell'  \qquad \mbox{for }2\le \ell\le k-1
\end{split}
\end{equation}
where $\tau' \in \R$, $e' \in \mathbb S^{n-1}\cap \{x_n=0\}$, and $Q_\ell'$ are homogeneous polynomials of degree $\ell$ satisfying $\heatop (x_nQ_\ell') =0$. 
In particular, \eqref{gauijiahiqhiq} gives a diffeomorphic parametrization of $\mathcal N_\delta$.

Moreover, for $0\le i \le 2k$, let $\varpi_i [ \widetilde\anz_k ]$ denote the  $i$-homogeneous part of $\widetilde\anz_k$. Then the differential of the map
\[
\big(\tau ,e,  Q_2, Q_3, \dots, Q_{k-1} \big)\ \mapsto     \left( \varpi_1 \big(\widetilde\anz_k\big), \varpi_2 \big(\widetilde\anz_k\big), \dots, \varpi_k \big(\widetilde\anz_k\big)\right)
\]
at $\mathcal O$ has maximal rank. In particular there exists $\delta_k>0$ small enough such that, inside $\mathcal N_{\delta_k}$, the whole polynomial $\widetilde \anz_k$ is determined (in a continuous way) by its coefficients of  terms of degree $\le k$.
\end{lemma}

\begin{proof}
The result follows from a direct computation, recalling that
\[
\widetilde\anz_k \big[\tau ,\boldsymbol R_e,  Q_2, Q_3, \dots, Q_{k-1}\big] (x,t) =  \frac  12 \big( x_n + Q_2 + x_nR_2  + \dots + Q_{k-1} + x_n R_{k-1}\big)^2(\boldsymbol R_ex+\tau   \boldsymbol e_n, t).
\]
Also, for $\delta$ small, the map is a diffeomorphism thanks to the implicit function theorem.
\end{proof}

Combining all the previous results we can now prove the following result, from which we will deduce 
Theorem~\ref{tdrsdfghkljjhgfxd}.

\begin{proposition}\label{fuihgftdfgh12}
Given $\alpha\in (0,1)$,  $\beta>0$, and $k \ge 3$, let $\delta=\delta_k$ be given by Lemma~\ref{tgbyhn}. There exist positive constants $N_\circ$ and $\ep_\circ$, depending only on $n$, $\beta$, and $k$, such that the following holds.

Let $u$ satisfy $\heatop u = \chi_{\{u>0\}}$, $u\ge 0$, and $\partial_t u\ge 0$ inside $\Omega^\beta_{N_\circ} \cap \CC_{1}$.
Assume  that, for some integer  $j\ge N_\circ$ and a sequence of Ansatz $\anz_k^{(\ell)}$,  we have
\begin{equation} \label{yguhjkljkh1}
\bigg(  \int_ {\Omega^\beta_{N_\circ}\cap \CC_{1}}  \left(u - \anz_k^{(m)} \right)^2   (2^{-m}x, 4^{-m}t)\,dx \,dt   \bigg)^{1/2}  \le \ep_\circ  (2^{-m})^{k+\alpha}  \quad \mbox{for  } 0\le m \le j ,
\end{equation}
where 
\begin{equation} \label{yguhjkljkh2}
\left\| \anz_k^{(0)} -\frac{1}{2} (x_n)_+^2 \right\|_{L^\infty(\CC_1)} \le \ep_\circ.
\end{equation}
Suppose also that
\begin{equation} \label{yguhjkljkh22}
\partial \big\{  u(2^{-j}\,\cdot\, ,4^{-j}\,\cdot\, )  =0 \big\} \cap \big( B_{N_\circ} \times \big(-N_\circ, {\textstyle -\frac{1}{N_\circ}} \big) \big) \subset  \big\{|x_n| \le\ep_\circ \big\}.
\end{equation}
Then, at the next scale $j+1$, there exists a new Ansatz  $\anz_k^{(j+1)} \in \mathcal N_{\delta_k}$ such that
\begin{equation} \label{yguhjkljkh3}
\bigg(  \int_ {\Omega^\beta_{N_\circ}\cap \CC_1}  \left(u - \anz_k^{(j+1)} \right)^2 \big(2^{-(j+1)} x, 4^{-(j+1)} t \big) \,dx \,dt   \bigg)^{1/2}  \le \ep_\circ \big(2^ {-(j+1)}\big)^{k+\alpha}.
\end{equation}
\end{proposition}

\begin{proof}
Let $\theta_\circ \in (0,1)$ 
be chosen so that $\Omega^\beta_{\theta_\circ R}$ is mapped to $\Omega^\beta_{R}$ under the parabolic doubling of the scaling $(x,t)\mapsto(2x, 4t)$. More precisly,
\[
(x,t)\in \Omega^\beta_{\theta_\circ R} \  \Leftrightarrow\  |x| \le  \theta_\circ R(-t)^{\frac{1}{2+\beta}}  \  \Leftrightarrow\   |2x| \le  \theta_\circ 2^{\frac{\beta}{2+\beta}}R (-4t)^{\frac{1}{2+\beta}}  \  \Leftrightarrow\  (2x,4t)\in \Omega^\beta_{R}
 \]
leads to 
 \[
 \theta_\circ :=  2^{-\frac{\beta}{2+\beta}}.
 \]
Given $M_\circ$ large, we define $N_\circ= N_\circ(\beta, M_\circ)$ as
\[
 N_\circ  := M_\circ \theta_\circ^{-M_\circ} \gg M_\circ.
\]
Defining $N_\circ$ in this way we guarantee that
\begin{equation}\label{wnoiwownojbentte1}
B_{M_\circ}\times(-2,-1) \subset  \Omega^\beta_{M_\circ}   =  \big\{ (2^{-M_\circ}x, 4^{-M_\circ} t) \ : \ (x,t)\in \Omega^\beta_{N_\circ} \big\}.
\end{equation}
We divide the proof in four steps. 

\smallskip

\noindent
$\bullet$ {\em Step 1}. We will first show that  assumptions \eqref{yguhjkljkh1} and \eqref{yguhjkljkh2} imply 
\begin{equation}\label{contolcoefs}
\anz_k^{(m)} \in \mathcal N_{\delta_k},\qquad \bigg\| \,\anz_k^{(m)} -\frac{1}{2} (x_n)_+^2\, \bigg\|_{L^\infty(\CC_1)}  \le \delta \quad \mbox{ for all }0 \leq m\le j, 
\end{equation}
where $\delta \leq \delta_k$ can be made arbitrarily small (by decreasing $\ep_\circ$).

Recall that, thanks to Lemma~\ref{tgbyhn}, the Ansatz $\anz^{(m)}_k = \anz^{(m)}_k\big[(Q_\ell^{(m)})_{2\le\ell \le k-1}, \sigma^{(m)},  \tau^{(m)}, \boldsymbol R^{(m)} \big]$ is determined by its coefficients of degree $\leq k$ (since these determine the spatial and time translations, the rotations, and the $Q_\ell$'s). 

Set $K_\circ : = \CC_{1/2} \cap\{x_n>\frac 1 {10}, \, t\le -\frac 1 {10}\}$  and note that  $K_\circ\subset \subset \{x_n>0\} \cap \Omega^\beta_{M_\circ}$. 
Then, for $\ep_\circ$ sufficiently small,  $K_\circ$ will belong to the positivity set of both $\anz_k^{(m)}$ and  $\anz_k^{(m+1)}$. 
Hence, defining $2^{-(m+1)}K_\circ : = \{(2^{-(m+1)}x, 4^{-(m+1)}t)\, : \,(x,t)\in K_\circ)\}$ and using \eqref{yguhjkljkh1} for two consecutive scales, we obtain 
\[
  \bigg( \ave_ {2^{-(m+1)}K_\circ}  \left(u - \anz_k^{(m)} \right)^2  dx \,dt   \bigg)^{1/2} + \bigg(\ave_ {2^{-(m+1)}K_\circ}  \left(u - \anz_k^{(m+1)} \right)^2 dx \,dt   \bigg)^{1/2}  \le C\ep_\circ  (2^{-(m+1)})^{k+\alpha},
\]
therefore
\[
  \bigg( \ave_ {2^{-(m+1)}K_\circ}  \left(\anz_k^{(m+1)} - \anz_k^{(m)} \right)^2 dx \,dt   \bigg)^{1/2}  \le C\ep_\circ (2^{-m})^{k+\alpha} ,
\]
or equivalently 
\begin{equation}\label{ahiohioha}
 \bigg( \ave_ {K_\circ}  \left(\anz_k^{(m+1)} - \anz_k^{(m)} \right)^2 (2^{-(m+1)}x, 4^{-(m+1)}t)  \,dx \,dt  \bigg)^{1/2}   \le C \ep_\circ (2^{-m})^{k+\alpha} .
\end{equation}
This implies, with the notation of Lemma \ref{tgbyhn},  that 
\begin{equation}\label{ahiohiohabis}
\left\| \varpi_i \big(\widetilde \anz_k^{(m+1)}\big) -\varpi_i \big(\widetilde \anz_k^{(m)}\big)\right\|\le  C \ep_\circ (2^{-m})^{k+\alpha-i} , 
\end{equation}
where $\|\cdot\|$ denotes any of the (equivalent) norms on the linear space of polynomials of degree $2k$.
Summing this bound over $m$, we deduce that
\[
\left\| \varpi_i \big(\widetilde \anz_k^{(m)}\big) -\varpi_i \big(\widetilde \anz_k^{(0)}\big)\right\|\le  C \ep_\circ  \quad \mbox{for all }  1\leq i\le k\mbox{ and }m\le j.
\]
Since $\varpi_1, \varpi_2, \dots,  \varpi_k$ determine all of  $\widetilde \anz_k$ in a continuous way (recall Lemma \ref{tgbyhn}), thanks to \eqref{yguhjkljkh2} we can guarantee that \eqref{contolcoefs} holds by choosing $\ep_\circ>0$ sufficiently small.

\smallskip

\noindent
$\bullet$ {\em Step 2}. Define 
\begin{equation}\label{ewhioewhoithe}
\anz^{(j+1)}_k  : = \arg\min_{\anz}   \int_{B_{M_\circ }\times (-2,-1) } (u-\anz)^2 (2^{-j}x, 4^{-j} t) G(x,t)  \,dx\,dt.
\end{equation}
We claim that, for some constant $C_\circ$ depending on $n$ and $\beta$, but {not} on $M_\circ$,  we have 
\begin{equation} \label{guaijhuag}
\biggl(\int_{\Omega^\beta_{N_\circ}  \cap \CC_1} \left( u-\anz^{(j+1)}_k\right)^2  (2^{-m} x, 4^{-m} t) \,dx\,dt\biggr)^{1/2}\le C_\circ \varepsilon_\circ(2^{-m})^{k+\alpha} \qquad \textrm{for}\quad m\leq j.
\end{equation}

Indeed, as a consequence of \eqref{ahiohiohabis}, summing the geometric series  from $m$ to $j-1$ we get
\begin{equation}\label{ahiohioha11}
\bigg(\int_{\Omega^\beta_{N_\circ}  \cap \CC_{1} }  \left(\anz_k^{(j)} - \anz_k^{(m)} \right)^2 (2^{-m}x, 4^{-m}t)  \,dx \,dt  \bigg)^{1/2}   \le C_1 \ep_\circ (2^{-m})^{k+\alpha} \qquad \textrm{for}\quad m\leq j,
\end{equation}
that combined with \eqref{yguhjkljkh1} gives
\begin{equation}\label{aguhaihia}
\int_{\Omega^\beta_{N_\circ}  \cap \CC_{1} } \left( u-\anz^{(j)}_k\right)^2(2^{-m}x, 4^{-m} t)\, \,dx \,dt \le C_2 \left(\varepsilon_\circ(2^{-m})^{k+\alpha}\right)^2 \qquad \textrm{for}\quad m\leq j.
\end{equation}
Recalling \eqref{wnoiwownojbentte1}, the bound above implies that
\[
\int_{B_{M_\circ}\times(-2,-1)   \cap \CC_{2^\ell} } \left( u-\anz^{(j)}_k\right)^2(2^{-j}x, 4^{-j} t)\, \,dx \,dt \le C_2 2^{\ell p} \left(\varepsilon_\circ(2^{-m})^{k+\alpha}\right)^2 \qquad \textrm{for}\quad  \ell \leq j,
\]
for some large exponent $p=p(n,k)>0$. (Note that, since $j\geq N_\circ$, $B_{M_\circ}\times(-2,-1)   \subset  \CC_{2^j}$.) In particular this implies that 
\[
\int_{B_{M_\circ}\times(-2,-1)   \cap \left(\CC_{2^\ell}\setminus \CC_{2^{\ell-1}}\right) } \left( u-\anz^{(j)}_k\right)^2(2^{-j}x, 4^{-j} t)\, \,dx \,dt \le C_2 2^{\ell p} \left(\varepsilon_\circ(2^{-m})^{k+\alpha}\right)^2 \qquad \textrm{for}\quad  \ell \leq j,
\]
and since the Gaussian Kernel is smaller than $C e^{-2^{\ell}}$ on the domain of integration (hence the exponential decay of $G(x,t)$ ``beats'' the polynomial growth  $2^{\ell p}$), we can find a constant $C_3$ independent\footnote{This will be very important later since we will need to take $M_\circ$ sufficiently large,  and it  will be crucial that the constant $C_3$ stays bounded independently of $M_\circ$.} of $M_\circ$ such that 
\begin{equation}\label{wiotewhoithew}
 \bigg(  \int_{B_{M_\circ}\times(-2,-1)}  \big(u-\anz^{(j)}_k\big)^2 (2^{-j}x, 4^{-j} t) G(x,t)\bigg)^{1/2}    \le C_3 \varepsilon_\circ (2^{-j})^{k+\alpha}\qquad \text{for}\quad j \geq N_\circ.
\end{equation}
for $j\geq j_\circ$. 
Hence, by definition of $\anz^{(j+1)}_k $
\[
 \bigg(   \int_{B_{M_\circ}\times(-2,-1) } \big(u-\anz^{(j+1)}_k\big)^2 (2^{-j}x, 4^{-j} t) G(x,t) \bigg)^{1/2}   \le C_3 \varepsilon_\circ(2^{-j})^{k+\alpha},
\]
and by triangle inequality we deduce that 
\[
 \bigg(   \int_{B_{M_\circ}\times(-2,-1) } \big(\anz^{(j+1)}_k- \anz^{(j)}\big)^2 (2^{-j}x, 4^{-j} t) G(x,t) \bigg)^{1/2}   \le C_3 \varepsilon_\circ(2^{-j})^{k+\alpha}.
\]
Since $\anz^{(j+1)}_k$ and  $\anz^{(j)}_k$ are positive parts of polynomials whose positivity set is approximately $\{x_n>0\}$ (see Step 1), this implies
\begin{equation}\label{whohwhe}
 \bigg(   \int_{\Omega^\beta_{N_\circ}  \cap \CC_1} \left( \anz^{(j)}_k-\anz^{(j+1)}_k \right)^2(2^{-j}x, 4^{-j} t) \bigg)^{1/2}  \le C_4 \varepsilon_\circ(2^{-j})^{k+\alpha}.
\end{equation}
Combining this bound with \eqref{ahiohioha11} and \eqref{aguhaihia}, \eqref{guaijhuag} follows.

\smallskip

\noindent
$\bullet$ {\em Step 3}. We now show that 
\begin{equation}\label{ehiohwohw}
\big|\heatop \anz^{(j+1)}_k -  1 \big| \le C\left(|x|^{2}+ |t|\right)^{k/2} \quad \mbox{inside }\CC_1,
\end{equation}
where $C$ is  independent of $M_\circ$ and $j$ are.

Indeed, using \eqref{whohwhe} and \eqref{ahiohiohabis} we obtain that the coefficients of  $\anz_k^{(j+1)}$ are bounded, independently of $j$. Hence, by construction of $\anz_k$ (see \eqref{fuyghjlkgjhfgy}) we obtain \eqref{ehiohwohw}.

\smallskip

\noindent
$\bullet$ {\em Step 4}. Define 
\[
w(x,t) =  \frac{\big(u-\anz^{(j+1)}_k\big) (2^{-j} x, 4^{-j} t) }{ C_\circ \varepsilon_\circ(2^{-j})^{k+\alpha}}.
\]
We want to show that, if $M_\circ$ is chosen large enough and $\delta>0$ (from Step 1) is  small enough, then Lemma \ref{lem-comp} applies.

Indeed, assumption \eqref{hyp1} follows from \eqref{guaijhuag},  \eqref{ehiohwohw}, and Lemma \ref{lem:u-v2}. 
Also, since $\heatop u = \chi_{\{u>0\}}$ and $\partial \{u (2^{-j} \,\cdot, 4^{-j} \,\cdot\,) >0\}$ is contained in a $1/M_\circ$-neighborhood of $\{x_n=0\}$ (thanks to \eqref{contolcoefs},  \eqref{whohwhe}, and \eqref{guaijhuag}),
assumption \eqref{hyp2} follows from \eqref{ehiohwohw}.
Finally, by definition of $\anz^{(j+1)}_k$ we have the optimality condition 
\begin{equation}\label{othorgonality}
  \int_{B_{M_\circ }\times (-2,-1) }\big( (u-\anz^{(j+1)}_k) \anz'\big) (2^{-j}x, 4^{-j} t) G(x,t)  \,dx\,dt=0, 
\end{equation}
where $\anz' $ is any element of the tangent space to the ``manifold of Ans\"atze'' $\mathcal N_\delta$ at  the ``point''  $\anz^{(j+1)}_k$ (see Lemma  \ref{tgbyhn}).
Since  the defining  parameters $\big(\tau,e,  Q_2 , \dots, Q_{k-1} \big)$ of   $\anz^{(j+1)}_k$ are a small perturbation of $\mathcal O : = \big(\tau=0,e= \boldsymbol e_n,  Q_2 =0, \dots, Q_{k-1}=0 \big)$
(cf. \eqref{contolcoefs} and \eqref{whohwhe}), \eqref{weoithoiwhw} and \eqref{othorgonality} imply the validity of assumption \eqref{hyp3 Q}, provided that $\delta$ is chosen sufficiently small.

Hence, by Lemma \ref{lem-comp}, 
\[
\bigg( \int_{\Omega^\beta_{N_\circ}\cap \CC_1} |w|^2 \bigg)^{1/2} \le \widetilde{\ep_\circ},
\]
where $\widetilde{\ep_\circ}$ is an arbitrarily small constant. In particular, 
choosing $\widetilde{\ep_\circ}=2^{-(k+\alpha)}/C_\circ$ we obtain \eqref{yguhjkljkh1} for $m=j+1$, as wanted.
\end{proof}

We finally  prove Theorem \ref{tdrsdfghkljjhgfxd}.

\begin{proof}[Proof of Theorem \ref{tdrsdfghkljjhgfxd}]
By choosing first $r_\circ>0$ sufficiently small, and then $\ep_\circ>0$ as small as needed, Proposition \ref{fuihgftdfgh12} can be applied to the function $\tilde u=r_\circ^{-2} u(r_\circ\,\cdot\, , r_\circ^2 \,\cdot\,)$
with $j= N_\circ$. But then the conclusion of Proposition \ref{fuihgftdfgh12} allows us to iterate it for any $j\geq N_\circ$, and we get
\begin{equation}\label{wohewohtrw}
\bigg(  \int_ {\Omega^\beta_{N_\circ}\cap \CC_{1}}  \left(\tilde u - \anz_k^{(j)} \right)^2   (2^{-j}x, 4^{-j}t)\,dx \,dt   \bigg)^{1/2}  \le \ep_\circ  (2^{-j})^{k+\alpha}  \quad \mbox{for all } N_\circ\le j <\infty.
\end{equation}
Also,  Step 1 in the proof  Proposition \ref{fuihgftdfgh12} (see in particular \eqref{ahiohiohabis}) shows that the coefficients of $\anz_k^{(j)}$ converge as $j \to \infty$,
hence  $\anz_k^{(m)} \to \anz_k^{(\infty)}$ for a certain limiting Ansatz. 
Arguing as in Step 2 in the proof  Proposition~\ref{fuihgftdfgh12} (see in particular \eqref{aguhaihia}), we also get
\begin{equation}\label{wohewohtrwbis}
\bigg(  \int_ {\Omega^\beta_{N_\circ}\cap \CC_{1}}  \left(\tilde u - \anz_k^{(\infty)} \right)^2   (2^{-j}x, 4^{-j}t)\,dx \,dt   \bigg)^{1/2}  \le \ep_\circ  (2^{-j})^{k+\alpha}  \quad \mbox{for  } N_\circ \leq j <\infty .
\end{equation}
Finally, for $r \ll 1$, we can  use Lemma \ref{lem:u-v} to improve the obtained  control  in $L^2$ inside $\Omega^\beta_{N_\circ}\cap \CC_1$ into an $L^\infty$-control inside the cylinder $B_r\times(-r^2,-\delta r^2)$, as wanted.
\end{proof}

\section{Proof of  Theorem \ref{thm-Stefan-intro} and of its consequences} \label{sec:ESC}

Combining the results of the previous sections, we can now prove our main Theorem \ref{thm-Stefan-intro}.

\begin{proof}[Proof of Theorem \ref{thm-Stefan-intro}]
Since this is a local regularity result, it suffices to prove it inside $B_1\times[-1,1]$.

Combining Proposition \ref{LinMon-m}, Proposition \ref{prop-sigma<3}, Lemma \ref{lem:P6784t78}, and Proposition \ref{prop:heiht4367}, we find that ${\rm dim}_{\mathcal H}\big(\pi_x(\Sigma\setminus \Sigma^{\diamond})\big)\leq n-2$.
Therefore, by Corollary \ref{corhaohowih} and Lemma \ref{prop:GMT4}, we deduce that ${\rm dim}_{\rm par}(\Sigma\setminus \Sigma^{\diamond})\leq n-2$. 

Then, we write $\Sigma^{\diamond} = \cup_{m\ge 1} \Sigma^{\diamond}_m$, where $\Sigma^{\diamond}_m$ is defined as the set of points in $\Sigma_\diamond\cap \overline{B_{1-1/m}}\times[-1+1/m, 1-1/m]$ for which  \eqref{wiohehth} holds for some $\beta\ge 1/m$ and $C_\circ\le m$. 
We claim that each set $\Sigma^{\diamond}_m$ can be covered by a  $(n-1)-$manifold  of class $C^\infty$.

Indeed, thanks to Theorems \ref{thm:euihwbegs} and \ref{tdrsdfghkljjhgfxd}, for every $(x_\circ,t_\circ) \in \Sigma^{\diamond}_m$ and  for all $r\in (0, r_{m,k})$, we have $u = u^{(1)} + u^{(2)}$, where $u^{(i)}$ have disjoint support and look like ``half-space solutions'' inside $(x_\circ,t_\circ)+K_r$, where
\[K_r:= B_r\times \textstyle \big(-r^2, - \frac{r^{2}}{100}\big).\]
Moreover, for any $k \geq 3$, there exist two  sequences of  Ans\"atze $\anz_{k,x_\circ, t_\circ}^{(1)}$ and $\anz_{k,x_\circ, t_\circ}^{(2)}$, satisfying
 \begin{equation}\label{whjiohewoth}
\big\| u^{(i)}(x_\circ + \,\cdot\, , t_\circ + \,\cdot) - \anz_{k,x_\circ, t_\circ}^{(i)}  \big\|_{L^\infty(K_r)}\le C_{m,k} r^{k+\alpha}.
\end{equation}
This implies in particular that the sequence of coefficients of $\anz_k$ is constant in $k$: more precisely
\begin{equation}\label{whjiohewoth2a}
 \varpi_\ell \big(\widetilde \anz_{k,x_\circ, t_\circ}^{(i)}\big)  =  \varpi_\ell \big(\widetilde \anz_{k' ,x_\circ, t_\circ}^{(i)}\big) \quad \mbox{for all }\, \ell \le \min(k, k'),
 \end{equation}
where $\widetilde {\anz}_{k,x_\circ, t_\circ}^{(i)}$ and $\varpi$ are defined as in Lemma~\ref{tgbyhn}.
In addition, as discussed in the proof of Theorem \ref{tdrsdfghkljjhgfxd}, we have a uniform control on the coefficients:
\begin{equation}\label{whjiohewoth2}
 \big|\varpi_\ell \big(\widetilde \anz_{k,x_\circ, t_\circ}^{(i)}\big) \big| \le C_{m,\ell} \quad \mbox{for all } 0\le \ell\le k.
 \end{equation}
Furthermore, by non-degeneracy (see \eqref{optimalreg+nondegP}), for all  $(x_1,t_1) \in \big((x_\circ,t_\circ)+K_r\big) \cap \partial \{ u^{(i)} >0\} $ we have 
\begin{equation}\label{wngiowhoithiew}
\sup_{|z|\le \varrho} u(x_1+ z,t_1)\ge c_{m} \varrho^2>0. 
\end{equation}

Assume now with no loss of generality that  $(0,0) \in \Sigma_m^{\diamond}$ and $p_2 = \frac 1 2 x_n^2$,
and consider any other singular point $(x_\circ, t_\circ) \in \Sigma_m^{\diamond}\cap (B_{r_m}\times (-r_m, r_m))$ with $r_m\ll 1$. Then we have  
\[
p_{2, x_\circ t_\circ} = \frac 1 2 (e\cdot x)^2,\quad \mbox{where}\quad  |e-\boldsymbol e_n| \le \frac 1 {1000}.
\]
Also, we can relabel  the Ansatz $\anz_{k,x_\circ, t_\circ}^{(i)}$ (simply interchanging $(1)$ and $(2)$, when necessary) so that 
\[
\partial_n\anz_{k,x_\circ, t_\circ}^{(1)} \sim (x_n)_+   \qquad \mbox{and}  \qquad \partial_n \anz_{k,x_\circ, t_\circ}^{(2)} \sim (x_n)_-.
\]
Recalling the notation $\widetilde {\anz}_{k,x_\circ, t_\circ}^{(i)}$ introduced in \eqref{eq:Pk tilde Pk}, let $g_{k,x_\circ, t_\circ}^{(i)} = g_{k,x_\circ, t_\circ}^{(i)}(x',t) $  be the unique polynomials of degree $k-1$ which solve  the polynomial equations
 \begin{equation}
 \label{eq:gk}
\partial_n {\widetilde\anz}_{k,x_\circ, t_\circ}^{(i)} (x , t)  =0  \qquad \Leftrightarrow  \qquad   x_n = g_{k,x_\circ, t_\circ}^{(i)}(x', t) + O\big( (|x|+|t|^{1/2})^{k} \big).
 \end{equation}
Also, using \eqref{wngiowhoithiew}, the corresponding nondegeneracy for $\anz_{k,x_\circ, t_\circ}^{(i)}$,  and \eqref{whjiohewoth}, 
we have
$$
\partial\{u^{(i)}(x_\circ + \,\cdot\, , t_\circ + \,\cdot)>0\}\cap K_r \subset  \partial \big\{ \anz_{k,x_\circ, t_\circ}^{(i)} >0\big\} + B_{Cr^{(k+\alpha)/2}}(0).
$$
Hence, since $\partial\{\anz_{k,x_\circ, t_\circ}^{(i)} >0\} = \{\partial_n \anz_{k,x_\circ, t_\circ}^{(i)} =0\}$, combining these two informations we obtain
\begin{equation}\label{hewioheioht}
\partial\{u^{(i)}(x_\circ + \,\cdot\, , t_\circ + \,\cdot)>0\}\cap K_r\subset \big \{ |x_n- g_{k,x_\circ, t_\circ}^{(i)}(x', t) | \le Cr^{(k+\alpha)/2}\big\} .
\end{equation}
Now, given a polynomial $P=P(x,t)$, we denote by $\Gamma_i [ P]$ the  $i$-homogeneous part of $P$ with respect to the usual Euclidean structure in $(x,t)$ (namely, the variables $x$ and $t$ have the same homogeneity).
Note that \eqref{whjiohewoth2a} implies, in particular, 
\begin{equation}\label{whjiohewoth3a}
\Gamma_\ell [g_{k,x_\circ, t_\circ}^{(i)}] =  \Gamma_\ell [g_{k',x_\circ, t_\circ}^{(i)}] \quad \mbox{for all }\,  \ell \le \min( \lfloor k/2\rfloor,  \lfloor k'/2\rfloor),
 \end{equation}
while \eqref{whjiohewoth2} gives
\begin{equation}\label{whjiohewoth3}
 \big|\Gamma_\ell [g_{k,x_\circ, t_\circ}^{(i)}](x_\circ + \cdot, t_\circ + \cdot) \big| \le C_{m,\ell} \quad \mbox{for all } 0 \le \ell\le \lfloor k/2\rfloor.
 \end{equation}
 This suggests the following definition of the $\ell$-jet at $(x_\circ,t_\circ)$:
 \begin{equation}
 \label{eq:jet}
J^{(i)}_{\ell,x_\circ, t_\circ}(x',t):=\sum_{l=0}^\ell \Gamma_\ell [g_{k,x_\circ, t_\circ}^{(i)}] (x'-x'_\circ, t-t_\circ),
 \end{equation}
 where $k$ is an arbitrary number such that $\lfloor k/2\rfloor \geq \ell$, say $k=2(\ell+1)$.
 Note that, with this definition, it follows from \eqref{hewioheioht} that
 \begin{equation}\label{hewioheioht2}
\partial\{u^{(i)}>0\}\cap \big((x_\circ,t_\circ)+K_r)\subset \big \{ |(x-x_\circ)\cdot \boldsymbol e_n- J_{\ell,x_\circ, t_\circ}^{(i)}(x'-x_\circ', t-t_\circ) | \le Cr^{\ell+1}\big\} .
\end{equation}
 We also recall that, thanks to  Proposition \ref{cleaning.iii}, for every pair of singular points $(x_\circ,t_\circ),(x_1,t_1)\in \Sigma_m^{\diamond}\cap (B_{r_m}\times (-r_m, r_m))$ we have
\begin{equation}\label{heiotheoihwwz}
|t_1-t_\circ| \le C_m|x_1-x_\circ|^2.
\end{equation}
Thanks to this bound, we can apply \eqref{whjiohewoth}  at two points $(x_\circ,t_\circ),(x_1,t_1)\in \Sigma_m^{\diamond}\cap (B_{r_m}\times (-r_m, r_m))$ inside $K_{r}$, with $r:=|x_1-x_\circ|^{1/2}$,
and the two sets 
$(x_\circ,t_\circ)+K_{r}$ and $(x_1,t_1)+K_{r}$ intersect in a domain which contains $B_{10 r^2}(x_\circ)\times [t_\circ - r^2/2,t_\circ-r^2/4]$. Hence, for any $k \geq 3,$
$$
\big\| \anz_{k,x_1, t_1}^{(i)}(\cdot - x_1,\cdot -t_1) - \anz_{k,x_\circ, t_\circ}^{(i)}(\cdot - x_\circ,\cdot -t_\circ)  \big\|_{L^\infty(B_{10 r^2}(x_\circ)\times [t_\circ - r^2/2,t_\circ-r^2/4])}\le C_{m,k} r^{k+\alpha}.
$$
Since $ \anz$ are essentially ``half-polynomials'', it is easy to check that the bound above implies that
$$
\big\|\partial_n \anz_{k,x_1, t_1}^{(i)}(\cdot - x_1,\cdot -t_1) - \partial_n\anz_{k,x_\circ, t_\circ}^{(i)}(\cdot - x_\circ,\cdot -t_\circ)  \big\|_{L^\infty(B_{10 r^2}(x_\circ)\times [t_\circ - r^2/2,t_\circ-r^2/4])}\le C_{m,k} r^{k-2+\alpha}.
$$
Now, given $\ell$, we apply the estimate above with $k:=2(\ell+2)$.
Then, recalling the definition of $J^{(i)}$ (see \eqref{eq:gk} and \eqref{eq:jet}),
since all terms of Euclidean homogeneity at least $\ell+1$ have a size bounded by $C_{m,\ell}(r^2)^{\ell+1}$, recalling that $r^2=|x_1-x_\circ|$ we deduce that 
\[
\big\| J^{(i)}_{\ell,x_1, t_1}-J^{(i)}_{\ell,x_\circ, t_\circ}\big\|_{L^\infty(B_{10 |x_1-x_\circ|}(x_\circ')\times [t_\circ - |x_1-x_\circ|/2,t_\circ-|x_1-x_\circ|/4])} \le C_{m,\ell}|x_1-x_\circ|^{\ell+1} \qquad \forall\,\ell \geq 1.
\]
Recalling \eqref{heiotheoihwwz}, it is easy to check that this bound is exactly what is required to apply Whitney's extension theorem in order to produce two $C^\infty$ functions $\mathcal G^{(i)} : B'_{r_m}\times (-r_m, r_m) \to \R$ (recall that $B'_{r_m}\subset \R^{n-1}$)
such that, for all $(x_\circ, t_\circ)$ as above and for all $\ell$, we have
$
\mathcal G^{(i)}(x,t) = J^{(i)}_{\ell,x_\circ, t_\circ}(x,t) + O\big ( (|x'-x_\circ'| +|t-t_\circ'|)^{\ell+1} \big).
$
In particular, thanks to \eqref{hewioheioht2},
every point in $\Sigma_m^{\diamond}\cap (B_{r_m}\times (-r_m, r_m))$ belongs both to $\{x_n = G^{(1)}(x',t)\}$  and to $\{x_n = G^{(2)}(x',t)\}$. 
This shows that $\Sigma_m^{\diamond}\cap (B_{r_m}\times (-r_m, r_m))$ can be covered by a $(n-1)$-dimensional manifold of class $C^\infty$, as wanted. 

In order to conclude the proof of the theorem we need to prove that, after removing a set of parabolic dimension $n-2$, we can obtain a set $\Sigma^{\infty}$ such that $\pi_t(\Sigma^{\infty})$ has zero Hausdorff dimension.
To this aim,  we define  $\Sigma^{\infty}$ as the set of points $(x_\circ,t_\circ)\in \Sigma_m^{\diamond}\cap ( B_{r_m}\times (-r_m, r_m))$ at which the two functions
$\mathcal G^{(1)} (x', t_\circ)$ and $\mathcal G^{(2)}(x', t_\circ)$ are tangent in $x'$ at infinite order.
The first important observation is that the set of  points which do \emph{not} belong to $\Sigma^{\infty}$ has parabolic dimension at most $n-2$. 

Indeed, thanks to  Lemma  \ref{lem:84y6t3} the two $n$-dimensional (in the Euclidean sense) manifolds $\{x_n = \mathcal G^{(i)} (x',t)\}$ intersect transversally ``in time'' at every singular point $(x_\circ, t_\circ)$ (namely, their derivatives in the time variable never match). Hence, their intersection has at most Euclidean dimension $n-1$.
Also, by definition of $\Sigma^{\infty}$,  if $(x_\circ,t_\circ)\in \Sigma^{\diamond}\setminus \Sigma^\infty$ then the two functions $\mathcal G^{(1)} (x', t_\circ)$ and $\mathcal G^{(2)}(x', t_\circ)$, when expanded at $x' = x_\circ$, have different coefficients at some order. 
This reduces the bound on the dimension by $1$ , hence $\Sigma^{\diamond}\setminus \Sigma^{\infty}$ has  Euclidean dimension at most $n-2$. Thanks to \eqref{hewioheioht2} and Lemma~\ref{prop:GMT4}, this implies that  
$\dim_{\rm par}(\Sigma^{\diamond}\setminus \Sigma^{\infty})\leq n-2$, as desired.

Finally, let us show for all $k \in \N$ and  for all pair of points $(x_\circ,t_\circ),(x_1,t_1) \in \Sigma^{\infty}$ we have
\begin{equation}\label{woiheiohe}
|t_1-t_\circ| \le C_{m,k} |x_1-x_\circ|^k.
\end{equation}
Indeed, by definition of $\Sigma^\infty$ and the discussion above, there exist functions $h,A^{(1)},A^{(2)}$ such that 
\begin{equation}\label{hwhehoehw}
\mathcal G^{(i)} (x', t+t_\circ) = h(x') + A^{(i)}(x',t) + O(|x'-x_\circ'|^k).
\end{equation}
where $A^{(i)}(x',0)\equiv 0$.
In addition, thanks to  Lemma  \ref{lem:84y6t3},  
\begin{equation}\label{hwhehoehw2}
(A^{(2)} -A^{(1)}) (x',t) \ge ct  >0 \qquad \mbox{  for }|x'-x_\circ'|+t\ll 1, \quad t>0.
\end{equation}
Then, since  $\mathcal G^{(1)}(x_1', t_1)=  \mathcal G^{(2)}(x_1', t_1)$ and assuming (with no loss of generality) that $t_\circ<t_1$,
we can evaluate \eqref{hwhehoehw} at $(x_1',t_1-t_\circ)$ to obtain 
\[
A^{(1)}(x_1', t_1-t_\circ)  - A^{(2)}(x_1', t_1-t_\circ) = O(|x_1-x_\circ|^k).
\]
Combining this bound with \eqref{hwhehoehw2} we get \eqref{woiheiohe}.
As shown for instance in \cite[Proposition 7.7(a)]{FRS},
 \eqref{woiheiohe} implies that ${\rm dim}_{\mathcal H}\big(\pi_t(\Sigma^\infty)\big)=0$.
\end{proof}

Finally, as a consequence of Theorem \ref{thm-Stefan-intro}, we deduce Theorem \ref{thm-Stefan-intro-3d} and Corollary \ref{thm-Stefan-intro-2d}.

\begin{proof}[Proof of Theorem \ref{thm-Stefan-intro-3d} and Corollary \ref{thm-Stefan-intro-2d}]
Thanks to Theorem \ref{thm-Stefan-intro} we know that $\Sigma = \Sigma^\infty \cup (\Sigma\setminus\Sigma^\infty)$, where ${\rm dim}_{\mathcal H}\big(\pi_t(\Sigma^\infty)\big)=0$ and ${\rm dim}_{\rm par}(\Sigma\setminus \Sigma^\infty)\leq n-2$.
Hence, recalling the definition of parabolic dimension, this implies that ${\rm dim}_{\mathcal H}\big(\pi_t(\Sigma\setminus \Sigma^\infty)\big)\leq \frac{n-2}{2}$, and the result follows.
\end{proof}

\end{document}